\documentclass[11pt]{article}

\usepackage[margin=1.0in]{geometry}
\usepackage[utf8]{inputenc}            % allow utf-8 input
\usepackage[T1]{fontenc}               % use 8-bit T1 fonts
\usepackage[dvipsnames]{xcolor}
\usepackage[
    pdftex=true,
    colorlinks=true,
    linkcolor=RoyalPurple,   % Color for cross-references, e.g., \ref, \eqref
    citecolor=Blue,     % Color for citations, e.g., \cite
    urlcolor=Violet     % Color for URLs
]{hyperref}
\usepackage[mathscr]{euscript}
\usepackage[T1]{fontenc}

\usepackage{url}                        % simple URL typesetting
\usepackage{booktabs}                   % professional-quality tables
\usepackage{amsfonts}                   % blackboard math symbols
\usepackage{nicefrac}                   % compact symbols for 1/2, etc.
\usepackage{bbm}
\usepackage{microtype}                  % microtypography
\usepackage{algorithm}
\usepackage{algpseudocode}
\usepackage{lmodern}

\usepackage[numbers, sort]{natbib}
\usepackage{amsmath, amssymb, amsfonts}
\usepackage{amsthm, thmtools, thm-restate}
\usepackage{mathtools, nccmath, extarrows}
\usepackage{enumitem}
\usepackage{subcaption}
\usepackage{bm}
\usepackage{scalerel}
\usepackage{booktabs}

\usepackage{empheq}
\usepackage[most]{tcolorbox}
\tcbuselibrary{theorems}

\usepackage{todonotes}

\newcommand{\revised}[1]{{#1}}

\newcommand{\approptoinn}[2]{\mathrel{\vcenter{
  \offinterlineskip\halign{\hfil$##$\cr
    #1\propto\cr\noalign{\kern2pt}#1\sim\cr\noalign{\kern-2pt}}}}}

\newcommand{\appropto}{\mathpalette\approptoinn\relax}

\usepackage{stackengine}
\mathtoolsset{showonlyrefs} % Only add number equations if they are referenced
\stackMath

\newcommand\lbr{%
	\ensurestackMath{\stackengine{0pt}{\lfloor}{\lceil}{O}{c}{F}{F}{L}}}
\newcommand\rbr{%
  \ensurestackMath{\stackengine{0pt}{\rfloor}{\rceil}{O}{c}{F}{F}{L}}}

\providecommand{\keywords}[1]{\textbf{\noindent\textit{Keywords---}} #1}
\providecommand{\msc}[1]{\textbf{\noindent\textit{MSC numbers---}} #1}
\DeclareFontEncoding{LS1}{}{}
\DeclareFontSubstitution{LS1}{stix}{m}{n}

\makeatletter
\newcommand*\textmathversion{\csname textmv@\math@version\endcsname}
\newcommand*\textmv@normal{m}
\newcommand*\textmv@bold{b}
\makeatother

\newtheorem{theorem}{Theorem}[section]
\newtheorem{proposition}[theorem]{Proposition}
\newtheorem{lemma}[theorem]{Lemma}

\newtheorem{claim}[theorem]{Claim}

\theoremstyle{definition}
\newtheorem{definition}[theorem]{Definition}
\newtheorem{assumption}{Assumption}

\newtcbtheorem{defn}{Definition}%
{	colframe=blue!30!gray, theorem name, fonttitle=\bfseries,
	colback=blue!5, top=1mm,bottom=1mm, boxrule=0.5pt}{Example}

\newtcbtheorem[use counter=theorem, number within=section]{THM}{Theorem}%
{	colframe=blue!30!gray, fonttitle=\bfseries,
	colback=blue!5,  fontupper=\itshape, top=1mm,bottom=1mm, boxrule=0.5pt}{Theorem}
\usepackage{varwidth}
\newtcbox{\mybox}{colback=blue!5,
	colframe=blue!30!black, center, enhanced, varwidth upper}
\newtcbox{\mymath}[1][]{%
	nobeforeafter, math upper, tcbox raise base,
	enhanced, colframe=blue!30!black,
	colback=blue!5, boxrule=0.5pt, top=1mm,bottom=1mm,
	#1}
\usepackage{stackengine}

\usepackage{bm}

\newcommand{\bA}{{A}}

\newcommand{\bP}{{P}}

\newcommand{\bR}{{R}}
\newcommand{\bS}{\bm{S}}

\newcommand{\bX}{X}

\newcommand{\bx}{{x}}
\newcommand{\bb}{{b}}

\newcommand{\cB}{\mathcal{B}}
\newcommand{\cC}{\mathcal{C}}

\newcommand{\cE}{\mathcal{E}}
\newcommand{\cF}{\mathcal{F}}
\newcommand{\cG}{\mathcal{G}}
\newcommand{\cH}{\mathcal{H}}
\newcommand{\cI}{\mathcal{I}}
\newcommand{\cJ}{\mathcal{J}}
\newcommand{\cK}{\mathcal{K}}

\newcommand{\cM}{\mathcal{M}}

\newcommand{\cP}{\mathcal{P}}
\newcommand{\cQ}{\mathcal{Q}}

\newcommand{\cT}{{\mathcal{T}}}
\newcommand{\cU}{\mathcal{U}}
\newcommand{\cV}{\mathcal{V}}

\newcommand{\cX}{\mathcal{X}}
\newcommand{\cY}{\mathcal{Y}}
\newcommand{\cZ}{\mathcal{Z}}

\newcommand{\bbeta}{{\beta}}

\newcommand{\blambda}{{\lambda}}

\newcommand{\bSigma}{{\Sigma}}

\newcommand{\R}{\mathbf{R}}

\newcommand{\RR}{{\mathbf R}}

\newcommand{\sd}{\mathsf{d}}

\newcommand{\sD}{\mathsf{D}}
\newcommand{\sP}{\mathsf{P}}

\newcommand{\Span}{\operatorname{span}}

\newcommand{\rank}{\operatorname{rank}}
\newcommand{\Tr}{\operatorname{tr}}
\newcommand{\tr}{\operatorname{tr}}

\newcommand{\Diag}{\operatorname{Diag}}

\newcommand{\Vect}{\operatorname{vec}}
\newcommand{\dist}{\operatorname{dist}}
\newcommand{\range}{\operatorname{range}}

\newcommand{\EE}{\operatorname{\mathbb E}}
\newcommand{\Var}{\operatorname{Var}}
\newcommand{\Cov}{\operatorname{Cov}}

\newcommand{\PP}{\operatorname{\mathbb P}}

\newcommand{\Supp}{\operatorname{supp}}
\newcommand{\supp}{\operatorname{supp}}

\newcommand{\argmin}{\operatornamewithlimits{argmin}}

\newcommand{\conv}{\operatorname{conv}}

\newcommand{\prox}{\operatorname{prox}}

\DeclarePairedDelimiter{\dotp}{\langle}{\rangle}
\def\shortdisplay{\setlength{\abovedisplayskip}{5pt}%
	\setlength{\belowdisplayskip}{5pt}%
	\setlength{\abovedisplayshortskip}{2pt}%
	\setlength{\belowdisplayshortskip}{2pt}}

\let\oldselectfont\selectfont
\def\selectfont{\oldselectfont\shortdisplay}

\newcommand{\hlS}{\mathsf{\Sigma}} % Hajek Le Cam covariance
\newcommand{\Ms}{M_{\star}} % Solution of matrix regression problems
\newcommand{\RSE}{\operatorname{RSE}}
\newcommand{\lb}{c_{\textup{\texttt{lb}}}}
\newcommand{\ub}{c_{\textup{\texttt{ub}}}}
\newcommand{\be}{{e}}
\newcommand{\bv}{{v}}
\newcommand{\x}{x}
\newcommand{\Q}{\cQ}
\newcommand{\betas}{{\beta}_\star}

\newcommand{\z}{{z}}

\begin{document}

%\tolerance 1414
%\hbadness 1414
\emergencystretch 1.5em
%\hfuzz 0.3pt
%\widowpenalty=10000
%\vfuzz \hfuzz
%\raggedbottom

\title{The radius of statistical efficiency}
\author{
Joshua Cutler\thanks{Department of Mathematics, U. Washington,
	Seattle, WA 98195; \texttt{jocutler@uw.edu}.} \and Mateo D\'\i az\thanks{Department of Applied Mathematics and Statistics, Johns Hopkins University, Baltimore, MD 21218, USA; \texttt{https://mateodd25.github.io}. Research of D\'iaz was supported by NSF awards CCF-2442615 and DMS-2502377.} \and Dmitriy Drusvyatskiy\thanks{Hal{\i}c{\i}o\u{g}lu Data Science Institute (HDSI),
University of California San Diego,
La Jolla, CA 92093; \texttt{https://sites.google.com/view/dmitriy-drusvyatskiy}. Research of Drusvyatskiy was supported by the NSF DMS-2306322, NSF DMS-2023166, and AFOSR FA9550-24-1-0092 awards.}
}

\date{}

\maketitle

\begin{abstract}
	Classical results in asymptotic statistics show that the Fisher information matrix controls the difficulty of estimating a statistical model from observed data.
	In this work, we introduce a companion measure of robustness of an estimation problem: the radius of statistical efficiency (RSE) is the size of the smallest perturbation to the problem data that renders the Fisher information matrix singular. We compute RSE up to numerical constants for a variety of testbed problems, including principal component analysis, generalized linear models, phase retrieval, bilinear sensing, and matrix completion.
	Interestingly, we observe a precise reciprocal relationship between RSE and the intrinsic complexity/sensitivity of the problem instance,  paralleling the classical Eckart–Young theorem in numerical analysis. To establish our results, we develop theory for spectral functions of measures that extends well-known results from matrix analysis and eigenvalue optimization---a contribution that may be of interest beyond our immediate findings.
\end{abstract}
\keywords{radius theorem, conditioning, statistical complexity, Eckart--Young theorem,  stability, robustness, Fisher information matrix, local minimax lower bounds.}\\

\noindent \msc{90C15, 49K40, 62F12, 90C31.}
%=================================================================
\section{Introduction}\label{section:intro}

A central theme in computational mathematics is that the numerical difficulty of solving a given problem is closely linked to both $(i)$ the sensitivity of its solution to perturbations and $(ii)$ the shortest distance  of the problem to an ill-posed instance.  As a rudimentary example, consider solving an $m\times d$ linear system $\bA\bx=\bb$. The celebrated Eckart–Young theorem asserts the equality
%\begin{equation}\label{eqn:ey}
%	\dist_2(\bA,{\cE})=\sigma_{\min}(\bA),
%\end{equation}

\begin{equation}\label{eqn:ey}
	\underbrace{\min_{B\in \R^{m\times d}}\left\{\|A-B\|_F\mid B~\rm{is ~singular}\right\}}_{\clap{\scriptsize{\rm distance to ill-posedness}}} \enspace = \qquad \underbrace{\sigma_{\min}(\bA)}_{\clap{\scriptsize{\rm difficulty/sensitivity}}}.
\end{equation}
%where $\dist_2(\bA,{\cE})$ denotes the minimal Frobenius distance of $\bA$ to $\cE$ and 
%$\sigma_{\min}(\bA)$ is the minimal singular value of $A$. 
Although the proof is elementary, the conclusion is intriguing since it equates two conceptually distinct quantities. Namely, the reciprocal of the minimal singular value $1/\sigma_{\min}(\bA)$ is classically known to control both the numerical difficulty of solving the linear system $\bA\bx=\bb$ and the Lipschitz stability of the solution to perturbations in $\bb$. In contrast, the left side of the equation \eqref{eqn:ey}  is geometric; it measures the smallest perturbation to the data that renders the problem ill-posed.

The {\em exact} equality in \eqref{eqn:ey} is somewhat misleading because it is specific to linear systems. We would expect that for more sophisticated problems, the difficulty/sensitivity of the problem  should be inversely proportional to the distance to ill-posedness. This is indeed the case for a wide class of problems in numerical analysis \cite{demmel1997applied,burgisser2013condition} and optimization  \cite{renegar1995incorporating, renegar1994some, freund1999some, pena2000understanding,dontchev2003radius}, including computing eigenvalues and eigenvectors, finding zeros of polynomials, pole assignment in control systems, conic optimization, nonlinear programming, and variational inequalities. Despite this impressive body of work, this line of research is largely unexplored in statistical contexts. Therefore, here we ask:

\mybox{\it \centering
  Is there a succinct relationship between complexity, sensitivity, and distance to ill-posedness for problems in statistical inference and learning?}

\revised{As stated, this question is somewhat ill-posed. Unlike in the linear system case where there is a clear parameterization of the problem, statistical inference problems take many different forms with varying structure. Instead of pursuing overgeneralized answers, we demonstrate that for a specific class of problems resembling maximum likelihood estimation, a precise relationship can indeed be established.}
%As stated this question is somewhat ill-posed, unlike in the linear system case statistical inference problems might take many different forms. Instead of focusing on general problems, we will see that in a certain precise sense the answer is indeed yes for a wide class of problems akin to MLE problems. %, including principal component analysis (PCA), generalized linear models, phase retrieval, bilinear sensing, and matrix completion.
The starting point for our development is that the statistical difficulty of estimation is tightly controlled by (quantities akin to) the Fisher information matrix for maximum likelihood estimation. This connection is made precise for example by the Cram\'{e}r-Rao lower bound \cite{cramer1999mathematical, rao1992information} and the local asymptotic minimax theory of H\'{a}jek and Le Cam \cite{le2000asymptotics,van2000asymptotic,duchi2021asymptotic}.
From an optimization viewpoint, the minimal eigenvalue of the Fisher information matrix is closely related to the quadratic growth constant of the objective function modeling the learning problem at hand.
In particular, (near-)singularity of this matrix signifies that the problem is ill-conditioned.
Inspired by this observation,  we introduce a new measure of robustness associated to an estimation task: the {\em radius of statistical efficiency (RSE)} is the size of the smallest perturbation to the problem data that renders the Fisher information matrix singular. Thus large RSE signifies existence of a large neighborhood around the problem instance comprised only of well-posed problems. 

 We compute RSE for a variety of testbed problems, including principal component analysis (PCA), generalized linear models, phase retrieval, bilinear sensing, and rank-one matrix completion. In all cases, the RSE exhibits a precise reciprocal relationship with the statistical difficulty of solving the target problem, thereby directly paralleling the Eckart–Young theorem and its numerous extensions in numerical analysis and optimization. Moreover, we provide gradient-based conditions for general estimation problems which ensure validity of such a reciprocal relationship. 
 Slope-based criteria for error bounds, due to Ioffe \cite{ioffe2000metric} and Az\'{e}-Corvellec \cite{aze2004characterizations}, play a key role in this development.

Before delving into the technical details, we illustrate the main thread of our work with two examples---linear regression and PCA---where the conclusions are appealingly simple to state.

\paragraph{Linear regression.}
The problem of linear regression is to recover a vector $\betas \in \RR^{d}$ from noisy linear measurements
$$y = \dotp{\bx, \betas} + \varepsilon,$$ where $\bx$ is a random vector in $\R^d$ with probability distribution $\sD_x$ and $\varepsilon$ is zero-mean noise that is independent of $\bx$. The standard approach to this task is to minimize the mean squared error
\begin{equation}
	\label{eq:lr-measurements}
	\min_{\bbeta \in \RR^{d}}\, \mathop{\EE}_{(x,y)\sim\sD} \tfrac{1}{2} \big( \dotp{\bx, \bbeta} - y \big)^{2}.
\end{equation}
Classical results show that the asymptotic performance of estimators for this problem is tightly controlled by $\bSigma^{-1}$, where $\bSigma:=\EE_{\sD_x} [\bx\bx^\top]$ is the second-moment matrix of the population. The closer the matrix $\Sigma$ is to being singular, the more challenging the problem \eqref{eq:lr-measurements} is to solve, requiring a higher number of samples. Seeking to estimate a neighborhood of well-posed problems around $\sD$, the RSE is defined to be the minimal Wasserstein-2 distance from $\sD_x$ to distributions with a singular second-moment matrix. We will see that for linear regression \eqref{eq:lr-measurements}, RSE is simple to compute:

\begin{empheq}[box=\mymath]{equation}
	\RSE(\sD) = \sqrt{\lambda_{\min}(\bSigma)}.\label{eqn:llin_square_bam}
\end{empheq}
That is, the simplest measure of ill-conditioning of the target problem $1/\lambda_{\min}(\bSigma)$ has a geometric interpretation as the reciprocal of the squared distance to the nearest ill-posed problem. The representation of RSE in \eqref{eqn:llin_square_bam} is not specific to linear regression and holds much more generally for (quasi) maximum likelihood estimation \cite{McCullagh1983} with strongly convex cumulant functions.

\paragraph{Principal component analysis (PCA).} Principal component analysis (PCA) seeks to find a $q$-dimensional subspace $\mathcal{V}\subset\R^d$ that captures most of the variance of a centered random vector $x$ drawn from a probability distribution $\sD$. Analytically, this amounts to solving the problem
\begin{equation} \label{eqn:pca_brah0}
	\max_{\bR\in {\rm Gr}(q, d)}\, \mathop{\EE}_{\bx\sim \sD}\| \bR \bx\|^2_2,
\end{equation}
where the Grassmannian manifold ${\rm Gr}(q, d)$ consists of all orthogonal projection matrices $\bR\in \R^{d\times d}$ onto $q$-dimensional subspaces of $\R^d$. The column space of an optimal matrix $R$ is called the {\em top $q$ principal subspace}. Intuitively, the hardness of the problem is governed by the gap between the $q$'th and $(q+1)$'st eigenvalues of $\Sigma:={\EE}_{\sD}{[x x^\top]}$---the smaller the gap, the higher the number of samples required for estimation, and this can indeed be made rigorous. We define the RSE for PCA to be the minimal  Wasserstein-2 distance from $\sD$ to a centered distribution with covariance having equal $q$'th and $(q+1)$'st eigenvalues. In this case, RSE admits the simple form
\begin{empheq}[box=\mymath]{equation}
	\RSE(\sD) = \tfrac{1}{\sqrt{2}}\bigg(\sqrt{\lambda_{q}(\bSigma)}-\sqrt{\lambda_{q+1}(\bSigma)}\bigg).\label{eqn:rse_pca}
\end{empheq}
In particular, the expression \eqref{eqn:rse_pca} endows the gap $\sqrt{\lambda_{q}(\bSigma)}-\sqrt{\lambda_{q+1}(\bSigma)}$---the reciprocal of the problem's complexity---with a geometric meaning as the distance to a nearest ill-conditioned problem instance.
\smallskip

The two examples of linear regression and PCA can be understood within the broader context of stochastic optimization:
\begin{equation}
	\label{eq:main-problem}
	\min_{\bbeta \in \cM}\, f(\bbeta)\quad\text{where}\quad f(\bbeta)=\mathop{\EE}_{\z \sim \sD} \ell(\bbeta; \z) \tag*{$\mathtt{SP}(\sD)$}.
	\noeqref{eq:main-problem}
\end{equation}
Here $\z$ is data drawn from a distribution $\sD$, the function $\ell(\bbeta;\cdot)$ is a loss parameterized by $\bbeta$, and $\cM\subset\R^d$ is a $C^{2}$ embedded submanifold comprised of allowable model parameters. 
For example, \ref{eq:main-problem} may correspond to least-squares regression or maximum likelihood estimation. In both cases, the asymptotic efficiency of estimators for finding the minimizer $\betas$ of \ref{eq:main-problem} is tightly controlled by the following matrix akin to Fisher information:
\begin{equation}
	\label{eq:fisher-information}
	\cI(\betas,\sD)=\sP_{\mathcal{T}}\nabla_{\bbeta\bbeta}^2 \mathcal{L}(\betas, \blambda_{\star}) \sP_{\mathcal{T}}.
\end{equation}
Here $\sP_{\mathcal{T}}$ is the projection onto the tangent space $\cT$ of $\cM$ at $\betas$ and $\mathcal{L}$ is the Lagrangian function for \ref{eq:main-problem} with optimal multipliers $\blambda_{\star}$. For simplicity, we will abuse notation and call $\cI(\betas,\sD)$ the Fisher information matrix. The matrix $\cI(\betas,\sD)$ plays a central role in estimation, as highlighted by the lower bounds of Cram\'{e}r-Rao and H\'{a}jek-Le Cam \cite{le2000asymptotics,van2000asymptotic}, as well as their recent extensions to stochastic optimization of Duchi-Ruan \cite{duchi2021asymptotic}.\footnote{\revised{In general, these lower bounds consider a slightly more complicated matrix (see~\eqref{eq:HLCCovariance} in Section~\ref{sec:second-moment}), but for the examples considered in this paper it reduces to (a scalar multiple of) the pseudoinverse of the matrix $\cI(\betas,\sD)$.}} Moreover, the minimal nonzero eigenvalue of $\cI(\betas,\sD)$ controls both the coefficient of quadratic growth of the objective function in \ref{eq:main-problem} and the Lipschitz stability of the solution under linear perturbations.
In particular, the problem \ref{eq:main-problem} becomes ill-conditioned when the minimal eigenvalue of $\cI(\betas,\sD)\vert_\cT$ is small.

In summary, the matrix $\cI(\betas,\sD)$ tightly controls the difficulty of solving \ref{eq:main-problem}. Consequently,  it is appealing to consider as a measure of robustness of \ref{eq:main-problem} the size of the smallest perturbation to the data $\sD$, say in the Wasserstein-2 distance $W_2(\cdot,\cdot)$, that renders the Fisher information matrix singular. This is the viewpoint we explore in the current work, and we call this quantity the {\em radius of statistical efficiency (RSE)}. 
We choose to use the Wasserstein-2 distance in the definition of RSE, as opposed to other metrics on measures, because it leads to concise and easily interpretable estimates in examples. \revised{Indeed, the core machinery underlying our arguments can be summarized via the following observation. For most problems we study, the set of distributions with singular Fisher information matrix takes the form $\cP_{\cC} = \{\mu \mid \Sigma_{\mu} \in \cC\}$, where $ \Sigma_{\mu}$ denotes the second-moment matrix of the distribution $\mu$ and $\cC$ is some given set. For each example, we establish a reduction akin to
$$W_{2}(\sD, \cP_{\cC}) = W_{2}(\Sigma_{\sD}, \cC),$$
where $W_{2}$ on the right-hand side denotes the Bures-Wasserstein distance between positive semidefinite matrices. Thus, computing the Wasserstein-2 distance reduces to an easier finite-dimensional problem; we defer a detailed discussion to Appendix~\ref{sec:wasserstein_geometry_review}.}
In the rest of the paper, we leverage this strategy to study basic properties of RSE and compute it up to numerical constants for a variety of testbed problems: generalized linear models, PCA, phase retrieval, blind deconvolution, and matrix completion. In all cases, the RSE translates intuitive measures of ``well-posedness'' into quantified neighborhoods of stable problems. Moreover, in all cases we show a reciprocal relationship between the minimal eigenvalue of $\cI(\betas,\sD)$ and RSE, thereby paralleling the Eckart–Young theorem in numerical analysis and optimization.

\paragraph{Outline} The remainder of this section covers related work and basic notation we use. Section~\ref{sec:second-moment} formally describes the radius of statistical efficiency and establishes a few general-purpose results relating RSE to the minimal eigenvalue of the Fisher information matrix. The subsequent sections characterize RSE for several problems: PCA (Section~\ref{section:pca}), generalized linear models and (quasi) maximum likelihood estimation (Section~\ref{sec:glm}), rank-one matrix regression (Section~\ref{sec:matrix-sensing}). Section~\ref{sec:conclusions} closes the paper with conclusions. All proofs appear in the appendix in order to streamline the reading.

\subsection{Related work} Our work is closely related to a number of topics in statistics and computational mathematics.

\paragraph{Local minimax lower bounds in estimation.} There is a rich literature on minimax lower bounds
in statistical estimation problems; we refer the reader to \cite[Chapter 15]{wainwright2019high} for a detailed treatment.
Typical results of this type lower-bound the performance of any statistical procedure on a worst-case
instance for that procedure. Minimax lower bounds can be quite loose as they do not consider
the complexity of the particular problem that one is trying to solve but rather that of an entire
problem class to which it belongs. More precise local minimax lower bounds, as developed by H\'{a}jek
and Le Cam \cite{le2000asymptotics,van2000asymptotic}, provide much finer problem-specific guarantees.
Simply put, a single object akin to the Fisher information matrix controls both the difficulty of estimation from finitely many samples and the stability of the model parameters to perturbation of the density.
Extensions of this theory to stochastic nonlinear programming were developed by Duchi and Ruan \cite{duchi2021asymptotic} and extended to decision-dependent problems in \cite{cutler2022stochastic} and
to a wider class of (partly smooth) problems in \cite{davis2023asymptotic}. In particular, it is known that popular algorithms such as sample average approximation \cite{van2000asymptotic} and stochastic gradient descent with iterate averaging \cite{polyak92,davis2023asymptotic} match asymptotic local minimax lower bounds, and are therefore asymptotically optimal. Weaker ad hoc results, based on the Cram\'er-Rao lower bound, have been established for a handful of problems \cite{niazadeh2011achievability, liu2001cramer, smith2005covariance, balan2016reconstruction, bandeira2014saving}. \revised{Our results are different in nature---we do not derive local minimax lower bounds. Instead, we establish a relationship between the quantities appearing in these lower bounds and the distance to a set of ill-conditioned instances.}

\paragraph{Radius theorems.}
Classical numerical analysis literature emphasizes the close interplay between efficiency of numerical algorithms and their sensitivity to perturbation. Namely, problems with solutions that change rapidly due to small perturbation are typically difficult to solve. Examples of this phenomenon abound in computational mathematics; e.g., eigenvalue problems and polynomial equations \cite{demmel1997applied,burgisser2013condition} and optimization \cite{dontchev2003radius,dontchev2009implicit,renegar1995linear}.
Motivated by this observation, Demmel in \cite{demmel1987condition} introduced a new robustness measure, called the radius of regularity, which measures the size of a neighborhood around a problem instance within which all other problem instances are stable. A larger neighborhood thereby signifies a more robust problem instance. Estimates on the radius of regularity have now been computed for a wealth of computational problems; e.g., solving polynomial systems \cite{burgisser2011problem, burgisser2013condition}, linear and conic programming \cite{renegar1995incorporating, renegar1994some, freund1999some, pena2000understanding}, and nonlinear optimization \cite{dontchev2003radius}.
The radius of statistical efficiency, introduced here, serves as a direct analogue for statistical estimation.

\paragraph{Conditioning and radius theorems in recovery problems.}
Several condition numbers---controlling the convergence of first-order methods---are closely related to notions of strong identifiability, e.g., the restricted isometry property (RIP), in the context of statistical recovery problems \cite{blumensath2009iterative, chen2015fast, ma2018implicit, charisopoulos2021low, charisopoulos2021composite, diaz2019nonsmooth}. A few works \cite{roulet2020computational, zhang2023new, bastounis2023can} have established connections between these notions of strong identifiability and
a suitably-defined  radius to ill-posed instances. In particular, \cite{roulet2020computational} established a connection between Renegar's conic distance to infeasibility \cite{renegar1995incorporating} and the nullspace property \cite{cohen2009compressed} in compressed sensing. In a similar vein, \cite{zhang2023new} linked the $\ell_{1}$ distance to ill-posed problems and the RIP for generalized rank-one matrix completion. Finally, \cite{bastounis2023can} defined a condition number for the LASSO variable selection problem  via the reciprocal of the distance to ill-posedness, designed an algorithm whose complexity depends solely on this condition number, and proved an impossibility result for instances with infinite condition number. The radius of statistical efficiency, defined in this work, is distinct from and complementary to these metrics.

\paragraph{Error bounds.}
The basic question we explore is the relationship between the minimal eigenvalue of the Fisher information matrix (complexity) and the distance to the set where this eigenvalue is zero (RSE). The theory of error bounds exactly addresses questions of this type: namely, when does a function's value polynomially bound the distance to its set of minimizers? See, for example, the authoritative monographs on the subject \cite[Chapter 8]{cui2021modern} and \cite[Chapter 3]{ioffe2017variational}.
Indeed, Demmel's original work \cite{demmel1987condition} makes heavy use of this interpretation. We explore this path here as well when developing infinitesimal characterizations of RSE in Theorem~\ref{thm:infinitesimal}. The added complication in our development is that we must deal with functions defined over a metric space rather than a Euclidean space, and therefore the techniques we use rely on variational principles (\'{a} la Ekeland) and computations of the slope.

%%=================================================================
\subsection{Notation}\label{sec:pre}
\paragraph{Linear algebra.} Throughout, we let $\R^d$ denote the standard $d$-dimensional Euclidean space equipped with the dot product $\langle x, y \rangle = x^{\top}y$ and the induced norm $\|x\|_2=\sqrt{\langle x,x\rangle}$ (the $\ell_2$ norm). The unit sphere in $\R^d$ will be denoted by $\mathbb{S}^{d-1}$, while the set of nonnegative vectors will be written as $\R^d_+$. The symbol $\R^{m\times n}$ will denote the Euclidean space of real $m \times n$ matrices, equipped with the trace inner product $\langle X,Y\rangle=\tr(X^\top Y)$ and the induced norm $\|X\|_F = \sqrt{\langle X, X\rangle}$ (the Frobenius norm). The operator norm will be written as $\|\cdot\|_{\rm op}$. The symbol $\otimes$ denotes the Kronecker product. The singular values of a matrix $A\in \R^{m\times n}$ will be arranged in nonincreasing order:
$$\sigma_1(A)\geq \sigma_2(A)\geq \cdots\geq \sigma_{m\wedge n}(A).$$
The space of real symmetric $d\times d$ matrices will be denoted by $\bS^d$ and is equipped with the trace product as well. The cone of real positive semidefinite (PSD) $d\times d$ matrices will be written as $\bS^d_+$.   The eigenvalues of a matrix $A\in \bS^{d}$ will be arranged in nonincreasing order:
$$\lambda_1(A)\geq \lambda_2(A)\geq \cdots\geq \lambda_{d}(A).$$
For any subspace $\cK\subset\R^d$, the symbol $\sP_{\cK}\colon \RR^{d}\rightarrow \cK$ will denote the orthogonal projection onto $\cK$. The compression of a linear operator $A$ on $\R^d$ to $\cK$, denoted  $A\vert_{\cK}\colon\cK\rightarrow\cK$, is the restriction of $\sP_{\cK}A$ to $\cK$.

\paragraph{Probability theory.} We will require some background on the Wasserstein geometry on the space of probability measures on a Euclidean space $(\mathbf{E},\|\cdot\|)$ equipped with its Borel $\sigma$-algebra. In order to streamline the reading, we record here only the most essential notation that we will need. A more detailed review of Wasserstein geometry appears in Section~\ref{sec:wasserstein_geometry_review}. To this end, we let $\cP_p(\mathbf{E})$ be the space of measures $\mu$ on $\mathbf{E}$ with finite $p$'th moment ${\EE}_{\mu}\|x\|^p<\infty$. The subset comprised of measures $\mu\in\cP_p(\mathbf{E})$ that are centered, meaning $\EE_{\mu}[x]=0$, will be written as $\cP_p^{\circ}(\mathbf{E})$. When the space $\mathbf{E}$ is clear from context, we use the shorthand $\cP_p$ and $\cP_p^{\circ}$. A convenient metric on $\cP_p$ is furnished by the Wasserstein-$p$ distance $$W_p(\mu,\nu) =\min_{\pi \in \Pi(\mu,\nu)}\bigg(\mathop{\EE}_{(x,y)\sim \pi}\|x-y\|^p\bigg)^{1/p}$$ (see Section~\ref{sec:wasserstein_geometry_review} for details). The distance function to a set of measures $\cQ\subset\cP_p$ is defined by $W_p(\mu,\cQ)=\inf_{\nu\in \cQ}W_p(\mu,\nu)$. The symbol $\Sigma_{\mu}=\EE_{\mu}[xx^\top]$ will denote the second-moment matrix of any measure $\mu\in \cP_2(\R^d)$. In the rest of the paper, we will use the symbol $\sD$ to denote a distinguished measure associated with the problem of interest, while we use $\mu$ as a placeholder for arbitrary measures.

%=================================================================
\section{The distance to ill-conditioned problems}\label{sec:second-moment}

In this section, we formally define the radius of statistical efficiency (RSE) and develop some techniques for computing it. Throughout, we will focus on the stochastic optimization problem
\begin{equation}\label{eq:main-problem2}
	\min_{\beta \in \cM}\, f(\beta)\quad\text{where}\quad f(\beta)=\mathop{\EE}_{z \sim \sD} \ell(\beta; z).\tag*{$\rm{SP}(\sD)$} \noeqref{eq:main-problem2}
\end{equation}
Here the set $\cM$ is a $C^{2}$ embedded submanifold of $\R^d$ and $z$ is drawn from a distribution $\sD\in \cP_{2}(\cZ)$, where  $\cZ$ is a Euclidean space equipped with its Borel $\sigma$-algebra. \revised{We assume that the function $f$ is $C^2$-smooth on $\cM$, i.e., at each point $\beta\in\cM$, there exists a $C^2$-smooth local extension of $f$ (which we shall also denote by $f$) defined on an open neighborhood of $\beta$ in $\R^d$.   }
%We assume that the function $\ell(\beta;z)$ is twice differentiable in $\beta$ for every $z$ and that $f$ is $C^2$-smooth. We also make the blanket assumption that the gradient and Hessian of $\ell(\cdot ; z)$ are $\sD$-integrable.

The difficulty of solving the problem \ref{eq:main-problem2} from finitely many samples $z_1,\ldots, z_n\stackrel{\text{iid}}{\sim} \sD$ is tightly controlled by a matrix akin to Fisher information. This object, which we now describe, plays a central role in our work. Let $\betas$ be a minimizer of \ref{eq:main-problem2} and  define the solution map
\begin{equation*}
	\sigma(v)=\argmin_{\beta\in \cM\cap B_{\varepsilon}(\betas)}\,f(\beta)-\langle v,\beta\rangle
\end{equation*}
for some $\varepsilon>0$. Thus, $\betas\in\sigma(0)$ and the set $\sigma(v)$ is comprised of all local minimizers ($\varepsilon$-close to $\betas$) of the problem obtained from \ref{eq:main-problem2} by subtracting a linear/tilt perturbation $\langle v,\beta\rangle$. Clearly, a desirable property is for $\sigma$ to be single-valued and smooth. With this in mind, we introduce the following notion due to Poliquin and Rockafellar \cite{poliquin1998tilt}. 

\begin{defn}{(Tilt-stable minimizer)}{}
The point $\betas$ is a {\em tilt-stable minimizer} of \ref{eq:main-problem2} if for some $\varepsilon>0$, the map $\sigma(\cdot)$ satisfies $\sigma(0) = \betas$ and is single-valued and $C^1$-smooth on some neighborhood of the origin. Then the {\em regularity modulus} of the problem is defined to be
$${\rm REG}(\sD)=\|\nabla \sigma(0)\|_{\rm op}.$$ 
If  $\betas$ is not tilt-stable, we call $\betas$ {\em unstable} and set ${\rm REG}(\sD)=+\infty$.
\end{defn}

In particular, we will regard ${\rm REG}(\sD)$ as a measure of difficulty of solving the problem \ref{eq:main-problem2}.
If $\cM$ is the whole space $\R^d$, then $\betas$ is a tilt-stable minimizer if and only if the Hessian  $\nabla^2 f(\betas)$ is nonsingular, in which case equality $\nabla \sigma(0)=[\nabla^2 f(\betas)]^{-1}$ holds \cite[Proposition 1.2]{poliquin1998tilt}. In particular, when \ref{eq:main-problem2} corresponds to maximum likelihood estimation, the matrix $\nabla \sigma(0)$ reduces to the inverse of the Fisher information.
More generally, tilt-stability can be characterized either in terms of definiteness of the covariant Hessian or the Hessian of the Lagrangian on the tangent space of $\cM$. Since we will use both of these viewpoints, we review them now; \revised{for further details regarding the following Lagrangian and intrinsic characterizations, we refer the reader to \cite{Lewis2013, poliquin1998tilt, davis2023asymptotic}.} The reader may safely skip this discussion during the first reading since it will not be used until the appendix.

\paragraph{Lagrangian characterization.}
Let $G=0$ be the local defining equations for $\cM$ around $\betas$. That is, $G\colon \R^d\to \R^m$ is a $C^2$-smooth map with surjective Jacobian $\nabla G(\betas)$ and the two sets $\cM$ and $\{\beta\in\R^d \mid G(\beta)=0\}$ coincide near $\betas$. Then $\mathcal{T}:=\ker(\nabla G(\betas))$ is the tangent space of $\cM$ at $\betas$. Define the Lagrangian function $$\mathcal{L}(\beta,\lambda):=f(\beta)+\langle \lambda, G(\beta)\rangle.$$
First order optimality conditions at $\betas$ ensure that there exists a unique vector $\lambda_\star\in \R^m$ satisfying $\nabla_{\beta} \mathcal{L}(\betas,\lambda_\star)=0$. Define the matrix
\begin{equation}\label{eqn:fishinf}
\cI(\betas,\sD):= \sP_{\mathcal{T}}\nabla^2_{\beta\beta}\mathcal{L}(\betas,\lambda_\star)\sP_{\mathcal{T}},
\end{equation}
\revised{where we let $\sP_{\mathcal{T}}$ denote the $d \times d$ matrix representation of the orthogonal projection of $\R^d$ onto the tangent space $\mathcal{T}$ relative to the standard basis.} 
Since $\betas$ is a local minimizer of \ref{eq:main-problem2}, it follows that $\cI(\betas,\sD)$ is positive semidefinite. Conversely:  
\begin{quote}
\centering $\cI(\betas,\sD)$ is positive definite on  $\mathcal{T}$ if and only if $\betas$ is a tilt-stable minimizer.
\end{quote}
\revised{An application of the implicit function theorem, in tandem with classical results on tilt-stability \cite{Lewis2013}, yields
$\nabla \sigma(0)=\cI(\betas,\sD)^{\dagger}$ where $\dagger$ denotes the Moore–Penrose pseudoinverse; see, for instance, \cite[Theorem 2.7]{davis2023asymptotic}}. In particular, one may regard $\nabla \sigma(0)$ as akin to the inverse of the Fisher information matrix for MLE. Note that when $\betas$ is tilt-stable, the reciprocal of the regularity modulus ${\rm REG}(\sD)^{-1}$ coincides with the minimal nonzero eigenvalue of $\cI(\betas,\sD)$.

\paragraph{Intrinsic characterization.} The local defining equations for the manifold $\cM$ are often either unknown or difficult to work with. In this case, tilt-stability can be characterized through second-order expansions along curves. Namely, for any tangent vector $u\in \cT$, there exists a $C^2$-smooth curve $\gamma_u\colon (-\epsilon,\epsilon)\to \cM$ for some $\epsilon>0$ satisfying $\gamma_u(0)=\betas$ and $\dot{\gamma_u}(0)=u$. Since $\beta_\star$ is a critical point of $f$ on $\cM$, the covariant Hessian of $f$ at $\betas$ is the unique symmetric bilinear form $\nabla^2_{\cM} f(\betas)\colon \cT\times \cT\to \R$ satisfying 
$$\nabla^2_{\cM} f(\betas)[u,u]=(f\circ \gamma_u)''(0)\qquad \forall u\in \mathcal{T}.$$
It is classically known that the following equality holds:
$$\nabla^2_{\cM} f(\betas)[u,u]=u^\top \cdot\cI(\betas,\sD)\cdot u\qquad \forall u\in \mathcal{T},$$
where the matrix $\cI(\betas,\sD)$ is defined in \eqref{eqn:fishinf}.
\revised{Thus, viewing $\nabla^2_{\cM} f(\betas)$ as a linear operator on $\cT$ (using the Riemannian metric induced by the inclusion $\mathcal{M}\hookrightarrow\R^d$), we may identify $\nabla^2_{\cM} f(\betas)$ with the compression of $\cI(\betas,\sD)$ to $\cT$.}  Consequently, $\nabla^2_{\cM} f(\betas)$ is positive semidefinite since $\betas$ is a local minimizer of \ref{eq:main-problem2}. Conversely:
\begin{quote}
	\centering $\nabla^2_{\cM} f(\betas)$ is positive definite if and only if $\betas$ is a tilt-stable minimizer.
\end{quote}
In this case, the equality ${\rm REG}(\sD)^{-1}=\lambda_{\min}(\nabla^2_{\cM} f(\betas))$ holds.

\bigskip

The sensitivity matrix $\nabla\sigma(0)$ figures prominently in the asymptotic performance of estimation procedures. Notably, building on classical ideas due to H\'ajek and Le Cam, the recent paper of Duchi and Ruan~\cite{duchi2021asymptotic} established a lower bound on asymptotic covariance of  arbitrary estimators $\widehat \beta_n$ of $\betas$ using $n$ samples $z_{1}, \ldots , z_{n}.$ The precise lower bound is quite technical, and we refer the interested reader to their paper. In summary, their result shows that if $\betas$ is a tilt-stable minimizer, then the asymptotic covariance of $\sqrt{n}(\widehat\beta_n - \betas)$ is lower-bounded in the Loewner order by the matrix
\begin{equation}
	\label{eq:HLCCovariance}
	\hlS:=\nabla \sigma(0)\cdot \Cov(\nabla \ell(\betas; z))\cdot \nabla \sigma(0).
\end{equation}
Moreover, \revised{for the problems considered in this work} the expression in \eqref{eq:HLCCovariance} simplifies to
% \todo{added ``equality up to a scalar factor'' (equality requires scaling by the dispersion parameter or noise variance)}
$\hlS \asymp \nabla \sigma(0)$ (equality up to a scalar factor); this is the case for example for (quasi) maximum likelihood estimation (Section~\ref{sec:glm}) and rank-one matrix regression problems (Section~\ref{sec:matrix-sensing}). Thus, asymptotically the best error that any estimator can achieve in the direction $u$ is on the order $n^{-1/2}\cdot u^{\top}\hlS u$. The direction $u$ with the worst error matches the top eigenvector of $\hlS$ and the number of samples necessary to find an accurate approximation of $\betas$ grows with $\lambda_{\max}(\hlS)$. Reassuringly, typical algorithms such as sample average approximation \cite{van2000asymptotic} and the stochastic projected gradient method \cite{polyak92,davis2023asymptotic} match the lower bound  \eqref{eq:HLCCovariance} and thus are asymptotically optimal.

In the rest of the paper, we analyze data distributions $\sD'$, nearest to a fixed measure $\sD$, for which the problem ${\rm SP}(\sD')$ admits unstable minimizers.
The formal definition will depend on whether the learning problem is of supervised or unsupervised type. We now describe these two settings in turn. The goal of unsupervised learning is to learn some property of a distribution $\sD$ from finitely many samples $z_1,\ldots, z_n\stackrel{\text{iid}}{\sim} \sD$. Dimension reduction with principal component analysis (PCA) is a primary example.
In this case, the solution $\betas$ of \ref{eq:main-problem2} strongly depends on the distribution $\sD$. Therefore a natural measure of robustness of the problem is the size of the smallest perturbation in the Wasserstein-2 distance $W_2(\sD',\sD)$ so that the problem $\rm SP(\sD')$ has an unstable minimizer. 

\begin{defn}{Radius of statistical efficiency (unsupervised)}{}
	Consider the problem \ref{eq:main-problem} and let $\Q \subset \cP_2(\cZ)$ be a distinguished set of permissible distributions. Define the \emph{set of ill-conditioned distributions} as
	\begin{equation}
%		\label{eq:17}
		\cE := \left\{\sD' \in \Q \mid \text{There exists a minimizer of ${\rm SP}(\sD')$ that is unstable}\right\}\!.
	\end{equation}
	The \emph{radius of statistical efficiency} of $\sD$ is defined to be
	$$
	\RSE(\sD) := W_{2}(\sD, \cE).
	$$
\end{defn}

\newpage
Problems of supervised learning are distinctly different. The data consists of pairs $z=(x,y)\sim \sD$ comprised of feature data $\x\sim \sD_x$ and labels $y\sim \sD_{y\mid \x}$. We will assume that the conditional distribution 
$\sD_{y\mid \x}$ depends on the features $\x$ and a latent parameter $\beta_{\star}$. A typical example is the setting of regression under a model $y=g(\x,\betas)+\varepsilon$ where $\varepsilon$ is a random noise variable that is independent of $\x$. The  goal of the corresponding optimization problem \ref{eq:main-problem2} is to recover $\betas$. In contrast to unsupervised learning, the latent parameter $\betas$ is fixed a priori and is not a function of the data distribution. Therefore a natural measure of robustness of the problem is the size of the smallest perturbation to the feature vectors $W_2(\sD'_x,\sD_x)$ so that the problem ${\rm SP}(\sD'_x\times {\sD}_{y\mid\x})$ has an unstable minimizer. Notice that the conditional distributions of $y$ given $x$ coincide for the two measures $\sD$ and $\sD'=\sD'_x\times \sD_{y\mid\x}$. This type of shift exclusively only in the feature data appears often in the literature under the name of {\em covariate shift}~\cite{shimodaira2000improving,shift_book,agarwal2011linear,wen2014robust,reddi2015doubly}. 

\begin{defn}{Radius of statistical efficiency (supervised)}{}
	Consider the problem \ref{eq:main-problem} with $\sD=\sD_x\times \sD_{y\mid \x}$ and let $\betas$ be its minimizer. Let $\Q \subset \cP_2$ be a distinguished set of permissible distributions. Define the \emph{set of ill-conditioned distributions} as
	\begin{equation}
%		\label{eq:17}
		\cE := \left\{\sD'_x \in \Q \mid \text{$\betas$ is not a tilt-stable minimizer of ${\rm SP}(\sD'_x\times \sD_{y\mid\x})$}\right\}\!.
	\end{equation}
	The \emph{radius of statistical efficiency} of $\sD$ is defined to be
	$$
	\RSE(\sD) := W_{2}(\sD_x, \cE).
	$$
\end{defn}

In the supervised setting, we will sometimes write $\RSE(\sD_x)$ in place of $\RSE(\sD)$, and, similarly, ${\rm SP}(\sD_x)$ and ${\rm REG}(\sD_x)$ in place of ${\rm SP}(\sD)$ and ${\rm REG}(\sD)$, respectively. In either setting, the quantity ${\rm RSE}(\sD)$ evidently measures the robustness of the problem because it quantifies the size of a neighborhood around $\sD$ for which all problem instances are stable. There is a small nuance in formalizing this statement due to a lack of compactness in the Wasserstein space; \revised{specifically, closed balls in $\cP_2$ are not compact.} To navigate this difficulty, we impose the minor assumption that any sequence of measures $\nu_i$ for which the problem ${\rm SP}(\nu_i)$ becomes progressively harder (${\rm REG}(\nu_i)\to \infty$) must approach the set of ill-conditioned distributions (${\rm RSE}(\nu_i)\to 0$). In all examples we consider, this holds at least on bounded sets $\cQ'\subset\cP_2$.
The proof of the following elementary proposition appears in Appendix~\ref{sec:proof_dist_rad}.

\begin{proposition}[RSE as a robustness measure]\label{prop:elementary_dist_rad}
Fix a set $\cQ'\subset\cQ$ and suppose that for any sequence of measures $\nu_i\in \cQ'\setminus \cE$ the implication holds:
\begin{equation}\label{eqn:basic_eqncompat}
{\rm REG}(\nu_i)\to \infty\qquad \Longrightarrow\qquad {\rm RSE}(\nu_i)\to 0.
\end{equation}
Then for any measure $\mu\in \cQ'\setminus \cE$ and any radius $0<r<{\rm RSE}(\mu)$, we have
\begin{equation}\label{eqn:blup_desired_M}
	\sup_{\nu\in \Q':\, W_2(\nu,\mu)\leq r}{\rm REG}(\nu)<+\infty.
\end{equation}
%\begin{equation}\label{eqn:blup_desired_M}
%{\rm REG}(\nu)\leq M~\textrm{for all}~\nu\in\cQ'~\textrm{with}~ W_2(\nu,\mu)\leq r.
%\end{equation}
Moreover, if for some $c,q>0$, the inequality ${\rm RSE}(\nu)^q\leq c\cdot {\rm REG}(\nu)^{-1}$ holds for all $\nu \in \Q'\setminus \cE$, then the supremum in \eqref{eqn:blup_desired_M} is upper-bounded by $c\cdot ({\rm RSE}(\mu)-r)^{-q}$. 
\end{proposition}

At first sight, it appears that computing ${\rm RSE}(\sD)$ in concrete problems is difficult. Indeed, the set of ill-conditioned distributions $\cE$ may be quite exotic and computing ${\rm RSE}(\sD)$ amounts to estimating the Wasserstein-2 distance $W_2(\sD,\cE)$. In contrast, computing the regularity modulus ${\rm REG}(\sD)$ should be relatively straightforward. The key observation now is that the two quantities ${\rm RSE}(\sD)$ and ${\rm REG}(\sD)$ are closely related since $\cE$ is the set of minimizers of the function $\mathcal{J}(\mu):={\rm REG}(\mu)^{-1}$ (i.e., $\cE$ coincides with the zero set $[\cJ = 0]$). Thus, it would be ideal if there were a quantitative relationship of the form
\begin{equation}\label{eqn:error_bound_ineq}
(W_2(\mu,\cE))^{\ell_1}\lesssim \mathcal{J}(\mu)\lesssim (W_2(\mu,\cE))^{\ell_2}\qquad \forall \mu\in \cP_2
\end{equation}
for some constants $\ell_1, \ell_2\geq1$. The upper bound should be elementary to establish because it amounts to upper-bounding the growth of the function $\mathcal{J}(\mu)$. The lower bound is more substantial because it requires lower-bounding the growth of $\mathcal{J}(\mu)$ by a nonnegative function of the distance. Such lower-estimates are called {\em error bounds} in nonlinear analysis and can be checked by various ``slope''-based conditions. See, for example, the authoritative monographs on the subject \cite[Chapter 8]{cui2021modern} and \cite[Chapter 3]{ioffe2017variational}. Indeed, Demmel's original work \cite{demmel1987condition} relies on verifying an error bound property as well, albeit in the much simpler Euclidean setting. The following theorem provides a sufficient condition \eqref{eqn:advanced_ineq} ensuring the relationship \eqref{eqn:error_bound_ineq}, \revised{which appears in the equivalent form
\begin{equation*}
	{\rm REG}(\mu)^{q_2-1}\lesssim {\rm RSE}(\mu)\lesssim {\rm REG}(\mu)^{q_1-1},
\end{equation*} 
where $q_i \in[0,1)$ corresponds to \eqref{eqn:error_bound_ineq} via  $q_i = 1 - \ell_i^{-1}$ for $i=1,2$.} 
We state the theorem loosely by compressing all multiplicative numerical constants via the symbol $\lesssim$. More precise and sharper guarantees appear in Appendix~\ref{sec:proof_sec_slopes}. In the theorem statement, the symbol $DF(x)\colon \R^d\to \bS^k$ denotes the differential of $F$ at $x$, while the symbol $DF(x)^*\colon \bS^k\to \R^d$ denotes the adjoint linear map of $DF(x)$.

\begin{THM}[label={thm:infinitesimal}]{(Infinitesimal characterization of RSE)}{}
	Consider the problem \ref{eq:main-problem} of supervised learning and let $\betas$ be its minimizer. Set $\Q =\cP_2$ and suppose the Hessian $\nabla^2_{\cM} f(\betas)$ corresponding to any measure $\mu\in \cP_2$ can be written as 
	$$\nabla^2_{\cM} f(\betas)={\EE}_{\mu} F(x)$$ 
	for some $C^1$-smooth map $F\colon\R^d\to\bS^k_+$ satisfying $\|DF(x)\|_{\rm op}\lesssim 1+\|x\|_2$ for all $x\in \R^d$. Suppose there exist $q_1,q_2 \in [0,1)$ such that
	for all measures $\nu\in \cP_2\setminus\cE$, the estimate
	\begin{equation}\label{eqn:advanced_ineq}
\lambda_{\min}({\EE}_{\nu}F(x))^{q_1}\lesssim \sqrt{{\EE}_{\nu}\|DF(x)^*[uu^\top]\|^2_2}
\lesssim \lambda_{\min}({\EE}_{\nu}F(x))^{q_2}
\end{equation}
holds for some eigenvector $u\in \mathbb{S}^{k-1}$ of the matrix ${\EE}_{\nu}F(x)$ corresponding to its minimal eigenvalue. Then for every $\mu\in \cP_2$, the following estimate holds:
\begin{equation}\label{eqn:from_rse_to_reg}
{\rm REG}(\mu)^{q_2-1}\lesssim {\rm RSE}(\mu)\lesssim {\rm REG}(\mu)^{q_1-1}.
\end{equation}
	\end{THM}

The expression \eqref{eqn:advanced_ineq} becomes particularly enlightening when $F(x)$ decomposes as $F(x)=g(x)g(x)^\top$ for some $C^1$-smooth map $g\colon\R^d\to \R^k$. This situation is typical for regression problems (see Section~\ref{sec:glm}). A simple computation then shows that the sufficient condition \eqref{eqn:advanced_ineq} reduces to 
$$\left(\EE_{\nu}\langle u, g(x)\rangle^2\right)^{q_1}\lesssim \sqrt{{\EE}_{\nu}\Big[\langle u, g(x)\rangle^2 \|\nabla g(x)^{\top}u\|^2_2\Big]}
\lesssim \left(\EE_{\nu}\langle u, g(x)\rangle^2\right)^{q_2}.
$$
Observe that all three terms would match exactly with $q_1=q_2=\frac{1}{2}$ were it not for the term $\|\nabla g(x)^{\top}u\|^2_2$ that reweighs the middle integral. It is this reweighing that may impact the values of $q_1$ and $q_2$. The salient feature of Theorem~\ref{thm:infinitesimal} is that it completely circumvents the need for explicitly estimating the distance to the exceptional set $\cE$. One could apply this theorem to some of the examples that will appear in the rest of the paper. That being said, in all the upcoming examples, we will be able to compute the distance to $\cE$ explicitly, thereby obtaining sharper estimates than would follow from Theorem~\ref{thm:infinitesimal}. \revised{The main drawback of this result is its stringent assumptions on the differential of $F,$ which may limit its applicability.} Nonetheless, we believe that Theorem~\ref{thm:infinitesimal} is interesting in its own right and may be useful for analyzing RSE in certain situations.

%=================================================================
\section{Principal component analysis}\label{section:pca}
Principal component analysis (PCA) is a common technique for dimension reduction. The goal of PCA is to find a low-dimensional subspace that captures the majority of the variance of the distribution. In this section, we compute the radius of statistical efficiency for PCA. Setting the stage, let $\x$ be a random vector in $\R^d$ drawn from a zero-mean distribution $\sD\in\cP_2^{\circ}(\R^d)$. A unit vector $v$ for which the random variable
$\langle v,\x\rangle$ has maximal variance is called the {\em first principal component} of $\sD$. Thus, the first principal component is a maximizer of the problem
\begin{equation}\label{eqn:basic_science0}
\max_{v\in \mathbb{S}^{d-1}}\,\mathop{\EE}_{x\sim\sD}\tfrac{1}{2}\langle v,x\rangle^2.
\end{equation}
 Equivalently, the first principal component is a unit eigenvector corresponding to the maximal eigenvalue of the covariance matrix $\Sigma_{\sD}:=\EE_{\sD}[xx^\top]$. Intuitively, the problem \eqref{eqn:basic_science0} is more challenging when the gap between the top two eigenvalues of $\Sigma_{\sD}$ is small. This is the content of the following lemma, whose proof appears in Appendix~\ref{sec:proof_pca_first}.

\begin{lemma}\label{lem:pca_first}
The set of ill-conditioned distributions for \eqref{eqn:basic_science0} is given by 
$$\mathcal{E}=\{\mu\in \cP_2^{\circ} \mid \lambda_1(\Sigma_{\mu})=\lambda_2(\Sigma_{\mu})\}.$$	
Moreover, for any $\mu\in \cP_2^{\circ}$, equality ${\rm REG} (\mu)^{-1}=\lambda_1(\Sigma_{\mu})-\lambda_2(\Sigma_{\mu})$ holds.
\end{lemma}

Computing the RSE amounts to determining the Wasserstein-2 distance of the base measure $\sD$ to the set $\cE$. The end result is the following theorem; we defer its proof to Appendix~\ref{sec:proof_pca1}.

\begin{THM}[label={thm:rse_pca1}]{(RSE for top principal component)}{}
	Consider the problem \eqref{eqn:basic_science0} with covariance matrix $\Sigma_\sD:=\EE_{\sD}[\x\x^\top]$. Then equality holds:
	\begin{equation}\label{eqn:rse_qmle_da_pca}
		\RSE(\sD) =  \tfrac{1}{\sqrt{2}}\bigg(\sqrt{\lambda_{1}(\Sigma_{\sD})}-\sqrt{\lambda_{2}(\Sigma_{\sD})}\bigg).
	\end{equation}
	In particular, if $\sD\in \cP_2^{\circ}\setminus \cE$, then  
	\begin{equation}\label{eqn:hardness_rse_toppca}
	\RSE(\sD)\cdot {\rm REG}(\sD)=\frac{1}{\sqrt{2}\Big(\sqrt{\lambda_{1}(\Sigma_{\sD})}+\sqrt{\lambda_{2}(\Sigma_{\sD})}\Big)}.
	\end{equation}
\end{THM}

% \noindent Thus, treating $\lambda_1(\Sigma_{\sD})$ in \eqref{eqn:hardness_rse_toppca} as being of constant order, we see that the hardness of the problem is inversely proportional to the distance to the nearest ill-posed problem: ${\rm REG}(\sD)\propto \RSE(\sD)^{-1}$.

 More generally, we may be interested in finding a $q$-dimensional subspace $\mathcal{V}\subset\R^d$ that captures most of the variance. Analytically, this amounts to solving the problem 
\begin{equation} \label{eqn:pca_brah}
	\max_{R\in {\rm Gr}(q, d)}~ f(R):=\mathop{\EE}_{x\sim \sD}\| R x\|^2_2,
\end{equation}
where the Grassmannian manifold ${\rm Gr}(q, d)$ consists of all orthogonal projections $R\in \bS^d$ onto $q$-dimensional subspaces of $\R^d$. The column space of an optimal matrix $R$ is called the {\em top $q$ principal subspace}. Equivalently, the top $q$ principal subspace is the span of the eigenvectors of $\Sigma_{\sD}$ corresponding to its top $q$ eigenvalues. The following lemma is a direct extension of Lemma~\ref{lem:pca_first}; see Appendix~\ref{section:proof_pca_multi} for a proof.

\begin{lemma}\label{lem:PCA_multi}
	The set of ill-conditioned distributions for \eqref{eqn:pca_brah} is given by 
	$$\mathcal{E}=\{\mu\in \cP_2^{\circ} \mid \lambda_q(\Sigma_{\mu})=\lambda_{q+1}(\Sigma_{\mu})\}.$$	
	Moreover, for any $\mu\in \cP_2^{\circ}$, equality ${\rm REG} (\mu)^{-1}=\lambda_q(\Sigma_{\mu})-\lambda_{q+1}(\Sigma_{\mu})$ holds.
\end{lemma}

Again, computing the RSE amounts to determining the Wasserstein-2 distance of the base measure $\sD$ to the set $\cE$. The end result is the following theorem; the proof appears in Appendix~\ref{sec:proof_pca2}.

\begin{THM}[label={thm:rse_pca2}]{(RSE for PCA)}{}
	Consider the problem \eqref{eqn:pca_brah} with covariance matrix $\Sigma_{\sD}:=\EE_{\sD}[\x\x^\top]$. Then equality holds:
	\begin{equation*}%\label{eqn:rse_qmle_da_pca_full}
		\RSE(\sD) =  \tfrac{1}{\sqrt{2}}\bigg(\sqrt{\lambda_{q}(\Sigma_\sD)}-\sqrt{\lambda_{q+1}(\Sigma_\sD)}\bigg).
	\end{equation*}
	In particular, if $\sD\in \cP_2^{\circ}\setminus \cE$, then 
\begin{equation}\label{eqn:hardness_rse_toppca2}
	\RSE(\sD)\cdot {\rm REG}(\sD)=\frac{1}{\sqrt{2}\Big(\sqrt{\lambda_{q}(\Sigma_{\sD})}+\sqrt{\lambda_{q+1}(\Sigma_{\sD})}\Big)}.
\end{equation}
\end{THM}

\noindent We see two different regimes for the reciprocal relationship between RSE and REG. %similarly to the case $q=1$, treating
When \begin{equation}\label{eq:prop-eig}\sqrt{\lambda_q(\Sigma_{\sD})} + \sqrt{\lambda_{q+1}(\Sigma_{\sD})} \gg \sqrt{\lambda_q(\Sigma_{\sD})} - \sqrt{\lambda_{q+1}(\Sigma_{\sD})},\end{equation} the hardness of the problem is approximately inversely proportional to the distance to ill-posed problems: ${\rm REG}(\sD)\appropto \RSE(\sD)^{-1}$. Alternatively, when the two quantities in \eqref{eq:prop-eig} are proportional to each other, we have ${\rm REG}(\sD)\propto \RSE(\sD)^{-2}$.

The proofs of Theorems~\ref{thm:rse_pca1} and \ref{thm:rse_pca2} rely on estimating the distance to the exceptional sets $\cE$. Notice that these two sets are defined purely in terms of the spectrum of the covariance matrix. Although such ``spectral'' sets in $\cP_2^\circ$ are quite complicated, their distance can be readily computed. Geometric properties of such sets are explored in Appendix~\ref{sec:spec_funct} and may be of independent interest.

%=================================================================
\section{Generalized linear models and (quasi) maximum likelihood estimation}\label{sec:glm}
In this section, we compute the RSE for a large class of supervised learning problems arising from
(quasi) maximum likelihood estimation (QMLE). The goal is to estimate a parameter $\betas \in \cM$ where the constraint set $\cM$ is a $C^2$-smooth embedded submanifold of $\R^d$. Setting the stage, suppose that we have an $L^2$ random feature vector $\x$ and an $L^2$ random label variable $y$ satisfying the generalized linear model (GLM) conditions
\begin{equation}\label{eq:glm}
	{\EE}\big[ y \!\mid\! \x \big] = h'(\langle \x, \beta_{\star} \rangle) \qquad\text{and}\qquad {\Var}\big[ y \!\mid\! \x \big] = \sigma^2 \cdot h''(\langle \x, \beta_{\star} \rangle)
\end{equation}
for some known $C^2$-smooth convex function
$h\colon\R \rightarrow\R$ with $h''>0$
and parameter $\sigma^2 > 0$. The function $h$ is called the \emph{cumulant function} of the model $\eqref{eq:glm}$, and $\sigma^2$ the dispersion parameter.
Here and from now on, we make the blanket assumption that the distribution $\sD_x$ of $x$ lies in a subset $\Q$ of  $\cP_2$  comprised of probability measures satisfying sufficient regularity conditions to take expectations and to interchange differentiation and expectation as necessary. In particular, we assume that the function $\phi\colon \cM \rightarrow\R$ given by $\phi(\beta) = {\EE}[h(\langle \x, \beta \rangle)]$ is well defined and has a $C^2$-smooth local extension to a neighborhood of $\betas$ in $\R^d$ with $\nabla\phi(\betas) = {\EE}[h'(\langle \x, \betas \rangle)\x]$ and $\nabla^{2}\phi(\betas) = \EE[h''(\langle \x, \betas \rangle)\x\x^{\top}]$. 

Following the seminal work of McCullagh \cite{McCullagh1983}, we consider the QMLE problem

\begin{equation}\label{eq:glm_min}
	\min_{\beta\in\cM}f(\beta) := \mathop{\EE}_{(x,y)\sim\sD}\big[h(\langle \x, \beta \rangle) - y\langle \x, \beta \rangle \big].
\end{equation}
The function $f$ given in \eqref{eq:glm_min} is called the negative log quasi-likelihood of the GLM \eqref{eq:glm}. In general, $f$ is not the negative of the log likelihood function, yet it shares many of its properties and hence the name. The motivation for this loss function comes from the canonical example of \eqref{eq:glm} where the conditional density of $y$ given $\x$
admits an exponential-family formulation.
In this case, standard maximum likelihood estimation of $\betas$ coincides with \eqref{eq:glm_min}. As an illustration, Table~\ref{table:glm_regression_1} lists some common GLM examples.

\begin{table}[h]
	\renewcommand{\arraystretch}{1.3}
	\begin{center}
		\begin{tabular}{lllll}
			% \begin{tabular}{| c | c | c |  c | c |}
			\toprule
			\textbf{Model}
			         & \textbf{Response variable}
			         & \textbf{Cumulant} $h(\theta)$                                                                                                            			         & \textbf{Second derivative} $h''(\theta)$
			\\
			\midrule
			Linear   & $y = \langle \x, \betas \rangle + \varepsilon$                                                                  & $\tfrac{1}{2}\theta^2$ 
			         & $1$
			\\
			Logistic & $y\!\mid\!\x \sim \mathsf{Ber}\Big(\tfrac{\exp\langle \x,\betas \rangle}{1+\exp\langle \x,\betas \rangle}\Big)$ & $\log(1 + \exp\theta)$ 			         & $\frac{\exp\theta}{(1+\exp\theta)^2}$
			\\
			Poisson  & $y\!\mid\!\x \sim \mathsf{Poi}({\exp}\langle \x,\betas \rangle)$                                                & $\exp\theta$           
			         & $\exp\theta$
			\\
			Gamma    & $y\!\mid\!\x \sim \Gamma(\sigma^{-2}, -\sigma^{-2}\langle \x,\betas \rangle)$                                   & $-\log(-\theta)$
			         & $\theta^{-2}$
			\\
			\bottomrule
		\end{tabular}
	\end{center}
	\caption{GLM examples corresponding to \eqref{eq:glm}. In the linear regression model, $\varepsilon$ is random noise 
	satisfying ${\EE}[\varepsilon \!\mid\! \x] = 0$ and ${\Var}[\varepsilon \!\mid\! \x]=\sigma^2$. 	}
	\label{table:glm_regression_1}
\end{table}

We begin with the following lemma that characterizes the set of ill-conditioned problem instances. Henceforth, we let  $\mathcal{T}:=T_{\cM}(\betas)$ denote the tangent space of $\cM$ at $\betas$.

\begin{lemma}\label{lem:qmle}
	The set of ill-conditioned distributions for \eqref{eq:glm_min} is given by
	$$\mathcal{E}=\big\{\mu\in \Q \mid \mathcal{T}\cap \ker(\Sigma_{\mu})\neq \{0\}\big\}.$$
	Moreover, for any $\mu\in \Q\setminus \cE$ for which there exist $\lb,\ub>0$ satisfying $\lb \leq  \revised{h''(\langle x,\beta_{\star}\rangle)^{-1}} \leq \ub$ for $\mu$-almost every $x$, we have
	$\lb\leq {\rm REG} (\mu)\cdot \lambda_{\min}(\Sigma_{\mu}|_{\cT})\leq \ub $.
\end{lemma}
The proof of this lemma is deferred to Appendix~\ref{sec:proof_qmle}.
The RSE for the problem \eqref{eq:glm_min} follows quickly by computing the distance to the set $\cE$ in Lemma~\ref{lem:qmle}; see Appendix~\ref{sec:proof_rse_qmle} for a proof.

\begin{THM}[label={thm:RSE_QMLE}]{(RSE for QMLE)}{}
	Consider the QMLE problem \eqref{eq:glm_min} and let $\sD_{\x}\in \Q$ be the distribution of $\x$ with second-moment matrix $\Sigma:=\EE[\x\x^\top]$. \revised{Suppose that $\cQ$ satisfies the following closure property: for every $\mu\in\cQ$, we have $(\sP_{v^{\perp}})_{\#}\mu\in\cQ$ for some minimal eigenvector $v$ of $\Sigma_{\mu}\vert_{\cT}$.} Then the identity holds:
	\begin{equation}\label{eqn:rse_qmle_da}
		\RSE(\sD) =  \sqrt{\lambda_{\min}(\Sigma\vert_{\cT})}.
	\end{equation}
	In particular, if $\sD_x\in\cQ\setminus\cE$ and for some $\lb,\ub>0$ the inequality $\lb\leq \revised{h''(\langle x,\beta_{\star}\rangle)^{-1}} \leq \ub$ holds for $\sD_x$-almost every $x$, then we have
	$$\lb\leq \RSE(\sD)^2\cdot {\rm REG}(\sD)\leq \ub.$$
\end{THM}

Thus we see that under mild conditions, the hardness of the problem is approximately inversely proportional to the square of the distance to the nearest ill-posed problem: ${\rm REG}(\sD)\appropto \RSE(\sD)^{-2}$. Unlike the reciprocal relationship for PCA, the exponent here does not change depending on the spectrum of the second-moment matrix.

Aside from the examples in Table~\ref{table:glm_regression_1}, an interesting problem instance occurs in sparse recovery.  Namely, let $\cM$ be the submanifold of $\R^d$ comprised of $k$-sparse vectors, i.e., those that have precisely $k$ nonzero components. Then the tangent space of $\cM$ at $\betas$ is the $k$-dimensional subspace of $\R^d$ on which $\betas$ is supported:
\begin{equation*}
	\cT = {\Span}\{\be_i \mid \langle \be_i, \betas \rangle \neq 0\},
\end{equation*}
where $\{\be_1,\ldots,\be_d\}$ denotes the standard basis of $\R^d$. %In this case any instance of \eqref{eq:glm_min} constrained to $\cM$ is ill-conditioned precisely when $\Supp(\x)\subset \bv^\perp$ for some nonzero $\bv\in\cT$.
 The right-side of \eqref{eqn:rse_qmle_da} in this case is the square root of the minimum eigenvalue of the submatrix of $\Sigma$ whose  columns and rows are indexed by the nonzero coordinates of $\betas$.

\revised{Nonsmooth variants of the QMLE problem \eqref{eq:glm_min} arise commonly in practice, and it is possible for the framework of this section to be extended to handle such problems by introducing more delicate regularity assumptions that guarantee the existence of an active manifold near the minimizer, as in \cite{davis2023asymptotic}. For example, consider the LASSO least squares problem
\begin{equation}\label{eq:LASSO}
	\min_{\beta:\,\|\beta\|_{1}\leq t} \,\tfrac{1}{2n}\|X\beta - y\|_{2}^2,
\end{equation}
where $X$ is an $n \times d$ feature matrix with full column rank. Then \eqref{eq:LASSO} admits a unique solution $\betas$, and under sufficient regularity assumptions (e.g., strict complementarity), the active manifold $\cM$ on which \eqref{eq:LASSO} restricts to a smooth problem is the relative interior of the lowest-dimensional face of the $\ell_1$ ball $\{\beta \in\R^d \mid \|\beta\|_{1} \leq t\}$ containing $\betas$, where $t=\|\betas\|_{1}$. In this case, \eqref{eqn:rse_qmle_da} applies with $\cT = \Span(\cM - \betas)$.}

%=================================================================
\section{Rank-one matrix regression problems}\label{sec:matrix-sensing}
In this section, we compute the RSE for a number of tasks in low-rank matrix recovery---a large class of problems with numerous applications in  control, system identification, recommendation systems, and machine learning. See, for example, \cite{davenport2016overview,chi2019nonconvex,charisopoulos2021low,wright2022high} for an overview.
Setting the stage, consider the measurement model
\begin{equation}
	\label{eq:matrix-labels}
	y = \dotp{\bX, \Ms} +\varepsilon,
\end{equation}
where the matrix $\bX\in \R^{d_1\times d_2}$ is drawn from some distribution $\sD_X$, the noise $\varepsilon$ is independent of  $\bX$ and has zero mean and variance $\sigma^2$, and $\Ms$ is a rank-$r$ matrix. The goal of low-rank matrix recovery is to estimate the latent parameter $\Ms$ from finitely many i.i.d.\ samples $(\bX_{1}, y_{1}), \dots, (\bX_{n}, y_{n})$. For the rest of the section, we focus on the simplest case of rank $r=1$ matrix recovery.

%=================================================================
\subsection{Phase retrieval}\label{sections:single-index}
The problem of phase retrieval corresponds to \eqref{eq:matrix-labels} with the ground truth $\Ms=\betas\betas^\top$ and the feature data $\bX=xx^\top$ being PSD rank-one matrices. We will let $\sD_x\in \cP_4(\R^d)$ denote the distribution of $\x\in \R^d$ and let $\sD$ denote the joint distribution of $(x,y)$. There are two standard ways to write the phase retrieval problem as a problem of stochastic optimization. The first is simply to minimize the mean squared error over PSD rank-one matrices:
\begin{equation}\label{eq:phaseretrival_rank1}
  \min_{{M\in\bm{S}_+^d:\, \rank(M)=1}}\,f(M) := \mathop{\EE}_{(x,y)
    \sim \sD}\tfrac{1}{2} \big(x^{\top} M x - y\big)^{2}.
\end{equation}
Alternatively, one may parameterize PSD rank-one matrices as $M=\beta \beta^\top$ and then minimize the mean squared error over the factors:
\begin{equation}\label{eq:phaseretrival}
	\min_{{\beta\in \R^d\setminus\{0\}}}\,f(\beta) := \mathop{\EE}_{(x,y)\sim \sD}\tfrac{1}{8}\big(\langle x, \beta \rangle^2 - y\big)^{2}.
\end{equation}
From the viewpoint of RSE, there is no significant distinction between these two formulations. The proof of the next lemma appears in Appendix~\ref{sec:proof_basic_phase}.

\begin{lemma}\label{lem:basic_except_1_phase}
The set of ill-conditioned distributions for both \eqref{eq:phaseretrival_rank1} and \eqref{eq:phaseretrival}  is given by 
\begin{equation*}
	\mathcal{E} = \bigcup_{v\in \mathbb{S}^{d-1}}\Big\{\mu\in\mathcal{P}_4(\R^d) \mid \Supp(\mu)\subset \betas^\perp \cup v^\perp\Big\}.
\end{equation*}
Fix any measure $\mu\in \cP_4\setminus \cE$ and define the matrix $\hat\Sigma_{\mu}:={\EE}_{\mu} \big[\langle x,\beta_{\star}\rangle^2xx^\top\big]$. Then
\begin{equation}
\lb \leq {\rm REG} (\mu)\cdot \lambda_{\min}(\hat\Sigma_{\mu})\leq \ub,\label{eqn:Lweird_asf2}
\end{equation}
where $(\lb,\ub)=\big(\tfrac{1}{2}\|\betas\|^2_2 , \|\betas\|^2_2\big)$ for problem \eqref{eq:phaseretrival_rank1} and $(\lb,\ub)=(1,1)$ for problem \eqref{eq:phaseretrival}.
\end{lemma}

It remains to compute the distance to the exceptional set $\cE$. The following result shows that minimizing the expected squared distance to $\beta_{\star}^\perp \cup v^\perp$ over $v\in\mathbb{S}^{d-1}$ yields the squared $W_2$ distance to $\cE$; its proof appears in Appendix~\ref{sec:proof_pse_generalphase}.

\begin{THM}[label={pse_generalphase}]{(RSE for phase retrieval)}{}
	For problems \eqref{eq:phaseretrival_rank1} and \eqref{eq:phaseretrival}, the following equality holds:
	\begin{equation*}%\label{eqn:rse_qmle_da_pca_full}
		 {\rm RSE}(\sD) =  \min_{v \in \mathbb{S}^{d-1}} \sqrt{\mathop{\EE}_{x\sim\sD_x} \!\bigg[ \Big\langle x, \mfrac{\beta_{\star}}{\|\beta_{\star}\|}\Big\rangle^2 \wedge \big\langle x, v \big\rangle^2 \bigg]}.
	\end{equation*}
\end{THM}

In light of Lemma~\ref{lem:basic_except_1_phase} and Theorem~\ref{pse_generalphase}, the following equalities hold for the formulation \eqref{eq:phaseretrival}:
\begin{align*}
{\rm RSE}(\sD) &=  \min_{v \in \mathbb{S}^{d-1}} \sqrt{\mathop{\EE}_{x\sim\sD_x} \!\bigg[ \Big\langle x, \mfrac{\beta_{\star}}{\|\beta_{\star}\|}\Big\rangle^2 \wedge \big\langle x, v \big\rangle^2 \bigg]},\\
{\rm REG}(\sD)^{-1} &=  \min_{v \in \mathbb{S}^{d-1}} \mathop{\EE}_{x\sim\sD_x} \!\Big[ \big\langle x, \beta_{\star}\big\rangle^2 \cdot \big\langle x, v \big\rangle^2 \Big].
\end{align*}

A moment of thought leads one to realize that for reasonable distributions, the first equation should scale as $\sqrt{\lambda_{\min}(\Sigma_{\sD_x})}$, while the scaling of the second should be at least $\lambda_{\min}(\Sigma_{\sD_x})$.
We now verify that this is indeed the case when the base distribution is Gaussian with $x \sim \mathsf{N}(0,\Sigma)$ for some covariance matrix $\Sigma\succeq 0$.  To this end, define the two functions
\begin{align*}
h_{\Sigma}(u, v) &:= \mathop{\EE}_{x \sim \mathsf{N}(0, \Sigma)}\!\big[ \langle x, u\rangle^{2} \wedge \langle x, v\rangle^{2} \big],\\
g_{\Sigma}(u,v) &:= \mathop{\EE}_{x \sim \mathsf{N}(0, \Sigma)}\!\big[ \langle x, u\rangle^{2} \cdot \langle x, v\rangle^{2} \big],
\end{align*}
where  $u, v \in \mathbb{S}^{d-1}$ vary over the unit sphere. We defer the proof of the next result to Appendix~\ref{sec:proof_computeforgauss}. 
\begin{theorem}\label{thm:compute_for_gauss}
  For any $u\in \mathbb{S}^{d-1}$, the following estimates hold:
  \begin{align}
  \big(1-\tfrac{2}{\pi} \big) \lambda_{\min}(\Sigma) \leq  &\min_{v \in \mathbb{S}^{d-1}}  h_{\Sigma}(u,v)\leq \lambda_{\min}(\Sigma),\label{eqn:min}\\
  \lambda_{\min}(\Sigma)\cdot \langle \Sigma u,u \rangle\leq &\min_{v \in \mathbb{S}^{d-1}}  g_{\Sigma}(u,v)\leq 3\lambda_{\min}(\Sigma)\cdot\langle \Sigma u,u \rangle\label{eqn:product}.
  \end{align}
\end{theorem}

Combining Lemma~\ref{lem:basic_except_1_phase}, Theorem~\ref{pse_generalphase}, and Theorem~\ref{thm:compute_for_gauss} directly yields the following estimate on RSE with Gaussian feature data.

\begin{THM}{(RSE for phase retrieval with Gaussian data)}{}
	Consider problems \eqref{eq:phaseretrival_rank1} and  \eqref{eq:phaseretrival}  with Gaussian data $\sD_x= \mathsf{N}(0,\Sigma)$. Then 
	\begin{equation*}%\label{eqn:rse_qmle_da_pca_full}
		 \sqrt{(1-\tfrac{2}{\pi})\lambda_{\min}(\Sigma)}\leq {\rm RSE}(\sD) \leq  \sqrt{\lambda_{\min}(\Sigma)}.
	\end{equation*}
	Moreover, if $\mathsf{N}(0,\Sigma)\notin\cE$, then the following estimate holds:
\begin{equation}
  \lb\cdot \frac{1-\tfrac{2}{\pi}}{3\langle \Sigma\betas,\betas\rangle}\leq \RSE(\sD)^2\cdot {\rm REG}(\sD)\leq \ub\cdot \frac{1}{\langle \Sigma\betas,\betas\rangle}
\end{equation}
where $(\lb,\ub)=\big(\tfrac{1}{2}\|\betas\|^2_2,\|\betas\|^2_2\big)$ for problem \eqref{eq:phaseretrival_rank1} and $(\lb,\ub)=(1,1)$ for problem \eqref{eq:phaseretrival}.
\end{THM}

Interestingly, we see a different scaling between $\RSE(\sD)$ and ${\rm REG}(\sD)$ depending on whether $\betas$ aligns with the minimal eigenspace of $\Sigma$. In the regime $\langle\Sigma\betas,\betas\rangle/\|\betas\|_{2}^2 \gg \lambda_{\min}(\Sigma)$, we observe the scaling ${\rm REG}(\sD)\appropto\RSE(\sD)^{-2}$, while in the regime $\langle\Sigma\betas,\betas\rangle/\|\betas\|_{2}^2 \approx \lambda_{\min}(\Sigma)$, we observe ${\rm REG}(\sD)\appropto\RSE(\sD)^{-4}$. Thus, in the latter regime, the distance to the nearest ill-posed problem has a much stronger effect on the hardness of the problem.

%=================================================================
\subsection{Bilinear sensing}\label{sections:blind-deconvolution}
The problem of bilinear sensing is an asymmetric analogue of phase retrieval, that is, the ground truth $\Ms=\beta_{1\star}\beta_{2\star}^\top$ and feature data $X=x_1x_2^\top$ are rank-one $d_1\times d_2$ matrices, where the factors $x_1\sim \sD_{x_1}$ and $x_2\sim \sD_{x_2}$ are independent. The standard way to write this problem as stochastic optimization is to minimize the mean squared error over rank-one rectangular matrices:
\begin{equation}\label{eq:phaseretrival_rank1_blind}
	\min_{{M\in \R^{d_1\times d_2}:\, \rank(M)=1}}\,f(M) :=  \mathop{\EE}_{(x_1,x_2,y)\sim\sD}\tfrac{1}{2}\big(x^{\top}_1 M x_2 - y\big)^{2}.
\end{equation}
Throughout, we fix $\Q= \cP_4(\R^{d_1})\times \cP_4(\R^{d_2})$ as the set of allowable feature data distributions. We disregard the factorized formulation with $M = \beta_{1} \beta_{2}^{\top}$ because it results in a continuum of minimizers that are not tilt-stable. This technical difficulty could be circumvented by introducing an additional constraint, such as %\todo{Add $\|\beta_2\|=1$?}
$\|\beta_{1}\| =1$; however, we do not pursue this approach to simplify the exposition.

The following lemma characterizes the set of ill-conditioned problems; see Appendix~\ref{sec:proof_basicillcond} for a proof. For any PSD matrix $\Sigma$, the symbol $\kappa(\Sigma)=\lambda_{\max}(\Sigma)/\lambda_{\min}(\Sigma)$ denotes its condition number.

\begin{lemma}\label{lem:basic_ill_cond}
	The set of ill-conditioned distributions for  \eqref{eq:phaseretrival_rank1_blind} is given by
	\begin{equation*}
		\mathcal{E} = \Big\{\mu\times\nu\in\Q\mid {\rm either }~\EE_{\mu}[xx^{\top}]~{\rm or ~}\EE_{\nu}[xx^{\top}] {\rm ~is~ singular}\Big\}.
	\end{equation*}
	Moreover, for any $\mu\times\nu\in\Q$ with $\Sigma_1:=\EE_{\mu}[x_1x_1^{\top}]$ and $\Sigma_2:=\EE_{\nu}[x_2x_2^{\top}]$, we have
	\begin{equation} \label{eqn:abithard}\frac{2\cdot\min\{\gamma_2\lambda_{\min}(\Sigma_1),\gamma_1\lambda_{\min}(\Sigma_2) \}}{\kappa(\Sigma_{1})\kappa(\Sigma_{2})+1}\leq {\rm REG}(\mu\times \nu)^{-1}\leq \min\{\gamma_2\lambda_{\min}(\Sigma_1),\gamma_1\lambda_{\min}(\Sigma_2) \} \end{equation}
	where $\gamma_i:=\left\langle\Sigma_i\frac{\beta_{i\star}}{\|\beta_{i\star}\|_2},\frac{\beta_{i\star}}{\|\beta_{i\star}\|_2}  \right\rangle$ for $i=1,2$.
    %\revised{Moreover, there are points  $\beta_{1\star} \in \RR$ and $\beta_{2\star}\in \RR^{2}$, and a distribution $\mu \times \nu$ that make the lower bound in \eqref{eqn:abithard} hold tightly.}
\end{lemma}
Thus, an application of Theorem~\ref{prop:Sdist} immediately yields the following expression for RSE.

\begin{THM}[label={thm:RSE_bilin}]{(RSE for bilinear sensing)}{}
	Consider the problem \eqref{eq:phaseretrival_rank1_blind} with $\Sigma_1:=\EE_{\sD_{x_1}}[x_1x_1^{\top}]$ and $\Sigma_2:=\EE_{\sD_{x_2}}[x_2x_2^{\top}]$. Then
	\begin{equation}%\label{eqn:rse_qmle_da_pca_full}
		{\rm RSE}(\sD) =  {\min}\bigg\{\sqrt{\lambda_{\min}(\Sigma_1)},\sqrt{\lambda_{\min}(\Sigma_2)}\bigg\}.
	\end{equation}
	In particular, if $\sD_{x_1}\times\sD_{x_2}\in\cQ\setminus\cE$, then the following estimate holds:
	$$\frac{\min\{\lambda_{\min}(\Sigma_1),\lambda_{\min}(\Sigma_2)\}}{\min\{\gamma_2\lambda_{\min}(\Sigma_1),\gamma_1\lambda_{\min}(\Sigma_2) \}}\leq \RSE(\sD)^2\cdot {\rm REG}(\sD)\leq C \cdot \frac{\min\{\lambda_{\min}(\Sigma_1),\lambda_{\min}(\Sigma_2)\}}{\min\{\gamma_2\lambda_{\min}(\Sigma_1),\gamma_1\lambda_{\min}(\Sigma_2) \}}$$
	where $\gamma_i:=\left\langle\Sigma_i\frac{\beta_{i\star}}{\|\beta_{i\star}\|_2},\frac{\beta_{i\star}}{\|\beta_{i\star}\|_2} \right\rangle$ for $i=1,2$ and $C:= \frac{\kappa(\Sigma_{1})\kappa(\Sigma_{2})+1}{2}$.
\end{THM}
Similar to phase retrieval, the scaling between RSE and REG for bilinear sensing depends on the simultaneous alignment between $\beta_{i\star}$ and the minimal eigenspace of $\Sigma_{i}$ for $i = 1, 2$. In particular, when $\kappa(\Sigma_{1})\kappa(\Sigma_{2}) \approx 1$ the upper and lower bounds are nearly equal and if, further, $\lambda_{\min}(\Sigma_1) \approx \lambda_{\min}(\Sigma_2),$ then there are two regimes: (\emph{i}) if $ \min\{\gamma_{1},\gamma_{2}\} \gg \lambda_{\min}(\Sigma_{1})$, then ${\rm REG}(\sD) \appropto  {\rm RSE}(\sD)^{-2}$, and (\emph{ii}) if $ \min\{\gamma_{1},\gamma_{2}\} \approx \lambda_{\min}(\Sigma_{1})$, then ${\rm REG}(\sD)  \appropto {\rm RSE}(\sD)^{-4}.$ \revised{It is unclear whether the dependency on the condition numbers of the second-moment matrices appearing in $C$ is tight; we leave this question for future work.}

%=================================================================
% \section{Linear regression over fixed-rank matrices}\label{sections:lin-reg-fixed-rank}
\subsection{Matrix completion}\label{sec:matrix-completion}
The problem of matrix completion corresponds to the measurement model \eqref{eq:matrix-labels} where the ground truth matrix $\Ms$ has low rank and the feature data matrices $\bX\sim \sD_{X}$ are drawn from a certain discrete distribution. We will focus on the simplified setting where $M_{\star}$ is rank-one and positive semidefinite. The standard way to write this problem as stochastic optimization is
\begin{equation} \label{eq:matrix-completion}
	\min_{M \in \cM} f(M) := \mathop{\EE}_{(X,y)\sim\sD}\tfrac{1}{2}\big(\langle X,M\rangle - y\big)^{2},
\end{equation}
where $\cM:=\{M\in\bm{S}_+^d \mid \rank(M)=1\}$ is the manifold of PSD rank-one $d\times d$ matrices. %We assume here that $\sD_{X}$ is a distribution on the symmetric matrices $\bS^d$, which is without loss of generality since the objective function in \eqref{eq:matrix-completion} is unchanged by replacing $X$ with its symmetric part $\tfrac{1}{2}(X+X^\top)$ due to the symmetry of each $M$ in $\cM$. 
In this section, we will compute the RSE for matrix completion problems of the form \eqref{eq:matrix-completion} in terms of the graph induced by the support of $\sD_{X}$.

We begin with a few observations. First, it is straightforward to verify $\nabla f(M_{\star})=0$ and 
\begin{equation}\label{eqn:hessian_rep}
\nabla^{2} f(M_{\star})[\Delta,\Delta]=\EE\langle X,\Delta\rangle^2\qquad \forall \Delta\in \R^{d\times d}.
\end{equation}
 In particular, the optimal Lagrange multipliers at $M_{\star}$ are zero. Forming the factorization $\Ms=\betas\betas^\top$, the tangent space of $\cM$  at $\Ms$ may be written as
$$\mathcal{T}=\{\betas v^{\top} +v \betas^{\top}\mid v\in \R^d\}.$$
Moreover, an elementary computation shows $\|\Delta\|_F^2/\|\betas\|^2_2 \|v\|^2_2 \in [2,4]$ for all $\Delta = \betas v^{\top} + v \betas^{\top}$ and nonzero $v\in\R^d$. In particular, $\Delta$ is zero if and only if $v$ is zero. Consequently, the set of ill-conditioned distributions for \eqref{eq:matrix-completion} takes the form:
\begin{equation}\label{eqn:exceptional_mc}
\cE=\bigcup_{v\in \R^d\setminus\{0\}}\big\{\mu\in \cP_2(\R^{d \times d}) \mid \supp(\mu) \subset (\beta_\star v^\top + v\betas^\top)^{\perp}\big\}.
\end{equation}
Estimating ${\rm REG}(\mu)$ at any $\mu\in \cP_2(\R^{d \times d})\setminus \cE$ is straightforward, and is the content of the following lemma; see Appendix~\ref{sec:proof_blah} for a proof.

\begin{lemma}\label{lem:blahblah_notneeded}
Consider any measure $\mu\in \cP_2(\R^{d \times d})\setminus \cE$ and define the matrices
$$\Phi_{\betas} := (I \otimes \betas) + (\betas \otimes I)\qquad  \textrm{and}\qquad \Sigma_{\mu} := \EE_{\mu}{\Vect}(X){\Vect}(X)^\top.$$ Then for the problem \eqref{eq:matrix-completion}, the following estimate holds: 
\begin{equation}\label{eqn:reg_modulus}
2\|\beta_{\star}\|^2_2 \leq {\rm REG}(\mu)\cdot\lambda_{\min}(\Phi_{\beta_{\star}}^\top\Sigma_{\mu}\Phi_{\beta_{\star}})\leq 4\|\beta_{\star}\|^2_2.
\end{equation}
\end{lemma}

Observe now that most measures $\mu\in\cP_2(\R^{d \times d})$ do not correspond to a matrix completion problem, for $\mu$ may not even be discrete. With this in mind, we now specify a set of admissible distributions $\Q$ that encodes matrix completion problems of the form \eqref{eq:matrix-completion}. We will represent each such distribution as a symmetric matrix as follows. In \eqref{eq:matrix-completion}, we take $X$ to be a discrete random matrix supported on the set of canonical basis matrices $e_ie_j^\top$ together with the zero matrix; further, since each $M\in\cM$ is symmetric, we suppose without loss of generality that the distribution of $X$ is symmetric in the sense that $\PP\{X = e_i e_j^\top\} = \PP\{X = e_j e_i^\top\}$ for all $i,j\in \lbr d \rbr$. Thus, the distribution of $X$ may be represented by the symmetric matrix of probabilities
\begin{equation*}
	\bP = (p_{ij}) \quad \text{where} \quad
	p_{ij} = {\PP}\big\{X = e_{i}e_{j}^\top\big\}.
\end{equation*}
We denote the set of all such matrices by
\begin{equation*}
	\cQ := \big\{ (p_{ij}) \in \bS^d \mid \textstyle\sum_{i, j} p_{ij} \leq 1 \text{ and } p_{ij} \geq 0 \text{ for all } i, j \in \lbr d \rbr \big\}.
\end{equation*}
\revised{Note that we allow for $\sum_{i,j} p_{ij} < 1$, meaning that the zero matrix may be sampled with positive probability, in which case no entries are observed.
For each $P =(p_{ij})$ in $\cQ$ , we let $\mu_{P} \in \cP_2(\R^{d \times d})$ denote the corresponding discrete distribution given by
$$\mu_{P}(0) = 1 - \sum_{i,j} p_{ij}\quad \text{ and } \quad \mu_{P}\big(e_ie_j^\top\big) = p_{ij}.$$}%\footnote{With a slight abuse of notation we use $Q$ to denote both the matrix and distribution over indices.}
Then the set of ill-conditioned distributions for matrix completion problems of the form \eqref{eq:matrix-completion} is
\begin{equation*}
	\cE^{\text{mc}} :=  \{\mu_{P} \mid P \in \cQ \} \cap \cE,
\end{equation*}
where $\cE$ is given by \eqref{eqn:exceptional_mc}.
We are now ready to study the Wasserstein-$2$ distance from $\mu_P$ to $\cE^{\textrm{mc}}$. 

We will see that the computation of ${\rm RSE}(\mu_{P})$ relies on the combinatorial structure of the matrix indices that are observed in the problem. To describe this structure, we will require some additional notation. To begin, we let $S^{2}(\lbr d \rbr)$ denote the set of symmetric subsets of $\lbr d \rbr\!\times\!\lbr d \rbr$, and for any $A\in S^{2}(\lbr d \rbr)$   we let $G_{A}$ denote the undirected graph $(\lbr d \rbr, E_A)$ where $E_A := \big\{ \{i,j\} \mid (i,j)\in A\big\}$ is the edge set induced by $A$. In particular, $E_{\supp(P)} = \big\{ \{i,j\} \mid p_{ij}>0 \big\}$ and hence the graph $G_{\supp(P)}$ encodes the observed indices in the matrix completion problem. Next, we let
\begin{equation}\label{eq:G-ast}  G^{*}_A = (V^{*}, E^{*}_A) \text{ be the subgraph of $G_A$ induced by }V^{*} := \supp(\betas),\end{equation}
meaning that $E^{*}_A$ consists of all the edges in $E_A$ between vertices in $V^{*}$, and we define
\begin{equation}\label{eq:sad-nodes}
	V^{0}_A := \big\{ i \in \lbr d \rbr\setminus V^{*}\mid \{i,j\} \notin E_A \text{ for all } j \in V^{*}\big\}.
\end{equation}
Thus, $V^{0}_A$ is comprised of the vertices $i$ with ${\beta_{\star}}_{i} = 0$ that are not adjacent in $G_A$ to any vertex $j$ for which ${\beta_{\star}}_{j} \neq 0$. 

With the preceding notation in hand, we may now describe precisely the combinatorial structure that characterizes well-posed instances of matrix completion.

\begin{assumption}\label{eq:mc-combinatorial-char}
	For a given $A\in S^{2}(\lbr d \rbr)$ with $G^{*}_A$ and $V^{0}_A$ as defined in \eqref{eq:G-ast} and \eqref{eq:sad-nodes}, respectively, assume that the graph $G_A$ is such that the following two conditions hold:
	\begin{enumerate}
		\item (\textbf{No bipartite components of $G_A^{*}$}) \label{it:1} Every connected component of $G^{*}_A$ is non-bipartite.
		\item (\textbf{No isolated zeros of $\betas$}) \label{it:2} The set $V^{0}_A$ is empty.
	\end{enumerate}
\end{assumption}

We will see that the set
$$\Omega_{\beta_{\star}} := \big\{ A \in S^{2}(\lbr d \rbr) \mid \text{$G_A$ does not satisfy Assumption~\ref{eq:mc-combinatorial-char}}\big\}$$
plays a crucial role in computing ${\rm RSE}(\mu_{P})$. We are now ready to state the main result of this section; see Appendix~\ref{sec:proof_rsemc} for a proof.
\begin{THM}[label={thm:rse_MC}]{(RSE for matrix completion)}{}
	Let $\bP = (p_{ij})$ and $M_{\star} = \betas\betas^{\top}$ be the data for a matrix completion problem of the form \eqref{eq:matrix-completion} with $X\sim \mu_{P}$. Then $\mu_{P}$ is well-posed, i.e., $\mu_{P} \notin \cE^{{\rm mc}}$, if and only if the graph $G_{\supp(P)}$ satisfies Assumption~\ref{eq:mc-combinatorial-char}, i.e., $\supp(P)\notin\Omega_{\betas}$. Additionally, the following identity holds:
	\begin{equation}
		\label{eq:mat-comp-hard-dist}
		{\rm RSE}(\mu_{P})^{2} = \min_{\substack{A \in \Omega_{\betas}\\ A\subset\supp(P)}}\sum_{(i,j)\in \supp(\bP) \setminus A} p_{ij}.
	\end{equation}
	Moreover, computing ${\rm RSE}(\mu_{P})$ is NP-hard in general.
\end{THM}

Thus, Theorem~\ref{thm:rse_MC} shows
\begin{equation}
	\cE^{\text{mc}} = \big\{\mu_{P} \mid P\in\cQ~\text{and}~{\supp}(P) \in \Omega_{\betas}\big\}
\end{equation}
and that one can compute ${\rm RSE}(\mu_{P})$ by enumerating over all symmetric subsets $A\subset \supp(\bP)$ such that either $G^*_A$ has a bipartite connected component or  $V^0_{A}$ is nonempty.
Then any such subset $A$ for which the mass $\sum_{(i, j) \in \supp(\bP) \setminus A} p_{ij}$ is smallest yields ${\rm RSE}(\mu_{P})^2$. Interestingly, computing  ${\rm RSE}(\mu_{P})$ is NP-hard in general, which follows from a polynomial-time reduction from the \texttt{MaxCut} problem.

\section{Conclusion}\label{sec:conclusions}
In this work, we introduced a new measure of robustness---{\em radius of statistical efficiency (RSE)}---for problems of statistical inference and estimation. We computed RSE for a number of testbed problems, including principal component analysis, generalized linear models, phase retrieval, bilinear sensing, and matrix completion. In all cases, we verified a precise reciprocal relationship between RSE and the intrinsic complexity/sensitivity of the problem instance, thereby paralleling the classical Eckart–Young theorem and its numerous extensions in numerical analysis and optimization. More generally, we obtained sufficient conditions for such a relationship to hold that depend only on local information (gradients, Hessians), rather than an explicit description of the set of ill-conditioned distributions.  We believe that this work provides an intriguing new perspective on the interplay between problem difficulty, solution sensitivity, and robustness in statistical inference and learning.

\subsection*{Acknowledgments}
We thank John Duchi, Jorge Garza-Vargas, Zaid Harchaoui, and Eitan Levin for insightful conversations during the development of this work. We also thank the anonymous reviewers for their useful feedback, especially for their improvement to the proof of Lemma~\ref{lem:basic_scie}.

%=================================================================
% \newpage

% \newpage
\bibliographystyle{abbrvnat}
\bibliography{biblio}
\appendix
%\section{Supplementary results}
%\label{section:supplement}

\section{Geometry of the Wasserstein space and distance estimation}\label{sec:wasserstein_geometry_review}
In this section, we introduce the necessary background of the Wasserstein geometry  and prove a number of results that may be of independent interest; in the subsequent section, we will use many of these results to prove estimates on RSE announced in the paper. We follow standard notation of optimal transport, as set out for example in the monographs of Villani \cite{Villani2009} and Santambrogio \cite{santambrogio2015optimal}.

Let $(\cX,\sd)$ be a  separable complete metric space equipped with its Borel $\sigma$-algebra $\cB_{\cX}$. \revised{The primary example for us will be a Euclidean space $(\mathbf{E},\|\cdot\|)$, specifically $\R^d$ or $\R^{m\times n}$.} The distance of a point $x\in \cX$ to a set $\cK\subset\cX$ will be denoted by $\dist(x,\cK)=\inf_{x'\in \cK} \sd(x,x')$.
The set of Borel probability measures on $\cX$ is denoted by $\cP(\cX)$, and will be abbreviated as $\cP$ if the space $\cX$ is clear from context. The support of a measure $\mu\in\cP$, written as $\Supp(\mu)$, is the smallest closed set $C\subset\cX$ such that the complement $\cX\setminus C$ is of zero $\mu$-measure. 
For any measurable map $T\colon(\Omega,\cF,\PP)\rightarrow(\cX,\cB_{\cX})$, the pushforward measure  $T_{\#}\PP$ is defined by setting
$$(T_{\#}\PP)(B)=\PP(T^{-1}(B))\qquad \text{for all } B\in \cB_{\cX}.$$ The support of the random variable $T$ in $\cX$ is the support of its distribution $T_{\#}\PP$.

 For any $p\geq 1$, the symbol $\cP_p$ denotes the set of all distributions $\mu$ on $\cX$ with finite $p^{\text{th}}$ moment, meaning $\EE_{x\sim\mu} \sd(x,x_0)^p < \infty$ for some (and hence any) $x_0\in \cX$.    
 The Wasserstein-$p$ distance between two measures $\mu,\nu\in \cP_p$ is given by 
\begin{equation*}
	W_p(\mu, \nu) =\min_{\pi \in \Pi(\mu,\nu)}\bigg( \mathop{\EE}_{(x,y)\sim \pi}\sd(x,y)^p \bigg)^{1/p}.
\end{equation*}
Here, the set $\Pi(\mu, \nu)$ is comprised of the couplings between $\mu$ and $\nu$, i.e., distributions on $\cX \times \cX$ having $\mu$ and $\nu$ as its first and second marginals. An important fact is that the pair $(\cP_p,W_p)$ is a separable complete metric space in its own right and is called the Wasserstein-$p$ space on $\cX$.

We will need a few basic estimates on the $W_p$ distance. First, consider any measures $\mu,\nu\in\cP_p$ and a measurable map $T\colon\cX\to\cX$ satisfying $\nu = T_{\#}\mu$. Then the law of the random variable $(x,T(x))$, with $x\sim\mu$, is a coupling between $\mu$ and $\nu$ and therefore the estimate holds:
\begin{equation}\label{eq:detcouplingupperbound}
W_p^p(\mu, \nu)\leq \mathop{\EE}_{x\sim \mu} \sd(x, T(x))^{p}.
\end{equation}
Another useful observation is that for any measures $\mu,\nu\in\cP_p$ such that the support of $\nu$ is contained in the closure of a set $\cK\subset\cX$, the estimate holds:
\begin{equation}\label{eqn:reverse_ineq}
W_{p}^p(\mu, \nu)\geq \mathop{\EE}_{x\sim \mu}\dist(x,\cK)^p.
\end{equation}
Consequently, equality holds in \eqref{eq:detcouplingupperbound} when $T$ is a projection---the content of the following lemma.

\begin{lemma}[$W_p$ metric \&  projections]\label{lem:w2_proj}
Consider a measure $\mu\in\cP_p$ and a set $\cK\subset\cX$. Suppose that the metric projection $\sP_{\cK}$ admits a measurable selection $s\colon\cX\rightarrow\cK$. Then equality holds:
\begin{equation}\label{eq:detcouplingequality}
	W_{p}^p(\mu, s_{\#}\mu) = \mathop{\EE}_{x\sim \mu}\dist(x, \cK)^{p}.
\end{equation}
\end{lemma}
\begin{proof}
We first show that $\nu:=s_{\#}\mu$ lies in $\cP_p$. Indeed, for any $x_0\in \cK$ and $x\in\cX$ we have
$$\sd(s(x), x_0)\leq \sd(s(x), x)+\sd(x, x_0)\leq  2\sd(x, x_0),$$
and therefore $\nu$ lies in $\cP_p$. Next, taking into account that the support of $\nu$ is contained in the closure of $\cK$, combining \eqref{eq:detcouplingupperbound} and \eqref{eqn:reverse_ineq} completes the proof.
 \end{proof}

For the rest of the section, we will focus exclusively on the setting where $\cX$ is the standard $d$-dimensional Euclidean space $\R^d$ equipped with the dot product $\langle x, y \rangle = x^{\top}y$ and the induced norm $\|x\|_2=\sqrt{\langle x,x\rangle}$. We denote the second-moment matrix of any measure $\mu\in \cP_2$ by the symbol 
$$\Sigma_{\mu}:=\mathop{\EE}_{x\sim \mu} xx^\top.$$
Given a set of measures $\cV\subset\cP_2$, the distance function $W_2(\cdot, \mathcal{V})$ is defined in the usual way:
$$W_2(\mu, \mathcal{V})=\inf_{\nu\in \mathcal{V}}\, W_2(\mu,\nu).$$
  %The results in both cases follow from the following reduction. Given two positive semidefinite matrices $A$ and $B$ the Bures-Wasserstein distance between them is given by
  %$$W^{2}_{2}(A, B) = \Tr(A + B) - 2\Tr(A^{1/2}BA^{1/2})^{1/2}.$$

Although estimating the distance function is difficult in general, we will focus on well-structured sets $\mathcal{V}$ for which $W_2(\mu, \mathcal{V})$ can be readily computed. The following two sections study, respectively, measures constrained either by the location of its support or by the spectrum of its covariance.
\subsection{Sets of measures constrained by their support}
Consider a linear subspace $\cK \subset \cX$ and a measure $\mu\in\cP_2$. Observe that the compression $\Sigma_{\mu}\vert_{\cK}$ of $\Sigma_{\mu}$ to $\cK$ is the positive semidefinite operator on $\cK$ with  quadratic form given by
\begin{equation}\label{eqn:basic_eqn_cov}
\langle \Sigma_{\mu} v, v \rangle = \mathop{\EE}_{x\sim\mu}\langle x, v \rangle^2\qquad \forall v\in\cK.
\end{equation}
We will need the following lemma that provides a convenient interpretation of the trace of $\Sigma_{\mu}\vert_{\cK}$. 

	\begin{lemma}[Trace of second moment $\&$ $W_2$ distance]\label{lem:dist-trace}
		Consider a measure $\mu\in\cP_2$  and let $\cK$ be a linear subspace of $\R^d$. Then	 the following equalities hold:
$$			{\Tr}\,\Sigma_{\mu}\vert_{\cK} = \mathop{\EE}_{\x\sim \mu}\dist(\x,\cK^{\perp})^2 = W_2^2(\mu, (\sP_{\cK^{\perp}})_{\#}\mu).$$
	\end{lemma}
	\begin{proof}
		We successively compute
		\begin{align*}
			{\Tr}\,\Sigma_{\mu}\vert_{\cK} = \mathop{\EE}_{x\sim \mu} {\Tr}(\sP_{\cK} xx^{\top} \sP_{\cK}) & = \mathop{\EE}_{x\sim \mu}\|\sP_{\cK}x\|^2_2
			= \mathop{\EE}_{x\sim \mu}\dist(x,\cK^{\perp})^2
			= W_2^2(\mu, (\sP_{\cK^{\perp}})_{\#}\mu),
		\end{align*}
		where the last equality follows from Lemma~\ref{lem:w2_proj}. 
	\end{proof}

In light of \eqref{eqn:basic_eqn_cov}, the matrix $\Sigma_{\mu}\vert_{\cK}$ is singular if and only if the inclusion $\Supp(\mu)\subset v^{\perp}$ holds for some $v\in\cK\cap\mathbb{S}^{d-1}$. Define the set of measures $\nu\in\cP_2$ for which $\Sigma_{\nu}\vert_{\cK}$ is singular:
\begin{equation}\label{eq:singular_set}
	\cV_{\cK} 	  := \big\{\nu\in\mathcal{P}_2 \mid \Supp(\nu)\subset v^{\perp}\text{~for some~} v\in\cK\cap\mathbb{S}^{d-1} \big\}. \end{equation}
The following theorem---the main result of this section---shows that the squared $W_2$ distance of $\mu$ to $\cV_{\cK}$ is simply the minimal eigenvalue of $\Sigma_{\mu}\vert_{\cK}$.

\begin{THM}[label={prop:Sdist}]{(Distance to $\cV_{\cK}$)}{}
	Consider a measure $\mu\in \cP_2$ and a nontrivial linear subspace $\cK$ of $\R^d$. Then equality holds:
	\begin{equation*}
		\label{eq:1}
		 W_2^2(\mu, \cV_{\cK}) = \lambda_{\min}(\Sigma_{\mu}\vert_{\cK}),
	\end{equation*}
	where $\cV_{\cK}$ is defined in \eqref{eq:singular_set}.
	 Moreover, the $W_2$ distance from $\mu$ to $\cV_{\cK}$ is attained by the measure $(\sP_{v^{\perp}})_{\#}\mu$, where  $v\in \cK$ is any eigenvector of $\Sigma_{\mu}\vert_{\cK}$ corresponding to its minimal eigenvalue.
\end{THM}

	\begin{proof}
For any vector $v\in \cK\cap \mathbb{S}^{d-1}$, applying Lemma~\ref{lem:dist-trace} with ${\rm span}(v)$ in place of $\cK$ yields
\begin{equation}\label{eqn:extress_eig}
\langle \Sigma_{\mu}v, v \rangle={\Tr}\,\Sigma_{\mu}\vert_{{\rm span}(v)} = W_2^2(\mu, (\sP_{v^{\perp}})_{\#}\mu).
\end{equation}
Let us now decompose $\cV_{\cK}$ into a union of simpler sets:
$$\cV_{\cK}=\bigcup_{v\in \cK\cap \mathbb{S}^{d-1}} L_{v}\qquad \textrm{where}\qquad L_v:=\{\nu\in\cP_2 \mid \supp(\nu)\subset v^{\perp}\}.$$
We now estimate
$$W^2_2(\mu,L_v)\leq W^2_2(\mu, (\sP_{\bv^{\perp}})_{\#}\mu)=\mathop{\EE}_{\x\sim \mu}\dist(x,v^{\perp})^2\leq W^2_2(\mu,L_v),$$
where the first inequality holds trivially, the equality follows from Lemma~\ref{lem:w2_proj}, and 
the last inequality follows from \eqref{eqn:reverse_ineq} with $v^\perp$ in place of $\cK$. Thus equality holds throughout. 
Using \eqref{eqn:extress_eig}, we then conclude
$$W^2_2(\mu, \cV_{\cK})=\inf_{v\in \cK\cap \mathbb{S}^{d-1}} W^2_2(\mu,L_v)=\inf_{v\in \cK\cap \mathbb{S}^{d-1}}\langle \Sigma_{\mu}v, v \rangle= \lambda_{\min}(\Sigma_{\mu}\vert_{\cK}),$$
as claimed. It also follows immediately that for any minimal eigenvector $\bv$ of $\Sigma_{\mu}\vert_{\cK}$, the pushforward measure $(\sP_{\bv^{\perp}})_{\#}\mu$ attains the $W_2$ distance from $\mu$ to $\cV_{\cK}$.
\end{proof}

\subsection{Spectral sets and functions of measures}\label{sec:spec_funct}
In this section, we investigate a special class of functions on $\cP_2^{\circ}$---called spectral---that depend on the measure only through the eigenvalues of its covariance matrix. This function class has close  analogues in existing literature on matrix analysis and eigenvalue optimization. We postpone a detailed discussion of the related literature until the end of the section.

%Let $\lambda\colon \bm{S}^d \to \mathbf{R}^d$ denote the eigenvalue map given by $\lambda(A) = (\lambda_1(A), \ldots, \lambda_d(A))$. 
For any matrix $A\in \bm{S}^d$, let $\lambda(A) = (\lambda_1(A), \ldots, \lambda_d(A))$ denote the vector of eigenvalues of $A$ arranged in nonincreasing order. The following is the key definition.

\begin{definition}[Spectral functions of measures]
A function $F\colon\cP_2^{\circ} \to \R\cup\{+\infty\}$ is called {\em spectral} if for any $\mu,\nu\in\cP_2^{\circ}$ satisfying $\lambda(\Sigma_{\mu})=\lambda(\Sigma_{\nu})$, the equality $F(\mu)=F(\nu)$ holds.
\end{definition}

A good example to keep in mind is the Schatten $q$-norm $F(\mu)=\|\lambda(\Sigma_{\mu})\|_{q}$ for any $q\in [1,\infty]$ (where $\|\cdot\|_q$ denotes the $\ell_q$ norm on $\R^d$). Notice that in this example $F$ factors as a composition of the eigenvalue map $\lambda(\cdot)$ and the permutation-invariant function $f(v)=\|v\|_q$ on $\R^d$. Evidently, all spectral functions arise in this way. 

\begin{definition}[Symmetric functions]
A function $f\colon\R^d_+\to \R\cup\{+\infty\}$ is called {\em symmetric} if the equality $f(s(x))=f(x)$ holds for all $x\in\R^d_+ $ and all permutations of coordinates $s(\cdot)$.
\end{definition}

An elementary observation is that a function $F\colon\cP_2^{\circ} \to \R\cup\{+\infty\}$ is spectral if and only if there exists a symmetric function $f:\R^d_+\to \R\cup\{+\infty\}$ satisfying
$$F(\mu)=f(\lambda(\Sigma_{\mu}))\qquad \forall \mu\in\cP_2^{\circ}.$$
Concretely, the symmetric function $f$ can be obtained from $F$ by restricting to Gaussian measures with diagonal covariance, i.e., $f(v)=F(\mathsf{N}(0,\Diag(v)))$. 
The definition of spectral and symmetric functions easily extends to sets through indicator functions. Namely, a set $\cG\subset\cP_2^{\circ}$ is  {\em spectral} if the indicator function $\delta_{\cG}$ is spectral, while a set $G\subset\R^d_+$ is {\em symmetric} if the indicator  $\delta_{G}$ is symmetric.

An interesting example of a spectral set is given by 
\begin{equation}\label{eqn:int_set}
\cG_q:=\{\mu\in\cP_2^\circ \mid \lambda_q(\Sigma_{\mu})= \lambda_{q+1}(\Sigma_{\mu})\}.
\end{equation}
Although complicated, we may represent $\cG_q$ with the symmetric set 
$$G_q:=\{v\in \R^d_+ \mid v^{(q)}=v^{(q+1)}\},$$
where $v^{(i)}$ is the $i$'th largest coordinate of $v$. Indeed, equality 
$\cG_q=\{\mu\in \cP_2^\circ \mid \lambda(\Sigma_{\mu})\in G_q\}$ holds.

In this section, we will show that for any symmetric set $G\subset\R^d_+$, one can express the $W_2$ distance to the corresponding spectral set $\cG=\{\nu\in \cP_2^\circ \mid \lambda(\Sigma_{\nu})\in G\}$ purely in terms of the Hellinger distance to the much simpler set $G$. Indeed, we will prove the following theorem, which specialized to example \eqref{eqn:int_set} yields the expression
$$W_2(\mu,\cG_q)={\dist}_2\Big(\sqrt{\lambda(\Sigma_{\mu})},G_q\Big)=\tfrac{1}{\sqrt{2}}\Big(\sqrt{\lambda_{q}}-\sqrt{\lambda_{q+1}}\Big).$$

\begin{THM}[label={thm:distance_est_spectral}]{(Distance to spectral sets in $\mathcal{P}_2^{\circ}$)}{}
Let $G\subset\R^d_+$ be a symmetric set, and consider the corresponding spectral set of measures
$$\cG:=\{\nu\in \cP_2^\circ \mid \lambda(\Sigma_{\nu})\in G\}.$$
Then for any $\mu\in \mathcal{P}_2^\circ$, equality holds: 
$$W_2(\mu,\cG)=\min_{v\in G}\Big\|\sqrt{\lambda(\Sigma_{\mu})}-\sqrt{v}\Big\|_2.$$
\end{THM}

In fact, we will prove a more general statement that applies to functions, with the distance replaced by the so-called Moreau envelope. We need some further notation to proceed. Let $(\cY,\sd)$ be a metric space and consider a function $f\colon\cY\to\R\cup\{+\infty\}$. Then for any parameter $\rho>0$, the {\em Moreau envelope} and the {\em proximal map} of $f$ \cite{moreau1965proximite,de1993new}, respectively, are defined as:
\begin{align*}
	f_{\rho}(y)&:=\inf_{y'\in \cY}~f(y')+\tfrac{1}{2\rho}\sd^2(y,y'),\\
	\prox_{\rho f}(y)&:=\argmin_{y'\in \cY}~f(y')+\tfrac{1}{2\rho}\sd^2(y,y').
\end{align*}
In particular, if $f$ is an indicator function of a set $Q$, then $f_{\rho}$ is proportional to the squared distance function to $Q$, while $\prox_{\rho f}$ reduces to the nearest-point projection onto $Q$.

We will be interested in three metric spaces and it is important to keep the metric in mind in all results that follow.
\begin{itemize}
\item  $(\cP_2^{\circ},W_2)$ The space $\mathcal{P}_2^{\circ}(\R^d)$ equipped with the  Wasserstein-2 distance $$W_2^2(\mu,\nu) =\min_{\pi \in \Pi(\mu,\nu)}\mathop{\EE}_{(x,y)\sim \pi}\|x-y\|_2^2.$$
\item $(\bS^d_{+},W_2)$ The cone of PSD matrices $\bS^d_{+}$ equipped with the Bures-Wasserstein distance
$$W^2_2(A,B)={\tr}\big(A + B - 2(A^{1/2}BA^{1/2})^{1/2}\big).$$
\item $(\R^d_+, W_2)$ The cone of nonnegative vectors $\R^d_+$ equipped with the Hellinger distance
$$W_2(x,y)=\|\sqrt{x}-\sqrt{y}\|_2,$$
where the square root is applied componentwise. 
\end{itemize}
Notice that we are abusing notation by using the same symbol $W_2$ to denote the metric in all three spaces. The reason we are justified in doing so is that the three metric spaces are related by isometric embedding. Namely, as shown in \cite[Section~1.6.3]{Panaretos2020Wasserstein}, the Wasserstein-2 distance between two centered Gaussian distributions $\mu= \mathsf{N}(0,\Sigma_\mu)$ and $\nu= \mathsf{N}(0,\Sigma_\nu)$ coincides with the Bures-Wasserstein distance between their covariance matrices:
$$W_2(\mu,\nu)=W_2(\Sigma_{\mu}, \Sigma_{\nu}).$$
Similarly, the Bures-Wasserstein metric restricted to diagonal PSD matrices is the Hellinger distance:
$$W_2(\Diag(x), \Diag(y))=W_2(x,y).$$ 

The following is the main result of this section.

\begin{THM}[label={thm:distance_est_spectral_moreau}]{(Diagonal reduction)}{}
	Let $f\colon\R^d_+\to \R\cup\{+\infty\}$ be a symmetric function and ${F\colon \cP_2^{\circ} \to \R\cup\{+\infty\}}$ be the spectral function given by
	$F(\mu)=f(\lambda(\Sigma_{\mu})).$ Then for any $\rho>0$ and $\mu\in \mathcal{P}_2^\circ$, equality holds: 
	$$F_{\rho}(\mu)=f_{\rho}(\lambda(\Sigma_{\mu})).$$
\end{THM}

Theorem~\ref{thm:distance_est_spectral} follows immediately by applying  Theorem~\ref{thm:distance_est_spectral_moreau} to indicator functions $\delta_{\cG}$ and $\delta_G$. 
The rest of the section is devoted to proving Theorem~\ref{thm:distance_est_spectral_moreau}.

\subsubsection{Proof of Theorem~\ref{thm:distance_est_spectral_moreau}}
We begin with some notation. The symbol $O(d)$ will denote the set of real orthogonal $d\times d$ matrices. For any matrix $A\in \R^{m\times n}$, let $\sigma(A) = (\sigma_1(A), \ldots, \sigma_{m\wedge n}(A))$ denote the vector of singular values of $A$ arranged in nonincreasing order. %The singular values for any matrix $A\in \R^{m\times n}$ are arranged in nonincreasing order: $$\sigma_1(A)\geq \sigma_2(A)\geq \cdots \geq \sigma_{m\wedge n}(A).$$
We say that two matrices $A,B\in\R^{m\times n}$ admit a {\em simultaneous ordered singular value decomposition (SVD)} if there exist matrices $U\in O(m),~ V\in O(n)$ satisfying $U^\top A V=\Diag(\sigma(A))$ and $U^\top B V=\Diag(\sigma(B))$. 
The following trace inequality, essentially due to \cite{theobald1975inequality,von1962some}, will play a central role in the section. The result as stated, along with a proof, may be found in \cite[Theorem 4.6]{lewis2005nonsmooth}.

\begin{lemma}[von Neumann-Theobald]\label{thm:trace_ineq}
Any two matrices $A,B\in \R^{m\times n}$ satisfy 
$$ \langle \sigma(A),\sigma(B)\rangle\geq \langle A,B\rangle.$$
Moreover, equality holds if and only if $A$ and $B$ admit a simultaneous ordered SVD.
\end{lemma}

The following lemma shows that the Bures-Wasserstein distance can be written in terms of a Procrustes problem \cite[Theorem 1]{bhatia2019bures}. 
\begin{lemma}[Procrustes distance]\label{lem:procrustes} For any two matrices $A,B\in \bS^d_+$, equality holds:
$$W_2(A,B)=\min_{U\in O(d)} \|A^{1/2}-B^{1/2}U\|_F.$$
\end{lemma}

We will also need the following variational form of the Bures-Wasserstein distance; see for example  \cite[Section 3]{bhatia2019bures}.
\begin{lemma}[Variational form]\label{lem:var_form} For any two matrices $A,B\in \bS^d_+$, equality holds:
$$W_2^2(A,B)=\min_{\substack{x,y:\, \EE[xx^\top]=A,\, \EE[yy^\top]=B}} \,\EE\|x-y\|^2_2.$$
\end{lemma}

With these results in place, we are ready to start proving  Theorem~\ref{thm:distance_est_spectral_moreau}. To this end, we will first establish the theorem in the Gaussian setting and then deduce the general case by a reduction. As the first step, we will compute the $W_2$ distance of a matrix to an orbit of $B$ under conjugation: 
$$\mathcal{O}(B):=\{VBV^\top \mid V\in O(d)\}.$$ 

\begin{lemma}[Distance to orbit]\label{lem:gauss}
For any two matrices $A,B \in \bS^d_+$, we have
\begin{equation}\label{eqn:thing_we_want}
W_2(A,\mathcal{O}(B))=\Big\|\sqrt{\lambda(A)}-\sqrt{\lambda(B)}\Big\|_2 = W_2(\lambda(A),\lambda(B)).
\end{equation}
Moreover, the set of nearest points of $\mathcal{O}(B)$ to $A$ is given by 
\begin{equation}\label{eqn:projection_formula}
\{U\Diag(\lambda(B))U^\top \mid A=U\Diag(\lambda(A))U^\top,~ U\in O(d)\}.
\end{equation}
\end{lemma}
\begin{proof}
To see the inequality $\leq$ in \eqref{eqn:thing_we_want}, consider an eigenvalue decomposition 
$A=U\Diag(\lambda(A))U^\top$ for some $U\in O(d)$. Then we have 
$$W_2(A,\mathcal{O}(B))\leq W_2(U\Diag(\lambda(A))U^\top,U\Diag(\lambda(B))U^\top)= W_2(\lambda(A),\lambda(B)).$$
Next, we show the reverse inequality. To this end, set $\bar A:=A^{1/2}$ and $\bar B:=B^{1/2}$. Then for any $V\in O(d)$, we successively compute  
\begin{align}
W_2^2(A,VBV^\top)&= \inf_{U\in O(d)}\|\bar A-(V\bar BV^\top) U\|_F^2\label{eqn:pro}\\
&=\|\bar A\|^2_F+\|\bar B\|^2_F-2\sup_{U\in O(d)} \langle \bar A,V\bar B (V^\top U)\rangle\label{eqn:expand_square}\\
&=\|\bar A\|^2_F+\|\bar B\|^2_F-2\sup_{Z\in O(d)} \langle \bar A,V\bar B Z^\top\rangle\label{eqn:expand_square2}\\
&\geq \|\bar A\|^2_F+\|\bar B\|^2_F-2 \langle \sigma(\bar A),\sigma(\bar B)\rangle\label{eqn:von_neumann}\\
&=\|\lambda(\bar A)\|^2_2+\|\lambda(\bar B)\|^2_2-2 \langle \lambda(\bar A),\lambda(\bar B)\rangle\label{eqn:eig_sing}\\
&=\|\lambda(\bar A)-\lambda(\bar B)\|^2_2 \label{eqn:unneeded}\\
&=\Big\|\sqrt{\lambda(A)}-\sqrt{\lambda(B)}\Big\|^2_2,\label{eqn:unneeded2}
\end{align}
where \eqref{eqn:pro} follows from Lemma~\ref{lem:procrustes}, the equality \eqref{eqn:expand_square} follows from expanding the Frobenius norm, \eqref{eqn:expand_square2} uses the variable substitution $Z=V^\top U$, the inequality \eqref{eqn:von_neumann} follows from von Neumann's trace inequality (Lemma~\ref{thm:trace_ineq}), and \eqref{eqn:eig_sing} uses the fact that eigenvalues and singular values coincide for PSD matrices. Taking the infimum over $V\in O(d)$ shows the claimed inequality $\geq$ in \eqref{eqn:thing_we_want}.

Next, the fact that any matrix in the set \eqref{eqn:projection_formula} is a nearest point of $\mathcal{O}(B)$ to $A$ follows directly from the expression \eqref{eqn:thing_we_want}. To see the converse, let  $V\in O(d)$ and observe from \eqref{eqn:pro}--\eqref{eqn:unneeded2} that $VBV^\top$ is a nearest point of $\mathcal{O}(B)$ to $A$ if and only if \eqref{eqn:von_neumann} holds with equality; applying Lemma~\ref{thm:trace_ineq}, we see that equality holds if and only if there exist matrices $M_1,M_2, Z\in O(d)$ such that
\begin{equation}\label{eqn:blah_blha}
\bar A=M_1\Diag(\lambda(\bar A))M_2^\top\qquad \textrm{and}\qquad V\bar B Z^\top= M_1\Diag(\lambda(\bar B))M_2^\top.
\end{equation}
In particular, right-multiplying each equation by its transpose yields the expressions
$$A=M_1\Diag(\lambda( A))M_1^\top\qquad\textrm{and}\qquad V BV^\top =M_1\Diag(\lambda(B)) M_1^\top,$$
thereby completing the proof.
\end{proof}

We are now ready to complete the proof of Theorem~\ref{thm:distance_est_spectral_moreau} in the Gaussian setting. In the proof, we will use the basic fact that for any vectors $v,w\in \R^d$, we have
\begin{equation}\label{eqn:nonincreaing_basic}
\|v^{\uparrow}-w^{\uparrow}\|_2\leq \|v-w\|_2,
\end{equation} 
where $v^{\uparrow}$ and $w^{\uparrow}$ are the vectors obtained by permuting the coordinates of $v$ and $w$ to be nonincreasing.

\begin{theorem}[Envelope and prox-map on Gaussian space] 	\label{thm:gauss_space}
	Consider a symmetric function $f\colon\R^d_+\to \R\cup\{+\infty\}$ and define the function on PSD matrices $F\colon\bS^d_+\to  \R\cup\{+\infty\}$ by setting
	$F(A)=f(\lambda(A)).$
	Then, for any $\rho>0$ and $A \in \bS^d_+$, the following equalities hold:
	$$F_{\rho}(A)=f_{\rho}(\lambda(A))$$
	and
	$$\prox_{\rho F}(A)=\{U\Diag(v)U^\top \mid v\in \prox_{\rho f}(\lambda(A)),~ A=U\Diag(\lambda(A))U^\top,~ U\in O(d)\}.$$
\end{theorem}
\begin{proof}
Given $\rho>0$ and $A \in \bS^d_+$, we successively compute
\begin{align}
	F_{\rho}(A)&=\inf_{B\succeq 0}~ F(B)+\tfrac{1}{2\rho} W^2_2(A,B)\\
	&=\inf_{v\geq 0}\,\inf_{B\in \mathcal{O}(\Diag(v))}~ f(\lambda(B))+\tfrac{1}{2\rho} W^2_2(A,B)\\
	&=\inf_{v\geq 0}~ f(v)+\tfrac{1}{2\rho} \inf_{B\in \mathcal{O}(\Diag(v))} W^2_2(A,B)\\
	&=\inf_{v\geq 0}~ f(v)+\tfrac{1}{2\rho}  W^2_2(A,\mathcal{O}(\Diag(v)))\\
	&=\inf_{v\geq 0}~ f(v)+\tfrac{1}{2\rho} \Big\|\sqrt{\lambda(A)}-\sqrt{v^{\uparrow}}\Big\|^2_2\label{eqn:orbit_dist}\\
	&=\inf_{v\geq 0}~ f(v)+\tfrac{1}{2\rho} W^2_2(\lambda(A),v)\label{eqn:reordering_works}\\
	&=f_{\rho}(\lambda(A)),
\end{align}
where \eqref{eqn:orbit_dist} follows from Lemma~\ref{lem:gauss}  and \eqref{eqn:reordering_works} follows from \eqref{eqn:nonincreaing_basic} and the symmetry of $f$. Next, observe from the chain of equalities above that $B=U\Diag(v)U^\top$ lies in $\prox_{\rho F}(A)$ for some $U\in O(d)$ if, and only if,
$v$ lies in $\prox_{\rho f}(\lambda(A))$ and the equality 
$W_2(A,B)=W_2(A,\mathcal{O}(\Diag(v)))$ holds; appealing to Lemma~\ref{lem:gauss}, this last equality holds if and only if we may write $A=U\Diag(\lambda(A))U^\top$.
Thus the proof is complete.
\end{proof}

We now move on to establishing Theorem~\ref{thm:distance_est_spectral_moreau} in full generality, i.e., outside the Gaussian setting. To begin, we establish an analogue of Lemma~\ref{lem:gauss}.

\begin{lemma}\label{lem:basic_scie}
  Let $\cC \subset \bS^{d}_{+}$ be a set of positive semidefinite matrices, and consider the set $\cP_{\cC} := \{\nu \in \cP_{2}^{\circ} \mid \Sigma_{\nu} \in \cC\}.$
 Then for any $\mu \in \mathcal{P}^{\circ}_2$, the following equality holds:
			\begin{equation}
              \label{eqn:basic_science}
              W_2(\mu, \mathcal{P}_{\mathcal{C}}) = W_2(\Sigma_{\mu},\cC).
			\end{equation}
          \end{lemma}
          \begin{proof} 
Note first that for any $\mu \in \mathcal{P}_2$, we have 
\begin{align*}
W^2_2(\mu, \mathcal{P}_{\cC}) = \inf_{\pi \in \Pi(\mu,\nu),\, \nu\in \mathcal{P}_{\cC}} \mathop{\EE}_{(x,y)\sim\pi}\|x-y\|^2_2 &\geq \inf_{x,y:\,\EE[xx^\top]=\Sigma_{\mu}, \,\EE[yy^\top]\in \cC} \EE \|x-y\|^2_2 \\&= W_2^{2}(\Sigma_{\mu},\cC),
\end{align*}
where the last equality follows from  Lemma~\ref{lem:var_form}. This establishes the inequality $\geq$ in \eqref{eqn:basic_science}. To see the reverse inequality, let $\mu \in \mathcal{P}^{\circ}_2$ and $\Sigma\in\cC$. Suppose first that $\Sigma_{\mu}$ is invertible. In this case, the optimal transport map from $\mathsf{N}(0,\Sigma_\mu)$ to $\mathsf{N}(0, \Sigma)$ is given by a linear function $T$, as in \cite[Section~1.6.3]{Panaretos2020Wasserstein}. Consider the pushforward measure $\bar{\nu} := T_{\#}\mu$. Since the integral of a quadratic function depends on a probability measure only through its first two moments, it follows immediately from the linearity of $T$ that the first two moments of $T_{\#}\mu = \bar{\nu}$ and $T_{\#}\mathsf{N}(0,\Sigma_\mu) = \mathsf{N}(0, \Sigma) $ coincide, so $\bar{\nu}\in\cP_\cC$. Likewise, ${\EE}_{\mu}\|x-Tx\|^2_2={\EE}_{\mathsf{N}(0,\Sigma_\mu)}\|x-Tx\|^2_2$, so we conclude
$$
W^2_2(\mu, \mathcal{P}_{\cC}) \leq W_2^2(\mu,\bar{\nu}) \leq {\EE}_{\mu}\|x-Tx\|^2_2 =  W_2^2(\mathsf{N}(0,\Sigma_\mu),\mathsf{N}(0,\Sigma)) =W_2^2(\Sigma_\mu, \Sigma).
$$
Taking the infimum over $\Sigma\in\cC$ yields $W^2_2(\mu, \mathcal{P}_{\cC})\leq W_2^2(\Sigma_\mu, \cC)$, as desired.

% moreover, Brenier's theorem \cite[Theorem~1.6.2]{Panaretos2020Wasserstein} yields $T = \nabla \varphi$ for some convex function $\varphi$, and since $T$ is linear, we conclude that $\varphi$ is quadratic. In turn, this implies that the Kantorovich potentials are quadratic functions: invoking strong duality, we have
%$W_{2}^{2}(\Sigma_{\mu}, \Sigma_{\nu}) =  \int{\phi\,d\mathsf{N}(0,\Sigma_\mu)} + \int{\psi\,d\mathsf{N}(0,\Sigma_\nu)}$ where $\phi$ and $\psi$ are quadratic functions satisfying $\phi(x) + \psi(y) \leq \|x - y \|^{2}$ (see \cite[Theorem~1.4.2]{Panaretos2020Wasserstein}). Therefore, weak duality yields
%		\begin{align*}
%						W_2^2(\mu, \nu) \geq \int{\phi\,d\mu} + \int{\psi\,d\nu} &= \int{\phi\,d\mathsf{N}(0,\Sigma_\mu)} + \int{\psi\,d\mathsf{N}(0,\Sigma_\nu)} = W_2^2(\Sigma_\mu, \Sigma_\nu),
%			\end{align*}
%            where the first equality holds because the integral of a quadratic function depends on the measure only through its first two moments. Minimizing over $\nu \in \cP_{\cC}$ yields the result when $\Sigma_{\mu}$ is invertible.

Finally, if $\Sigma_{\mu}$ is not invertible, then $\mu$ is supported on the proper subspace $\cH:=\range(\Sigma_\mu)$ of $\R^d$. We then define the perturbed distribution $\mu_t := \mu\vert_{\cH} \times \gamma_t$ on $\R^d$, where $\gamma_t$ is the zero-mean Gaussian distribution supported on $\mathcal{H}^{\perp} = \ker(\Sigma_\mu)$ with covariance $t\cdot I$. Now, Lemma~\ref{lem:dist-trace} shows $W_{2}^2(\mu_{t}, \mu) = W_{2}^2(\mu_{t}, (\sP_{\mathcal{H}})_{\#}\mu_{t}) = \tr\Sigma_{\gamma_t} \rightarrow 0$ as $t\rightarrow 0$, and clearly $\Sigma_{\mu_t}$ is invertible with $\Sigma_{\mu_t} \rightarrow \Sigma_{\mu}$ as $t\rightarrow 0$. From \eqref{eqn:basic_science}, we have
$W_2(\mu_t, \mathcal{P}_{\cC})=W_2(\Sigma_{\mu_t}, \cC)$. Letting $t\rightarrow 0$, we deduce 
$W_2(\mu, \mathcal{P}_{\cC})=W_2(\Sigma_{\mu}, \cC)$ by continuity.
\end{proof}

The proof of Theorem~\ref{thm:distance_est_spectral_moreau} now proceeds in exactly the same way as that of Theorem~\ref{thm:gauss_space}, with Lemma~\ref{lem:basic_scie} being used in conjunction with Lemma~\ref{lem:gauss}.

\begin{proof}[Proof of Theorem~\ref{thm:distance_est_spectral_moreau}]
%The argument is essentially the same as in Theorem~\ref{thm:gauss_space}. We detail it here for completeness.
For any set $\cC \subset \bS^{d}_{+}$, define $\cP_{\cC} := \{\nu \in \cP_{2}^{\circ} \mid \Sigma_{\nu} \in \cC\}$. Then for any $\rho>0$ and $\mu\in\cP^{\circ}_2$, we successively compute
	\begin{align}
		F_{\rho}(\mu)&=\inf_{\nu\in \cP_2^{\circ}}~ F(\nu)+\tfrac{1}{2\rho} W^2_2(\mu,\nu)\\
		&=\inf_{u\geq 0}\,\inf_{\nu \in \cP_{\mathcal{O}(\Diag(u))}}~ f(u)+\tfrac{1}{2\rho} W^2_2(\mu,\nu)\\
		&=\inf_{u\geq 0}~ f(u)+\tfrac{1}{2\rho} + W^2_2\big(\mu, \cP_{\mathcal{O}(\Diag(u))}\big)\\
		&=\inf_{u\geq 0}~ f(u)+\tfrac{1}{2\rho} W_2^2(\Sigma_{\mu},\mathcal{O}(\Diag(u)))\label{eqn:here_we_are}\\
		&=\inf_{u\geq 0}~ f(u)+\tfrac{1}{2\rho} \Big\|\sqrt{\lambda(\Sigma_{\mu})}-\sqrt{u^{\uparrow}}\Big\|^2_2\label{eqn:orbit_dist2}\\
		&=\inf_{u\geq 0}~ f(u)+\tfrac{1}{2\rho} W_2^2(\lambda(A),u)\label{eqn:reordering_works2}\\
		&=f_{\rho}(\lambda(A)),
	\end{align}
where \eqref{eqn:here_we_are} follows from Lemma~\ref{lem:basic_scie}, the equality \eqref{eqn:orbit_dist2} follows from Lemma~\ref{lem:gauss},  and \eqref{eqn:reordering_works2} follows from \eqref{eqn:nonincreaing_basic} and the symmetry of $f$. This completes the proof.
\end{proof}

\paragraph{Connection to existing literature.}
The results presented in this section have close analogues in the existing literature on matrix analysis and optimization. Namely, a function $F\colon \bS^{d}\to\R\cup\{+\infty\}$ is called {\em orthogonally invariant} (or {\em spectral}) if the equality holds:
$$F(UXU^\top)=F(X)\qquad \forall X\in \bS^d,\, U\in O(d).$$
Evidently, such functions are fully described by their restriction to diagonal matrices. More precisely, a function $F$ is orthogonally invariant if and only if there exists a symmetric function $f\colon\R^d\to\R\cup\{+\infty\}$ satisfying $F(X)=f(\lambda(X))$. A pervasive theme in the study of such functions is that various variational properties of the permutation-invariant function $f$ are inherited by the induced
spectral function $F=f\circ\lambda$; see, e.g., \cite{lewis2005nonsmooth,miroslav,lewis1996convex,lewis1999nonsmooth,daniilidis2014orthogonal,daniilidis2008prox,sendov2007higher,davis1957all}. For example, $f$ is convex if and only if $F$ is convex \cite{davis1957all,lewis1996convex}, $f$ is $C^p$-smooth if and only of $F$ is $C^p$-smooth \cite{sendov2007higher,miroslav,sylvester}, and so forth. A useful result in this area is the expression for the Moreau envelope obtained in \cite{daniilidis2008prox,drusvyatskiy2018variational}:
\begin{equation}\label{eqn:moreau_summ}
F_{\rho}(X)=f_{\rho}(\lambda(X))\qquad \forall X\in \bS^d.
\end{equation}
For example, as explained in \cite{drusvyatskiy2018variational}, it readily yields expressions relating generalized derivatives of $f$ and $F$. Crucially, in \eqref{eqn:moreau_summ} the Moreau envelope $F_{\rho}$ is computed with respect to the Frobenius norm on $\bS^d$ and the Moreau envelope $f_{\rho}$ is computed with respect to the $\ell_2$ norm on $\R^d$. Thus 
the results presented in the section extend this circle of ideas to functions defined on the Wasserstein-2 space.

\section{Proofs from Section~\ref{sec:second-moment}}
\subsection{Proof of Proposition~\ref{prop:elementary_dist_rad}}\label{sec:proof_dist_rad}
Suppose otherwise that there exists a sequence $\nu_i\in \Q'$ with $W_2(\nu_i,\mu)\leq r$ satisfying ${\rm REG}(\nu_i)\to \infty$. Then from \eqref{eqn:basic_eqncompat} we deduce $W_2(\nu_i,\cE)={\rm RSE}(\nu_i)\to 0$. On the other hand, using the triangle inequality yields ${\rm RSE}(\mu)=W_2(\mu,\cE)\leq W_2(\nu_i,\mu)+W_2(\nu_i,\cE)$. Letting $i$ tend to infinity, we deduce ${\rm RSE}(\mu)\leq r$, which is a contradiction. Thus no such sequence $\nu_i$ exists and \eqref{eqn:blup_desired_M} holds.

Next, suppose that for some $c,q>0$, the inequality ${\rm RSE}(\nu)^q\leq c\cdot {\rm REG}(\nu)^{-1}$ holds for all $\nu \in \Q'\setminus \cE$. Then for any $\nu \in \cQ'$ such that $W_2(\nu,\mu)\leq r$, the triangle inequality yields
$${\rm RSE}(\mu)=W_2(\mu,\cE)\leq W_2(\nu,\mu)+W_2(\nu,\cE)\leq r+{\rm RSE}(\nu)\leq r+(c/{\rm REG}(\nu))^{1/q}.$$
Rearranging yields the estimate ${\rm REG}(\nu)\leq c\cdot({\rm RSE}(\mu)-r)^{-q}$, thereby completing the proof.

\subsection{Proof of Theorem~\ref{thm:infinitesimal}}\label{sec:proof_sec_slopes}
In this section, we prove Theorem~\ref{thm:infinitesimal}, or rather a stronger version thereof. To this end, suppose there exists a  $C^1$-smooth map $F\colon\R^d\to \bS^k_+$ such that the Hessian $\nabla^2_{\cM} f(\betas)$ corresponding to any measure $\mu\in \cP_2$ can be written as $\nabla^2_{\cM} f(\betas)={\EE}_{\mu} F(x)$. The differential of $F$ at $x$ will be denoted by $DF(x)\colon \R^d\to \bS^k$, while the symbol $DF(x)^*\colon \bS^k\to \R^d$ will denote the adjoint linear map of $DF(x)$.  We assume that there exists a constant $L >0$ satisfying $$\|DF(x)\|_{\rm op}\leq L(1+\|x\|_2) \qquad\forall x\in\R^d.$$ \revised{By the mean value theorem, there exists a constant $C>0$ such that $\|F(x)\|_{F}\leq C(1+\|x\|_2^{2})$ for all $x\in\mathbf{R}^d$.} 
%\textcolor{red}{$\mu\in \cQ$, where $\cQ=(\cP_{n(n+1)},W_2)$ for some $n\geq1$. We assume that there exists a constant $L >0$ satisfying $$\|DF(x)\|_{\rm op}\leq L(1+\|x\|_2^n) \qquad\forall x\in\R^d.$$ By the mean value theorem, it follows that there exists a constant $C>0$ such that $\|F(x)\|_{\rm op}\leq C(1+\|x\|_2^{n+1})$ for all $x\in\R^d$.} 
Now define the function $\mathcal{J}\colon\cP_2\to\R$ by setting $$\mathcal{J}(\mu)=\lambda_{\min} ({\EE}_{\mu}F(x)).$$
It follows immediately from \cite[Theorem 6.9]{Villani2009} and the continuity of the function $\lambda_{\min}$ that $\mathcal{J}$ is continuous; in particular, the zero set $\cE=[\cJ = 0]$ is a closed subset of $\cP_2$. 

In order to simplify notation, for any matrix $A\in \bS^k$ we let $E_k(A)$ denote the set of all unit eigenvectors of $A$ corresponding to the minimal eigenvalue $\lambda_{\min}(A)$. For any measure $\mu\in\cP_2$, we define  $E_k(\mu):=E_k({\EE}_{\mu} F(x))$. We will use the elementary facts that $\lambda_{\min}$ is a Lipschitz continuous concave function on $\bS^k$ and its supdifferential at any $A\in\bS^k$ is the set 
$$\partial \lambda_{\min}(A)=\conv\{uu^\top \mid u\in E_k(A)\};$$
see, for example, \cite[Corollary 5.2.3, Corollary 5.2.4 (iii)]{borwein2006convex}.

The proof of Theorem~\ref{thm:infinitesimal} will be divided into two parts, corresponding to the two inequalities in \eqref{eqn:from_rse_to_reg}. We begin by establishing the first inequality ${\rm REG}(\mu)^{q_2-1}\lesssim {\rm RSE}(\mu)$; this is the content of the following theorem. The proof amounts to applying the fundamental theorem of calculus to a power of the function $\mathcal{J}$ along a geodesic $\mu_t$ joining a measure $\mu\in\cP_2\setminus\cE$ to a nearby measure in $\cE$.

\begin{theorem}[Small distance implies small value]\label{thm:small_dist_to_small_val}
Fix a measure $\mu\in \cP_2\setminus\cE$, constants $c,\varepsilon>0$, and a power $q\in [0,1)$. Suppose that \revised{for all measures $\nu\in\cP_2\setminus\cE$ satisfying $W_2(\mu,\nu)<W_2(\mu,\cE)+\varepsilon$,} the estimate holds:
$$\min_{u\in E_k(\nu)}{\EE}_{\nu}\|DF(x)^*[uu^\top]\|^2_2\leq c\cdot \mathcal{J}(\nu)^{2q}.$$
Then
$$\mathcal{J}(\mu)^{1-q} \leq \sqrt{c}\,(1-q) \cdot W_2(\mu, \cE).$$
\end{theorem}

\begin{proof}
\revised{Given any $\delta\in(0,\varepsilon]$, select a measure $\bar\nu\in\cE$ satisfying $W_2(\mu,\bar\nu)<W_2(\mu,\cE)+\delta$,} and let $\pi\in \Pi(\bar\nu,\mu)$ be an optimal transport plan from $\bar\nu$ to $\mu$. Define the functions $\rho_t(x,y)=(1-t)x+t y$ for all $t\in[0,1]$ and $x,y\in\R^d$. Then the curve $\mu_t=(\rho_t)_{\#}\pi$ is a constant speed geodesic in $\cP_2$ from $\bar\nu$ to $\mu$ \cite[Theorem 5.27]{santambrogio2015optimal}:
$$W_2(\mu_t,\mu_s)=|t-s|\cdot W_2(\mu,\bar\nu)\qquad \forall t,s\in [0,1].$$
Now define the curve $\gamma\colon [0,1]\to \bS_+^k$ by setting $\gamma(t)={\EE}_{\mu_t} F(x)$. We would like to compute $\dot{\gamma}(t)$ by interchanging differentiation and integration in the expression
\begin{equation}\label{eqn:exchange2}
\dot{\gamma}(t)=\frac{d}{dt}\,{\EE}_{\pi} F((1-t)x+ty).
\end{equation}
To this end, we bound the derivative of the integrand uniformly in $t$:
$$\|DF((1-t)x+ty)[y-x]\|_{F}\leq L(1+\|x\|_2+\|y\|_2)\|y-x\|_2.$$
Applying H\"{o}lder's inequality, we see that this upper bound is $\pi$-integrable. Therefore, interchanging differentiation and integration in \eqref{eqn:exchange2} yields the expression
$\dot{\gamma}(t)={\EE}_{\pi} DF((1-t)x+ty)[y-x]$.

It is clear that $\gamma$ is absolutely continuous and therefore, using the supdifferential chain rule for concave functions \cite[Lemma 3.3, p. 73]{brezis1973ope}, we deduce that for almost every $t\in (0,1)$ we have
\begin{equation}\label{eqn:running_ass_continued}
\frac{d}{dt} \mathcal{J}(\mu_t)=\frac{d}{dt}(\lambda_{\min}\circ\gamma)(t)=\langle U_t, 
\dot\gamma(t)\rangle\qquad \forall U_t\in \partial \lambda_{\min}(\gamma(t)).
\end{equation}
\revised{Define $t_0 := \max\{t\in[0,1]\mid \cJ(\mu_t)=0\}$ and observe $t_0\in[0,1)$ and, for all $t\in(t_0,1]$, we have $\mu_t\in\cP_2\setminus\cE$ and $W_2(\mu,\mu_t)=(1-t)\cdot W_2(\mu,\bar\nu)<W_2(\mu,\cE)+\varepsilon.$ Thus, by assumption, for each $t\in(t_0,1)$,} we may choose $U_t\in\partial \lambda_{\min}(\gamma(t))$ satisfying
${\EE}_{\mu_t}\|DF(x)^*[U_t]\|^2_2\leq c\cdot\mathcal{J}(\mu_t)^{2q}$. Continuing with \eqref{eqn:running_ass_continued}, we successively compute
\begin{align}
\frac{d}{dt} \mathcal{J}(\mu_t)&={\EE}_{\pi}\langle DF((1-t)x+ty)^*[U_t],y-x\rangle\\
&\leq {\EE}_{\pi}\big[\|DF((1-t)x+ty)^*[U_t]\|_{2}\cdot \|y-x\|_2\big]\\
&\leq  \sqrt{{\EE}_{\pi}\|DF((1-t)x+ty)^*[U_t]\|^2_{2}}\cdot \sqrt{{\EE}_{\pi}\|y-x\|_2^2}\label{eqn:holder}\\
&\leq \sqrt{c}\cdot W_2(\mu,\bar\nu)\cdot  \mathcal{J}(\mu_t)^q
\end{align}
for a.e. $t\in(t_0,1)$, where \eqref{eqn:holder} follows from H\"{o}lder's inequality.
Raising $\mathcal{J}(\mu_t)$ to power $1-q$, we deduce
$$\frac{d}{dt}\mathcal{J}(\mu_t)^{1-q}\leq \sqrt{c}\,(1-q)\cdot W_2(\mu,\bar\nu)\qquad \textrm{for a.e. }t\in (t_0,1).$$
Integrating both sides from $t=t_0$ to $t=1$ and leveraging absolute continuity, we conclude
\begin{equation}\label{eqn:final_punch}
 \mathcal{J}(\mu)^{1-q}-\mathcal{J}(\mu_{t_0})^{1-q}\leq (1-t_0)\cdot\sqrt{c}\,(1-q)\cdot W_2(\mu,\bar\nu).
 \end{equation}
 \revised{Taking into account the equality $\mathcal{J}(\mu_{t_0})=0$ and the estimate $W_2(\mu,\bar\nu) < W_2(\mu,\cE)+\delta$, we obtain $$\mathcal{J}(\mu)^{1-q} < \sqrt{c}\,(1-q)\cdot\big(W_2(\mu,\cE)+\delta\big).$$ This last inequality holds for all $\delta\in(0,\varepsilon]$, demonstrating $\mathcal{J}(\mu)^{1-q} \leq \sqrt{c}\,(1-q) \cdot W_2(\mu, \cE)$.}
 \end{proof}

Next, we pass to the reverse inequality 
${\rm RSE}(\mu)\lesssim{\rm REG}(\mu)^{q_1-1}$, which is a more substantive conclusion. The main tool we will use is the characterization of an ``error bound property'' using the slope. In what follows, for any real number $r$, the symbol $r_+=\max\{0,r\}$ denotes its positive part.

\begin{definition}[Slope]
Consider a function $f\colon\cX\to \R\cup\{+\infty\}$ defined on a metric space $(\cX,\sd)$. The {\em slope} of $f$ at any point $x$ with $f(x)$ finite is defined by 
$$|\nabla f|(x) := \limsup_{x'\to x} \frac{(f(x)-f(x'))_+}{\sd(x,x')}.$$
\end{definition}

Importantly, if the slope is large on a neighborhood, then the function must decrease significantly. This is the content of the following theorem; see \cite[Basic Lemma, Chapter 1]{ioffe2000metric} or \cite[Lemma 2.5]{drusvyatskiy2015curves}.

\begin{theorem}[Decrease principle]\label{thm:decr_princip}
Consider a lower semicontinuous function $f\colon\cX\to \R\cup\{+\infty\}$ on a complete metric space $(\cX,\sd)$. Fix a point $x$ with $f(x)$ finite, and suppose that there are constants $\alpha<f(x)$
 and $r, \kappa>0$ such that the implication holds:
 $$ \alpha<f(u)\leq f(x)\quad \textrm{and}\quad \sd(u,x)\leq r\qquad \Longrightarrow \qquad |\nabla f|(u)\geq \kappa.$$ 
 If the inequality $f(x)-\alpha< \kappa r$ is valid, then  the estimate holds:
 $$\sd(x,[f\leq \alpha])\leq \kappa^{-1}(f(x)-\alpha).$$ 
 \end{theorem}

We will apply this theorem to a power of the function $\cJ$. The key step therefore is to estimate the slope of $\mathcal{J}$ \revised{(with respect to the metric $W_2$ on $\cP_2$)}. This is the content of the following lemma.

\begin{lemma}[Slope computation]\label{lem:slope}
 For any measure $\mu\in \cP_2$, the estimate holds:
$$|\nabla \mathcal{J}|(\mu)\geq \max_{u\in E_k(\mu)}\sqrt{{\EE}_{\mu}\|DF(x)^*[uu^\top]\|^2_2}.$$
\end{lemma}
\begin{proof}
We begin by writing  $\mathcal{J}(\mu)=(\lambda_{\min} \circ G)(\mu)$, where we define the map $G(\mu):={\EE}_{\mu} F(x)$. 
Next, fix a measure $\mu\in \cP_2$ and a matrix $U\in \partial \lambda_{\min}(G(\mu))$, and define the transport map $T(x) := x-DF(x)^*[U]$. %Clearly, we may assume $DF(x)^*[U]$ is not $\mu$-almost surely zero, since otherwise the theorem holds trivially for $U=uu^\top$.
Observe that $I-T$ is square $\mu$-integrable since 
$${\EE}_{\mu} \|x-T(x)\|_2^2={\EE}_{\mu} \|DF(x)^*[U]\|^2_2\leq L^2\cdot\|U\|_F^2\cdot {\EE}_{\mu} (1+\|x\|_2)^2<\infty.$$
Define now the curve $\gamma\colon[0,1]\to \cP_2$ by setting $\gamma(t)=(I+t(T-I))_{\#}\mu$. Note that from \eqref{eq:detcouplingupperbound}, we have
\begin{equation}\label{eqn:curve_dist}
W^2_2(\gamma(t),\gamma(0))\leq t^2\,{\EE}_{\mu}\|x-T(x)\|^2_2.
\end{equation}
Next, from the concavity of $\lambda_{\min}$ we deduce
\begin{equation}\label{eqn:concavity}
\mathcal{J}(\gamma(t))-\mathcal{J}(\gamma(0))\leq \langle U,(G\circ \gamma)(t)-(G\circ \gamma)(0) \rangle.
\end{equation}
We would like to compute $\frac{d}{dt}\langle U,G\circ\gamma(t)\rangle$ by interchanging differentiation and integration in the expression
\begin{equation}\label{exchange_int_doiff}
\frac{d}{dt}\langle U,G\circ\gamma(t)\rangle=\frac{d}{dt} \EE_{\mu} \langle U, F(x+t(T(x)-x))\rangle.
\end{equation}
To this end, we bound the derivative of the integrand uniformly in $t$:
\begin{align*}
|\langle U, DF(x+t(T(x)-x))[T(x)-x]\rangle|
&\leq \|U\|_F\cdot\|DF(x+t(T(x)-x))[T(x)-x]\|_F\\
&\leq L(1+\|x+t(T(x)-x)\|_2)\cdot \|T(x)-x\|_2\\
&\leq L(1+\|x\|_2+t\|T(x)-x\|_2) \cdot \|T(x)-x\|_2\\
&\leq L\Big(\|T(x)-x\|_2+\tfrac{1}{2}\|x\|_2^2+\tfrac{3}{2}\|x-T(x)\|^2_2\Big).
\end{align*}
Clearly, the right-side is $\mu$-integrable and therefore by the dominated convergence theorem, we may interchange differentiation and integration in \eqref{exchange_int_doiff}, yielding:
\begin{equation}\label{eqn:deriv_curve_inner_prod}
\frac{d}{dt}\langle U,G\circ\gamma(t)\rangle\vert_{t=0}=\EE_{\mu} \langle DF(x)^*[U],T(x)-x\rangle= -{\EE}_{\mu}\|T(x)-x\|^2_2.
\end{equation}
In particular, we deduce $(\mathcal{J}\circ \gamma)(t)<(\mathcal{J}\circ \gamma)(0)$ for all small $t>0$. Therefore, dividing \eqref{eqn:concavity} by $W_2(\gamma(t),\gamma(0))$ and taking the limit as $t\downarrow 0$ yields
\begin{align}
|\nabla \mathcal{J}|(\mu)&\geq \lim_{t\downarrow 0}\frac{\mathcal{J}(\gamma(0))-\mathcal{J}(\gamma(t))}{W_2(\gamma(0),\gamma(t))}\\
&\geq \lim_{t\downarrow 0}\frac{\mathcal{J}(\gamma(0))-\mathcal{J}(\gamma(t))}{ t\sqrt{{\EE}_{\mu}\|x-T(x)\|^2_2}}\label{eqn:usenoneg}\\
&\geq \left\langle U,\lim_{t\downarrow 0}\frac{(G\circ \gamma)(0)-(G\circ \gamma)(t)}{t\sqrt{{\EE}_{\mu}\|x-T(x)\|^2_2}} \right\rangle\\
&=\sqrt{{\EE}_{\mu}\|T(x)-x\|^2_2}\label{eqn:deriv_final}\\
&=\sqrt{{\EE}_{\mu}\|DF(x)^*[U]\|^2_2},
\end{align}
where the estimate \eqref{eqn:usenoneg} follows from \eqref{eqn:curve_dist} and the equality \eqref{eqn:deriv_final} follows from \eqref{eqn:deriv_curve_inner_prod}.
\end{proof}

 Finally, combining the decrease principle (Theorem~\ref{thm:decr_princip}) and the estimate on the slope of $\cJ$ (Lemma~\ref{lem:slope}), we arrive at the main result.

\begin{theorem}[Small value implies small distance]\label{thm:small_val_to_dista}
Fix a measure $\mu\in\cP_2\setminus\cE$, a constant $c>0$, a radius $r>0$, and a power $q\in [0,1)$. Suppose that for all measures $\nu\in  [0<\mathcal{J}\leq \mathcal{J}(\mu)]\cap \overline{\mathbb{B}}_2(\mu; r)$, the estimate holds:
$$\max_{u\in E_k(\nu)}{\EE}_{\nu}\|DF(x)^*[uu^\top]\|^2_2\geq c\cdot \lambda_{\min}({\EE}_{\nu}F(x))^{2q}.$$
Then
$$W_2(\mu, [\mathcal{J}= 0])\leq \tfrac{1}{(1-q)\sqrt{c}} \cdot\mathcal{J}(\mu)^{1-q}$$
provided $r$ is large enough so that $\mathcal{J}(\mu)^{1-q}< r(1-q)\sqrt{c}$.
\end{theorem}  
\begin{proof}
Define the function $\mathcal G(\nu):=\mathcal{J}(\nu)^{1-q}$ and note that the slope chain rule and Lemma~\ref{lem:slope} imply $$| \nabla \mathcal{G}|(\nu)=(1-q)\cdot \frac{| \nabla \mathcal{J}|(\nu)}{\mathcal{J}(\nu)^q}\geq (1-q)\sqrt{c}$$
whenever $0<\mathcal{G}(\nu)\leq \mathcal{G}(\mu)$ and $W_2(\nu,\mu)\leq r$. Applying Theorem~\ref{thm:decr_princip} to $\mathcal{G}$ with $\alpha=0$ and $\kappa=(1-q)\sqrt{c}$ completes the proof.
\end{proof}

Theorem~\ref{thm:infinitesimal} follows immediately from Theorems~\ref{thm:small_dist_to_small_val} and \ref{thm:small_val_to_dista}.

\section{Proofs from Section~\ref{section:pca}}
\subsection{Proof of Lemma~\ref{lem:pca_first}}\label{sec:proof_pca_first}
We first argue the inclusion $\supset$. Observe that for any measure $\mu\in \cP_2^{\circ}$, the set of maximizers of \eqref{eqn:basic_science0} is the intersection of the unit sphere $\mathbb{S}^{d-1}$ and the top eigenspace of $\Sigma_{\mu}$. Thus, if $\lambda_1(\Sigma_{\mu})=\lambda_2(\Sigma_{\mu})$, then none of the maximizers are isolated, so they are unstable and therefore $\mu\in\cE$. 

To see the reverse inclusion $\subset$, fix a measure $\mu\in \cP_2^{\circ}$ and suppose that the top two eigenvalues $\lambda_1$ and $\lambda_2$ of $\Sigma_{\mu}$ are distinct. Then, up to sign, the normalized top eigenvector $v_{\star}$ of $\Sigma_{\mu}$ is the unique maximizer of \eqref{eqn:basic_science0}. It remains to verify that $v_{\star}$ is a tilt-stable maximizer of \eqref{eqn:basic_science0}. To this end, define the Lagrangian function $\mathcal{L}(v,\lambda):=-\frac{1}{2}\EE_{\mu} \langle v, x\rangle^2+\frac{\lambda}{2} (\|v\|^2_2 - 1)$. Then the following equalities hold: 
$$\nabla_v \mathcal{L}(v,\lambda)=(\lambda I-\Sigma_{\mu})v\qquad \textrm{and}\qquad \nabla^2_{vv} \mathcal{L}(v,\lambda)=\lambda I-\Sigma_{\mu}.$$
In particular, the first equality shows that the optimal Lagrange multiplier $\lambda_\star$ is $\lambda_1$. Let $V\in \R^{d\times (d-1)}$ be a matrix with an orthonormal $\Sigma_{\mu}$-eigenbasis for $v_{\star}^\perp$ as its columns. Then an elementary computation yields
\begin{equation}\label{eqn:yoyo}
\min_{y\in \mathbb{S}^{d-2}} \langle \nabla^2_{vv} \mathcal{L}(v_{\star},\lambda_{\star})(Vy),Vy\rangle=\min_{y\in \mathbb{S}^{d-2}} \sum_{i=2}^{d} (\lambda_1-\lambda_i) y_{i-1}^2=\lambda_1-\lambda_2>0.
\end{equation}
Therefore $v_{\star}$ is a tilt-stable maximizer of  \eqref{eqn:basic_science0} and hence $\mu \notin \cE $. Moreover, the claimed equality ${\rm REG} (\mu)^{-1}=\lambda_1-\lambda_2$ follows directly from \eqref{eqn:yoyo}, thereby completing the proof.

\subsection{Proof of Theorem~\ref{thm:rse_pca1}}\label{sec:proof_pca1}
This follows directly by applying Theorem~\ref{thm:distance_est_spectral} with $\cG=\cE$ from Lemma~\ref{lem:pca_first} and $G={\{v\in \R^d_+ \mid v^{(1)}=v^{(2)}\}}$, where $v^{(i)}$ denotes the $i$'th largest coordinate value of $v$.

\subsection{Proof of Lemma~\ref{lem:PCA_multi}}\label{section:proof_pca_multi}
	Henceforth, fix a measure $\mu\in \cP_2^{\circ}$  and define the shorthand $\lambda:=\lambda(\Sigma_{\mu})$.
	Let's dispense first with the simple direction $\supset$. To this end, suppose that equality $\lambda_q=\lambda_{q+1}$ holds. Then we may select two orthonormal sets $\cU:=\{u_1,\ldots,u_q\}$ and $\cU':=\{u_1,\ldots, u_{q-1},u_q'\}$ with $\langle u_q, u_q'\rangle =0$ that are contained in the span of the eigenspaces corresponding to the top $q$ eigenvalues. We may further interpolate between these two sets via $\cU_t:=\{u_1,\ldots,tu_q+(1-t)u_q'\}$ for $t\in [0,1]$. The orthogonal projections onto the span of each $\cU_t$ furnish a continuous path of distinct maximizers of \eqref{eqn:pca_brah}, which are therefore not tilt-stable. Thus $\mu \in \mathcal{E}$, as claimed.
  	
	We now establish the reverse inclusion $\subset$. Suppose therefore that $\lambda_q$ and $\lambda_{q+1}$ are distinct. We begin by conveniently parameterizing the Grassmannian manifold ${\rm Gr}(q, d)$ as follows. 
Define the $d \times d$ matrix $A:=\begin{bmatrix} I_{q} & 0\\
	0 & 0
\end{bmatrix}$.
Then, as in \cite[Section 2.1]{bendokat2024grassmann}, we may write ${\rm Gr}(q, d)$ as the orbit of $A$ under conjugation by orthogonal matrices:
$${\rm Gr}(q, d)= \{UAU^\top \mid U\in O(d)\}.$$
Fix a skew-symmetric $d \times d$ matrix $V:=\begin{bmatrix}
	V_1 & V_2\\
	-V_2^\top & V_4
\end{bmatrix}
$, where $V_2\in\R^{q \times (d-q)}$, and define the curve $\gamma\colon \R\to {\rm Gr}(q, d)$ by 
$$\gamma(t):=\exp(-tV)\cdot A\cdot\exp(tV).$$
Differentiating the curve $\gamma$ yields the expression
$$\dot\gamma(0)=AV-VA= \begin{bmatrix}
	0 & V_2\\
	V_2^{\top} & 0
\end{bmatrix}\!.$$
Moreover,  \cite[Section 2.3]{bendokat2024grassmann} shows that varying $V$ among all skew-symmetric $d \times d$ matrices yields the entire tangent space at $A$:  
$$T_{{\rm Gr}(q, d)}(A)=\left\{\begin{bmatrix}
	0 & B\\
	B^{\top} & 0 \end{bmatrix} ~\Bigg\vert~ B\in \R^{q\times (d-q)}\right\}\!.$$
Without loss of generality, we may assume that $\Sigma_{\mu}$ is diagonal with 
$\Sigma_{\mu}={\rm Diag}(\lambda)$. Then $A$ is clearly the unique maximizer of the problem \eqref{eqn:pca_brah}. Consider now the second-order expansion 
$$(f\circ\gamma)(t)=\langle \gamma(t), \Sigma_{\mu}\rangle=f(A)+t\underbrace{\langle AV-VA,\Sigma_{\mu}\rangle}_{=0} + \tfrac{1}{2}t^2\langle AV^2 + V^2A - 2VAV, \Sigma_{\mu}\rangle+O(t^3).$$ 
In particular,  
$$\nabla^2_{{\rm Gr}(q, d)} f(A)[\dot\gamma(0), \dot\gamma(0)] = (f\circ \gamma)''(0) = \langle AV^2 + V^2A -2VAV, \Sigma_{\mu}\rangle.$$
Taking into account the definition of $A$, a quick computation shows
\begin{align*}
	AV^2 + V^2A - 2VAV &=\begin{bmatrix} 2(V_1^2-V_2V_2^\top) & V_1V_2 + V_2V_4\\
		(V_1V_2 + V_2V_4)^\top &0 \end{bmatrix} - 2\begin{bmatrix} V_1^2 & V_1 V_2\\ (V_1 V_2)^\top & -V_2^\top V_2\end{bmatrix}\\
	&=\begin{bmatrix}-2V_2V_2^\top & V_2V_4-V_1V_2\\
		(V_2V_4-V_1V_2)^\top & 2V_2^\top V_2\end{bmatrix}\!.
\end{align*}
Then taking the trace product with the diagonal matrix $\Sigma_{\mu}={\rm Diag}(\lambda)$ yields 
\begin{align}
	(f\circ \gamma)''(0)&=-2\langle V_2V_2^\top,{\rm Diag}(\lambda_{1:q})\rangle + 2\langle V_2^\top V_2,{\rm Diag}(\lambda_{q+1:d})\rangle\\
	&\leq -2\langle V_2V_2^\top,\lambda_q I_{q}\rangle + 2\langle V_2^\top V_2,\lambda_{q+1} I_{d-q}\rangle\label{eqn:bluh}\\
	&= -(\lambda_q-\lambda_{q+1})\cdot 2\|V_2\|^2_F\\
	&= -(\lambda_q-\lambda_{q+1})\|\dot\gamma(0)\|^2_F.
\end{align}
Consequently, the covariant Hessian $\nabla^2_{{\rm Gr}(q, d)} f(A)$ is negative definite on the tangent space $T_{{\rm Gr}(q, d)}(A)$. Therefore $A$ is a tilt-stable maximizer of the problem \eqref{eqn:pca_brah} and hence $\mu\notin \cE$. Moreover, if we take $V_2$ to be any scalar multiple of the $q \times (d-q)$ matrix with $(q,1)$'st entry $1$ and all other entries $0$, then \eqref{eqn:bluh} holds with equality. Hence ${\rm REG} (\mu)^{-1}=\lambda_q - \lambda_{q+1}$, as claimed.

\subsection{Proof of Theorem~\ref{thm:rse_pca2}}\label{sec:proof_pca2}
This follows directly by applying Theorem~\ref{thm:distance_est_spectral} with $\cG=\cE$ from  Lemma~\ref{lem:PCA_multi} and $G={\{v\in \R^d_+ \mid v^{(q)}=v^{(q+1)}\}}$, where $v^{(i)}$ denotes the $i$'th largest coordinate value of $v$.

\section{Proofs from Section~\ref{sec:glm}}
\subsection{Proof of Lemma~\ref{lem:qmle}}\label{sec:proof_qmle}
	First, observe  
  		\begin{align*}
  		\nabla f(\beta) =  {\EE}\big[ ( h'(\dotp{\x, \beta}) - y)\x\big] &=  {\EE} \big[{\EE}\big[ h'(\dotp{\x, \beta}) - y \!\mid\! \x\big]\x\big]\\
  		&= {\EE} \big[ (h'(\dotp{\x, \beta}) - h'(\dotp{\x, \betas}))x\big].
  		\end{align*}
  	Therefore, the equality $\nabla f(\betas)=0$ holds and hence $\betas$ is a critical point for the problem with optimal Lagrange multipliers $\lambda_\star=0$. Differentiating again yields the expression for the Hessian 
	\begin{equation}\label{eqn:stup_H}
  	H:=\nabla^2 f(\betas)= {\EE}\big[h''(\langle \x, \beta_{\star} \rangle) \x \x^{\top}\big].
	\end{equation}
	 Note that $H$ is positive semidefinite since $h''> 0$. Consequently, the set of ill-conditioned distributions $\cE$ corresponds to those distributions $\mu\in\cQ$ of $\x$ for which  $\ker(H)$ nontrivially intersects $\mathcal{T}$. Clearly, a vector $v$ lies in $\ker(H)$ if and only if $\langle Hv,v\rangle=0$, or equivalently ${\EE}\big[h''(\dotp{\x, \betas})\langle \x, v \rangle^2\big]=0$. Taking into account the assumption $h''>0$, this occurs precisely when $v$ lies in the nullspace of $\Sigma_{\mu}$. Thus  $\cE$ is comprised of all measures $\mu\in \Q$ satisfying
	$\mathcal{T}\cap \ker(\Sigma_{\mu})\neq \{0\}.$ Finally, it follows directly from \eqref{eqn:stup_H} that if for some $a,b>0$ the inequalities $a \leq h''(\langle x,\beta_{\star}\rangle) \leq b$ hold for $\mu$-almost every $x$, then $\lambda_{\min}(H\vert_{\cT}) \in \lambda_{\min}(\Sigma_{\mu}\vert_{\cT})\cdot [a,b]$, thereby completing the proof.

\subsection{Proof of Theorem~\ref{thm:RSE_QMLE}}\label{sec:proof_rse_qmle}
  This follows directly from Lemma~\ref{lem:qmle} and Theorem~\ref{prop:Sdist}.

\section{Proofs from Section~\ref{sections:single-index}}
\subsection{Proof of Lemma~\ref{lem:basic_except_1_phase}}\label{sec:proof_basic_phase}
We begin by verifying the claim for the formulation \eqref{eq:phaseretrival}. 
To this end, a quick computation shows $\nabla f(\beta_{\star})=0$ and 
$
\nabla^{2} f(\beta_{\star}) = \EE_{\mu} \dotp{x, \betas}^{2}xx^{\top} = \hat{\Sigma}_{\mu},
$
where $\mu = \sD_x$. Therefore we deduce 
$$\lambda_{\min}(\nabla^2 f(\beta_{\star}))=\min_{v\in \mathbb{S}^{d-1}} \EE_{\mu}  \langle x,\beta_{\star}\rangle^2\langle x,v\rangle^2.$$
In particular, $\nabla^2 f(\beta_{\star})$ is singular if and only if the support of $\mu$ is contained in $\betas^\perp \cup v^\perp$ for some $v\in\mathbb{S}^{d-1}$. If $\mu\in\cP_4 \setminus \cE$, then clearly \eqref{eqn:Lweird_asf2} holds with $(\lb,\ub)=(1,1)$.

Next, we verify the claim for the formulation \eqref{eq:phaseretrival_rank1}. To this end, let $\cM$ denote the manifold of symmetric PSD rank-one matrices:
$$\cM=\{M\in \bS^d_+ \mid \rank(M)=1\}.$$
A quick computation yields
$$\nabla f(M)= {\EE_\mu} \big\langle M-\Ms,xx^\top\big\rangle xx^\top \qquad \textrm{and}\qquad \nabla^2 f(\Ms)[\Delta,\Delta]= {\EE_\mu} \big\langle \Delta ,xx^\top\big\rangle^2.$$
In particular, equality $\nabla f(\Ms)=0$ holds and therefore the optimal Lagrange multipliers $\lambda_{\star}$ are zero. Hence the Hessian of the Lagrangian at $(\Ms,\lambda_{\star})$ coincides with $\nabla^2 f(\Ms)$. Classically, if we form the factorization $\Ms=\betas\betas^\top$, then the tangent space $\cT$ to $\cM$ at $\Ms$ can be written as
$$\mathcal{T}=\{\betas v^{\top} +v \betas^{\top} \mid v\in \R^d\}.$$
Consequently, for any $\Delta=\betas v^{\top} +v \betas^{\top}$ we compute 
$$\nabla^2 f(\Ms)[\Delta,\Delta]= {\EE_\mu} \big\langle \Delta ,xx^\top\big\rangle^2=4\EE_\mu \langle \betas,x\rangle^2\langle v,x\rangle^2 = 4\langle \hat{\Sigma}_{\mu}v, v \rangle.$$
Note $\|\Delta\|_F^2= 2\langle\betas, v\rangle^2+2\|\betas\|^2_2 \|v\|^2_2$ and therefore $\|\Delta\|_F^2/\|\betas\|^2_2 \|v\|^2_2 \in [2,4]$ whenever $v$ is nonzero. Hence $\Delta$ is nonzero whenever $v$ is nonzero. We therefore again deduce $\nabla^2 f(\Ms)[\Delta,\Delta]=0$ for some nonzero $\Delta\in \cT$ if and only if the support of $\mu$ is contained in $\betas^\perp \cup v^\perp$ for some $v\in\mathbb{S}^{d-1}$. Finally, if $\mu\in\cP_4 \setminus \cE$, then \eqref{eqn:Lweird_asf2} follows immediately with $(\lb,\ub)=\big(\tfrac{1}{2}\|\betas\|^2_2 , \|\betas\|^2_2\big)$.

\subsection{Proof of Theorem~\ref{pse_generalphase}}\label{sec:proof_pse_generalphase}
Define the set $\cK_v := \beta_{\star}^\perp \cup v^\perp$ for each unit vector $v\in\mathbb{S}^{d-1}$ and let $h\colon\mathbb{S}^{d-1}\rightarrow[0,\infty)$ be given by
\begin{equation*}
	h(v) = \mathop{\EE}_{x\sim \mu}{\dist}(x, \cK_v)^2=\mathop{\EE}_{x\sim\mu} \bigg[ \Big\langle x, \mfrac{\beta_{\star}}{\|\beta_{\star}\|}\Big\rangle^2 \wedge \big\langle x, v\big\rangle^2 \bigg].
\end{equation*}
Fatou's lemma directly implies that $h$ is lower semicontinuous and therefore admits a minimizer $u_\star\in\argmin_{u\in\mathbb{S}^{d-1}} h(u)$. Let  $s\colon\R^d\rightarrow \cK_{u_\star}$ be a Borel measurable selection of the metric projection $\sP_{\cK_{u_\star}}$. Define the pushforward measure $\bar \nu := s_{\#}\mu$. Clearly $\bar \nu$ lies in $\cE$ and hence
\begin{equation}\label{eqn:wambam}
    \inf_{\nu \in \cE} W_2^2(\mu, \nu) \leq W_2^2(\mu, \bar{\nu}) 
    	= \mathop{\EE}_{x\sim \mu}{\dist}(x,\cK_{u_\star})^2 = h(u_\star) 
	  \end{equation}
 with the first equality holding by \eqref{eq:detcouplingequality}. On the other hand, for any $\nu\in\cE$ there exists $w\in\mathbb{S}^{d-1}$ such that $\Supp(\nu)\subset\cK_w$ and hence
\begin{align*}
      h(u_\star) \leq h(w) = \mathop{\EE}_{x\sim \mu}{\dist}(x, \cK_w)^2 \leq W_2^2(\mu, \nu) 
  \end{align*}
with the last inequality holding by \eqref{eqn:reverse_ineq}. Taking the infimum over $\nu\in \cE$, we deduce that \eqref{eqn:wambam} holds with equality, thereby completing the proof.

\subsection{Proof of Theorem~\ref{thm:compute_for_gauss}}\label{sec:proof_computeforgauss}
 We will need the following two elementary lemmas.
\begin{lemma}\label{lem:correlation}
    If $(y_1,y_2)$ is a centered Gaussian vector with $\EE y_1^2=\sigma_1^2$, $\EE y_2^2=\sigma^2_2$, and $\EE y_1 y_2=\rho\sigma_1\sigma_2$ for some $\sigma_1, \sigma_2>0$ and $\rho\in[-1,1]$, then the following equalities hold:
   \begin{align*}
   \EE|y_1 y_2|&=\frac{2}{\pi}\bigg(\sqrt{1-\rho^2}+\rho\arcsin(\rho)\bigg)\sigma_1\sigma_2,\\
    \EE (y_1 y_2)^2&=(1+2\rho^2)\sigma_1^2\sigma^2_2.
    \end{align*}
\end{lemma}
\begin{proof}
The first equation is proved for example in \cite[Corollary 3.1]{li2009gaussian}. To see the second equation, standard results show $y_1\!\mid\! y_2 \sim\mathsf{N}\big(\rho(\sigma_1/\sigma_2) y_2, (1-\rho^2)\sigma_1^2\big)$. Hence $\EE[y_1^2\!\mid\! y_2]=(1-\rho^2)\sigma_1^2+\frac{\rho^2\sigma_1^2}{\sigma_2^2}y_2^2$, and since $\EE y_2^4 = 3\sigma_2^4$, the tower rule now implies
$\EE (y_1 y_2)^2={\EE}\big[{\EE}[y_1^2 \!\mid\! y_2] y_2^2\big]=(1+2\rho^2)\sigma_1^2\sigma^2_2$.
\end{proof}

  \begin{lemma} \label{lem:simple-ub}
    Consider the function $\psi\colon [-1, 1] \rightarrow \RR$ given by $\psi(t) = \sqrt{1 - t^{2}} + t \arcsin(t)$ and the function $\phi \colon [-1, 1] \rightarrow \RR$ given by $\phi(t) = \big(\frac{\pi}{2} - 1\big)t^{2} + 1$. Then $\psi(t) \leq \phi(t)$ for all $t \in [-1, 1]$.
  \end{lemma}
  \begin{proof}
    Taking the derivative of $\psi$  at $t\in[-1,1]$ and using the fundamental theorem of calculus gives
    \begin{align}\label{eq:complex-series}
      \psi(t) = 1 + \int_{0}^{t} \arcsin(s) ds =  1+ \sum_{n = 0}^{\infty} {2n \choose n}\frac{t^{{2n + 2}}}{4^{n} (2n+1)(2n+2)},
    \end{align}
    where the last equality follows by integrating the Taylor expansion of $\arcsin(s)$. Therefore
    \begin{align*}
      \psi(t) &= 1+ t^{2 }\sum_{n = 0}^{\infty} {2n \choose n}\frac{t^{{2n}}}{4^{n} (2n+1)(2n+2)} \\ &\leq 1+ t^{2 }\sum_{n = 0}^{\infty} {2n \choose n}\frac{1}{4^{n} (2n+1)(2n+2)} = 1 + t^{2}\bigg(\frac{\pi}{2} - 1\bigg) = \phi(t),
    \end{align*}
    where the inequality follows since $|t| \leq 1$ and the penultimate equality follows from evaluating \eqref{eq:complex-series} at $t = 1$ and noting $\psi(1) = \pi/2$.  
    \end{proof}

 With these two lemmas in hand, we start the proof of Theorem~\ref{thm:compute_for_gauss}. Let us first verify \eqref{eqn:product}. To this end, let $u, v\in \mathbb{S}^{d-1}$ and notice that we may write $g_{\Sigma}(u,v)=\EE (y_1y_2)^2$, where we define the centered random variables $y_1:=\langle x,u\rangle$ and $y_2:=\langle x,v\rangle$ with $x \sim \mathsf{N}(0,\Sigma)$. We compute $\EE y_1^2=\langle \Sigma u,u\rangle$, $\EE y_2^2=\langle \Sigma v,v\rangle$, and $\EE y_1 y_2=\langle \Sigma u,v\rangle$. Therefore, Lemma~\ref{lem:correlation} directly implies 
  $$g_{\Sigma}(u,v)=\langle \Sigma u,u\rangle\langle \Sigma v,v\rangle+2\langle \Sigma u,v\rangle^2=\|\Sigma^{1/2}u\|_2^2\cdot\|\Sigma^{1/2}v\|_2^2+2\langle \Sigma^{1/2}u,\Sigma^{1/2}v \rangle^2.$$
 Consequently, applying the Cauchy-Schwarz inequality yields the two-sided bound
 $$\|\Sigma^{1/2}u\|_2^2\cdot\|\Sigma^{1/2}v\|_2^2\leq g_{\Sigma}(u,v)\leq 3\|\Sigma^{1/2}u\|_2^2\cdot\|\Sigma^{1/2}v\|_2^2.$$
 Taking the infimum over $v\in\mathbb{S}^{d-1}$ completes the proof of \eqref{eqn:product}.
  
  Next, we verify~\eqref{eqn:min}.
 Notice that the upper bound in \eqref{eqn:min} follows trivially since for any $u\in \mathbb{S}^{d-1}$,
  $$
  \min_{v \in \mathbb{S}^{d-1}} h_{\Sigma} (u, v)=  \min_{v \in \mathbb{S}^{d-1}} {\EE} \big[ \langle x, u\rangle^{2} \wedge \langle x, v\rangle^{2} \big] \leq \min_{v \in \mathbb{S}^{d-1}}  \EE \dotp{x, v}^{2} = \lambda_{\min} (\Sigma).
  $$
  To prove the lower bound in \eqref{eqn:min}, we will show  that for any $u, v\in \mathbb{S}^{d-1}$, we have
  $$ \bigg( 1 - \frac{2}{\pi}\bigg)\langle\Sigma v, v \rangle \leq h_{\Sigma} (u, v).$$
  To this end, first recall
  $\min\{a,b\}=\frac{a+b}{2}-\frac{|a-b|}{2}$
  for any $a,b\in\R$.
  Therefore
  \begin{align*}
    h_{\Sigma}(u,v)&=\mathop{\EE}_{x\sim \mathsf{N}(0,\Sigma)} \bigg[\frac{\langle x,u\rangle^2+\langle x,v\rangle^2}{2}-\frac{|\langle x,u\rangle^2-\langle x,v\rangle^2|}{2}\bigg]\\
                   &=\frac{\langle\Sigma u,u\rangle + \langle\Sigma v,v\rangle}{2}-\frac{\mathop{\EE}_{x\sim \mathsf{N}(0,\Sigma)}|\langle x,u\rangle^2-\langle x,v\rangle^2|}{2}.
  \end{align*}
  Now,
  $$|\langle x,u\rangle^2-\langle x,v\rangle^2|=|\underbrace{\langle x,u+v\rangle}_{=:y_1} \underbrace{\langle x,u-v\rangle}_{=:y_2}|.$$
  Then $\sigma_{1}^{2} := \EE y_{1}^{2} = \|\Sigma^{1/2} (u + v)\|^{2}_2 $ and $\sigma_{2}^{2} := \EE y_{2}^{2} = \|\Sigma^{1/2} (u - v)\|^{2}_2 $, and
  \begin{align*}
    \EE y_1 y_2&=\langle\Sigma(u+v), u-v\rangle =\rho \sigma_{1} \sigma_{2},
  \end{align*}
  where
  $\rho:=\frac{\langle \Sigma^{1/2} (u+v),\Sigma^{1/2} (u-v)\rangle}{\|\Sigma^{1/2}(u+v)\|_2\|\Sigma^{1/2}(u-v)\|_2}$.
  Thus, Lemma~\ref{lem:correlation} implies
  \begin{align*}
    \EE|y_1 y_2|&=\frac{2}{\pi}\bigg(\underbrace{\sqrt{1-\rho^2}+\rho\arcsin(\rho)}_{=:\psi(\rho)}\bigg)\sigma_1\sigma_2.
  \end{align*}
  Therefore, after the relabeling $\hat u := \Sigma^{1/2}u$ and $\hat v := \Sigma^{1/2}v$ we deduce
   $$h_{\Sigma}(u,v)= \frac{\|\hat u\|^2_2+\|\hat v\|^2_2}{2}-\frac{\psi(\rho)}{\pi}\|\hat u+\hat v\|_2\|\hat u-\hat v\|_2.$$
  By Lemma~\ref{lem:simple-ub}, we have
  \begin{equation}\label{eq:simpler-h} h_{\Sigma}(u,v) \geq \frac{\|\hat u\|^2_2+\|\hat v\|^2_2}{2}- \bigg(\bigg( \frac{1}{2} - \frac{1}{\pi} \bigg) \rho^{2} + \frac{1}{\pi}\bigg)\|\hat u+\hat v\|_2\|\hat u-\hat v\|_2.
  \end{equation}
  To complete the proof, we upper-bound the second term on the right-hand side of this inequality. Without loss of generality, we assume $\|\hat v\|_2 \leq \|\hat u\|_2.$ By definition, $\rho$ is equal to $\cos(\alpha)$ where $\alpha$ is the angle between $\hat u + \hat v$ and $\hat u - \hat v.$ Thus, by the law of cosines we have
  \begin{equation} \label{eq:onenice} 2 \rho \|\hat u+\hat v\|_2\|\hat u-\hat v\|_2 = 2 \cos(\alpha) \|\hat u+\hat v\|_2\|\hat u-\hat v\|_2 = \|\hat u+\hat v\|^{2}_2 + \|\hat u-\hat v\|^{2}_2 - 4\|\hat v\|^{2}_2 = 2(\|\hat u\|^{2}_2 - \|\hat v\|^{2}_2),\end{equation}
  where the last equality follows by the parallelogram law. Similarly, using Young's inequality:
  \begin{equation} \label{eq:twonice}
    \|\hat u+\hat v\|_2\|\hat u-\hat v\|_2 \leq  \frac{\|\hat u+\hat v\|^{2}_2 + \|\hat u-\hat v\|^{2}_2}{2} = \|\hat u\|_2^{2}+ \|\hat v\|^{2}_2.
  \end{equation}
  Then applying \eqref{eq:onenice} and \eqref{eq:twonice} yields
  \begin{align*}
    \bigg(\bigg( \frac{1}{2} - \frac{1}{\pi} \bigg) \rho^{2} + \frac{1}{\pi}\bigg)\|\hat u+\hat v\|_2\|\hat u-\hat v\|_2 & \leq \bigg( \frac{1}{2} - \frac{1}{\pi} \bigg) \rho\Big( \|\hat u\|^{2}_2 - \|\hat v\|^{2}_2 \Big) + \frac{ 1}{\pi} \Big(\|\hat u\|^{2}_2 + \|\hat v\|^{2}_2\Big) \\
 	&= \bigg(\frac{1}{\pi} + \frac{\rho}{2} - \frac{\rho}{\pi} \bigg)\|\hat u\|^{2}_2  + \bigg(\frac{ 1}{\pi} - \frac{\rho}{2} + \frac{\rho}{\pi} \bigg) \|\hat v\|^{2}_2 \\
	& \leq \frac{1}{2}\|\hat u\|^{2}_2  + \bigg(\frac{2}{\pi} - \frac{1}{2}\bigg) \|\hat v\|^{2}_2,
  \end{align*}
  where the last inequality follows by adding and subtracting $(2/\pi - 1/2)$ to the coefficient of $\|\hat v\|_2$ and noting $\|\hat v\|_2 \leq \|\hat u\|_2$. Combining this inequality with \eqref{eq:simpler-h} yields
  $$
  h_{\Sigma}(u,v) \geq \bigg( 1 - \frac{2}{\pi}\bigg) \|\hat v\|^{2}_2 =  \bigg( 1 - \frac{2}{\pi}\bigg) \langle\Sigma v, v \rangle \geq \bigg( 1 - \frac{2}{\pi}\bigg)\lambda_{\min}(\Sigma),
  $$
  which proves the lower bound in \eqref{eqn:min}.
  \section{Proofs from Section~\ref{sections:blind-deconvolution}}
  \subsection{Proof of Lemma~\ref{lem:basic_ill_cond}}\label{sec:proof_basicillcond}

	Set $\Sigma_1=\EE_{\mu}x_1x_1^{\top}$ and $\Sigma_2=\EE_{\nu}x_2x_2^{\top}$.
	Let $\cM$ denote the manifold of rank-one $d_1\times d_2$ matrices:
	$$\cM = \big\{M \in\R^{d_1 \times d_2} \mid \rank(A) = 1 \big\}.$$
	A quick computation yields the expression
	$\nabla f(M)= \EE\, \langle M-\Ms,x_1x^\top_2\rangle x_1x_2^\top$.
	In particular, equality $\nabla f(\Ms)=0$ holds and therefore the optimal Lagrange multipliers $\lambda_{\star}$ are zero. Hence the Hessian of the Lagrangian at $(\Ms,\lambda_{\star})$ coincides coincides with $\nabla^2 f(\Ms)$. We now successively compute
	\begin{align}
		\nabla^2 f(\Ms)[\Delta,\Delta] & =\EE\,\langle \Delta, x_1x^\top_2\rangle^2\label{eqn:bloop}            \\
		                               & ={\EE}\big[x^\top_1\Delta (x_2x^\top_2)\Delta^\top x_1\big]                     \\
		                               & ={\EE}_{\mu} \big[x^\top_1\Delta \Sigma_2\Delta^\top x_1\big]                    \\
		                               & ={\EE}_{\mu} \tr(\Delta \Sigma_2\Delta^\top x_1x^\top_1)               \\
		                               & =\tr(\Delta \Sigma_2\Delta^\top \Sigma_1)                            \\
		                               & =\|\Sigma^{1/2}_1\Delta\Sigma^{1/2}_2\|_F^2 \label{eqn:letsgo}       \\
		                               & \geq \lambda_{\min}(\Sigma_1)\lambda_{\min}(\Sigma_2)\|\Delta\|^2_F.
	\end{align}
	In particular, if both $\Sigma_1$ and $\Sigma_2$ are nonsingular, then $\mu\times\nu$ does not lie in $\cE$.
	We now leverage the tangent structure to get a better lower bound on the quadratic form	$\nabla^2 f(\Ms)[\Delta,\Delta]$.
	To this end, standard computations (see, e.g., \cite[Section 7.5]{boumal2023introduction}) show that the tangent space $\cT$ of $\cM$ at $\Ms=\beta_{1\star}\beta_{2\star}^\top$ is given by
	\begin{align}
		\cT &= \big\{ a\beta_{1\star}\beta_{2\star}^{\top} + u\beta_{2\star}^{\top} + \beta_{1\star}v^{\top} \mid a\in\R,\, u\in \beta_{1\star}^{\perp},\, v\in \beta_{2\star}^{\perp} \big\} \\
		&= \big\{ w\beta_{2\star}^{\top} + \beta_{1\star}v^{\top} \mid w\in\R^{d_1},\, v\in \beta_{2\star}^{\perp} \big\}.  
	\end{align}
	Henceforth, let $\Delta=w\beta_{2\star}^{\top} + \beta_{1\star}v^{\top}$ be an arbitrary element of $\cT$, where $w\in\mathbf{R}^{d_1}$ and $v\in \beta_{2\star}^{\perp}$. Observe that we may rewrite \eqref{eqn:letsgo} as
	\begin{align}
		\nabla^2 f(\Ms)[\Delta,\Delta] & = \left\|\Sigma^{1/2}_1\big(w\beta_{2\star}^{\top} + \beta_{1\star}v^{\top}\big)\Sigma^{1/2}_2\right\|_F^2                                                                                                    \\
		                               & = \left\| \big(\Sigma_{2}^{1/2} \otimes \Sigma^{1/2}_1\big){\Vect}\big(w\beta_{2\star}^{\top}\big) + \big(\Sigma^{1/2}_2\otimes \Sigma^{1/2}_1\big) {\Vect}\big(\beta_{1\star}v^{\top}\big) \right\|_2^2.
	\end{align}
	To obtain the desired lower bound in \eqref{eqn:abithard}, we leverage the following general claim.
	\begin{claim}
		Let $A$ be a positive semidefinite linear operator on a Euclidean space $\mathbf{E}$ and let $x, y \in \mathbf{E}$ be any pair of orthogonal vectors. Then the estimate holds:
		\begin{equation}\label{eqn:need_ineqq}
		\|Ax + Ay\|^{2} \geq \frac{2}{\kappa(A)^{2} + 1}\Big(\|Ax\|^{2} + \|Ay\|^{2}\Big).
		\end{equation}
        \revised{This bound is tight: equality holds if we take $x = (u_{1} + u_{n})/\sqrt{2}$ and $y=(u_{n} - u_{1})/\sqrt{2}$ with $u_{1}$ and $u_{n}$ orthonormal top and least eigenvectors of $A,$ respectively.
        %  Moreover, for any given $A,$ this bound is tight.
        }
	\end{claim}
	\begin{proof}
		 The result is trivial if $A$ is singular, so we may assume that $A$ is positive definite. The claim will follow from the inequality
		\begin{equation}\label{eqn:mydream}
			|\dotp{Ax, Ay}| \leq \bigg(1 - \frac{2}{\kappa(A)^{2} + 1 }\bigg)\|Ax\|\|Ay\|.
		\end{equation}
		Before establishing \eqref{eqn:mydream}, let us first show how it implies \eqref{eqn:need_ineqq}. Expanding the square yields
		\begin{align}
			\|Ax + Ay\|^{2} & = \|Ax\|^{2} + 2 \dotp{Ax, Ay} + \|Ay\|^{2}                                                                       \\
			                & \geq \|Ax\|^{2} - 2 |\dotp{Ax, Ay}| + \|Ay\|^{2}                                                                  \\
			                & \geq \|Ax\|^{2} - 2 \bigg(1 - \frac{2}{\kappa(A)^{2} + 1 }\bigg)\|Ax\|\|Ay\| + \|Ay\|^{2}                \label{eqn:sec_ineq_blah}        \\
			                & \geq \|Ax\|^{2} - \bigg(1 - \frac{2}{\kappa(A)^{2} + 1 }\bigg)\Big(\|Ax\|^{2} + \|Ay\|^{2}\Big) + \|Ay\|^{2} \label{eqn:third_ineq_blah} \\
			                & = \frac{2}{\kappa(A)^{2} + 1}\Big(\|Ax\|^{2} + \|Ay\|^{2}\Big),
		\end{align}
		where \eqref{eqn:sec_ineq_blah} follows from \eqref{eqn:mydream} and \eqref{eqn:third_ineq_blah} is an application of Young's inequality. Thus, \eqref{eqn:mydream} implies \eqref{eqn:need_ineqq}, and so it now suffices to establish \eqref{eqn:mydream}, which by the positive definiteness of $A$ is equivalent to
		\begin{equation}\label{eqn:myWielandtdream}
			|\dotp{x, A^2y}|^2 \leq \Bigg(\frac{\lambda_{\max}(A^2) - \lambda_{\min}(A^2)}{\lambda_{\max}(A^2) + \lambda_{\min}(A^2)}\Bigg)^2 \langle x, A^2 x\rangle \langle y, A^2 y \rangle.
		\end{equation}
		Since $A^2$ is positive definite, \eqref{eqn:myWielandtdream} holds as an instance of Wielandt's inequality \cite[Section~7.4.12]{horn2012matrix}. 
		
        \revised{Finally, a straightforward computation reveals that  equality holds in \eqref{eqn:need_ineqq} when $x = (u_{1} + u_{n})/\sqrt{2}$ and $y=(u_{n} - u_{1})/\sqrt{2}$ with $u_{1}$ and $u_{n}$ orthonormal top and least eigenvectors of $A,$ respectively.}
	\end{proof}
	We now instantiate the claim with $A = \Sigma_{2}^{1/2} \otimes \Sigma_{1}^{1/2}$, $x = {\Vect}\big(w\beta_{2\star}^{\top}\big)$, and $y = {\Vect}\big(\beta_{1\star}v^{\top}\big)$. Note that $x$ and $y$ are orthogonal since
	$$
		\dotp{x, y} = \dotp{w\beta_{2\star}^{\top}, \beta_{1\star}v^{\top}} = {\Tr}\big(\beta_{2\star}w^{\top}\beta_{1\star}v^{\top}\big) = {\Tr}\big(w^{\top}\beta_{1\star}v^{\top}\beta_{2\star}\big) = \dotp{w,\beta_{1\star}}\dotp{v,\beta_{2\star}} = 0,
	$$
	where we have leveraged the cyclic invariance of the trace and the equality $\dotp{v, \beta_{2\star}} = 0.$ Further, $\kappa(A)^{2} = \kappa(\Sigma_{1}) \kappa(\Sigma_{2})$ and $\|\Delta\|_{F}^{2} = \|w\|^{2}_2\|\beta_{2\star}\|_2^2 + \|v\|^{2}_2\|\beta_{1\star}\|_2^2$. Thus, altogether we derive
	
	\begin{align}
		\nabla^2 f(\Ms)[\Delta,\Delta] & \geq \frac{2}{\kappa(\Sigma_{1})\kappa(\Sigma_{2}) + 1} \Big(\dotp{\beta_{2\star}, \Sigma_{2}\beta_{2\star}} \dotp{w, \Sigma_{1}w} + \dotp{\beta_{1\star}, \Sigma_{1}\beta_{1\star}} \dotp{v, \Sigma_{2} v} \Big)                            \\
		                               & \geq \frac{2}{\kappa(\Sigma_{1})\kappa(\Sigma_{2}) + 1} \Big(\dotp{\beta_{2\star}, \Sigma_{2}\beta_{2\star}} \lambda_{\min}(\Sigma_{1})\|w\|^{2}_2 + \dotp{\beta_{1\star}, \Sigma_{1}\beta_{1\star}} \lambda_{\min}(\Sigma_{2})\|v\|^{2}_2\Big) \\
		                         & \geq\frac{2\cdot{\min}\{\gamma_2 \lambda_{\min}(\Sigma_{1}), \gamma_1 \lambda_{\min}(\Sigma_{2})\}}{\kappa(\Sigma_{1})\kappa(\Sigma_{2}) + 1}\|\Delta\|_{F}^2,
	\end{align}
	where $\gamma_i:=\left\langle\Sigma_i\frac{\beta_{i\star}}{\|\beta_{i\star}\|_2},\frac{\beta_{i\star}}{\|\beta_{i\star}\|_2}  \right\rangle$ for $i=1,2$. This establishes the lower bound in \eqref{eqn:abithard}.

	Next, we establish the upper bound in \eqref{eqn:abithard}. Observe that from \eqref{eqn:bloop}, we have
	\begin{equation}\label{eqn:blind_deconv}
		\nabla^2 f(\Ms)[\Delta,\Delta]={\EE}\big[\langle x_1, w \rangle \langle x_2, \beta_{2\star} \rangle + \langle x_1, \beta_{1\star} \rangle \langle x_2, v \rangle\big]^2.
	\end{equation}
	If $v=0$ and $w \ne 0$, then \eqref{eqn:blind_deconv} becomes
	\begin{align}
		\nabla^2 f(\Ms)[\Delta,\Delta] = {\EE}\big[\langle x_1, w \rangle \langle x_2, \beta_{2\star} \rangle\big]^2 = \langle\Sigma_{1}w, w \rangle \langle\Sigma_2\beta_{2\star},\beta_{2\star}\rangle = \gamma_{2}\Big\langle\Sigma_{1} \tfrac{w}{\|w\|_2}, \tfrac{w}{\|w\|_2} \Big\rangle\|\Delta\|^2_{F}
	\end{align}
	and hence minimizing over all such $\Delta$ yields
	\begin{equation}
		\lambda_{\min}(\nabla^2 f(\Ms))\leq\gamma_{2}\lambda_{\min}(\Sigma_1).
	\end{equation}
	Similarly, if $w = 0$ and $v \ne 0$, then \eqref{eqn:blind_deconv} becomes
	\begin{align}
		\nabla^2 f(\Ms)[\Delta,\Delta] = {\EE}\big[\langle x_1, \beta_{1\star} \rangle \langle x_2, v \rangle\big]^2 = \langle\Sigma_{1}\beta_{1\star}, \beta_{1\star} \rangle \langle\Sigma_{2}v,v\rangle = \gamma_{1}\Big\langle\Sigma_{2} \tfrac{v}{\|v\|_2}, \tfrac{v}{\|v\|_2} \Big\rangle\|\Delta\|^2_{F}
	\end{align}
	and hence minimizing over all such $\Delta$ yields
	\begin{equation}
		\lambda_{\min}(\nabla^2 f(\Ms))\leq\gamma_{1}\lambda_{\min}(\Sigma_2).
	\end{equation}
Thus, $\lambda_{\min}(\nabla^2 f(\Ms))\leq{\min}\{\gamma_2 \lambda_{\min}(\Sigma_{1}), \gamma_1 \lambda_{\min}(\Sigma_{2})\}$ and the upper bound in \eqref{eqn:abithard} is established. In particular, if either $\Sigma_1$ or $\Sigma_2$ are singular, then $\mu\times\nu$ lies in $\cE$, as claimed. This completes the proof.

	\subsection{Proof of Theorem~\ref{thm:RSE_bilin}}
	
		This follows directly from Lemma~\ref{lem:basic_ill_cond} and Theorem~\ref{prop:Sdist}.

 \section{Proofs from Section~\ref{sec:matrix-completion}}

 \subsection{Proof of Lemma~\ref{lem:blahblah_notneeded}}\label{sec:proof_blah}
  Consider any tangent vector $\Delta=\betas v^{\top} +v \betas^{\top}\in\cT$, where $v\in\R^d$. We may vectorize $\Delta$ as follows:
$$\Vect(\Delta)=\Vect(\betas v^{\top}I)+\Vect(I v \betas^{\top})=(I\otimes\betas)v+(\betas\otimes I)v=\Phi_{\betas}v.$$
 Therefore,  from \eqref{eqn:hessian_rep} we compute 
	\begin{equation}\label{eq:mat-comp-lambda}
	 \nabla^{2} f(M_{\star})[\Delta,\Delta]=\Vect(\Delta)^{\top} \Sigma_{\mu} {\Vect}(\Delta)
		 =   v^{\top} (\Phi_{\beta_{\star}}^{\top} \Sigma_{\mu}\Phi_{\beta_{\star}}) v.
\end{equation}
Assuming $v\ne0$, we recall $\|\Delta\|_F^2/\|\betas\|^2_2 \|v\|^2_2 \in [2,4]$ and hence \eqref{eq:mat-comp-lambda} implies 
	\begin{equation}\label{eq:mat-comp-lambda-2}
	 \frac{\big\langle(\Phi_{\beta_{\star}}^{\top} \Sigma_{\mu}\Phi_{\beta_{\star}}) v, v \big\rangle}{4\|\betas\|_2^2\|v\|_2^2} \leq \frac{\nabla^{2} f(M_{\star})[\Delta,\Delta]}{\|\Delta\|_{F}^2} \leq  \frac{\big\langle(\Phi_{\beta_{\star}}^{\top} \Sigma_{\mu}\Phi_{\beta_{\star}}) v, v \big\rangle}{2\|\betas\|_2^2\|v\|_2^2}.
\end{equation}
Minimizing the ratios in \eqref{eq:mat-comp-lambda-2} over all $v\in\R^d\setminus\{0\}$ yields the estimate \eqref{eqn:reg_modulus}.

\subsection{Proof of Theorem~\ref{thm:rse_MC}}\label{sec:proof_rsemc}

 \textbf{Characterization of well-posedness.}
	To simplify notation, let us relabel $\betas$ to $\beta$, $G_{\supp(P)}$ to $G= (\lbr d \rbr, E)$, $G_{\supp(P)}^{*}$ to $G^{*}= (V^{*}, E^{*})$, and $V_{\supp(P)}^{0}$ to $V^{0}$. Consider any tangent vector $\Delta=\beta v^{\top} +v \beta^{\top}\in\cT$, where $v\in\R^d$. Observe that \eqref{eqn:hessian_rep} yields
		\begin{equation}\label{eqn:expansion_hessian}
	\nabla^2 f(M_{\star})[\Delta,\Delta]=\EE\langle X, \Delta \rangle^2=\sum_{i,j:\, p_{ij}>0} (v_i\beta_j+v_j\beta_i)^2 p_{ij}.
	\end{equation}
%Recall moreover that  $\Delta$ is nonzero if and only if $v$ is nonzero.

	\noindent {$(\!\implies\!)$} We prove the contrapositive. Thus, suppose $G$ does not satisfy Assumption~\ref{eq:mc-combinatorial-char}. We consider two cases separately.\vspace{.2cm}
	
	\noindent\textit{Case 1}. Suppose $G^{*}$ has a bipartite connected component $G_{\ell}^{*} = (V^{*}_{\ell}, E^{*}_{\ell})$, so that there exists a disjoint decomposition $V^{*}_{\ell} = I \sqcup J$ of $V^{*}_{\ell}\subset\supp(\beta)$ such that every edge in $E^{*}_{\ell}$ is of the form $\{i,j\}$ for some $(i,j) \in I\times J$. Then, taking $v\in\R^d$ to be the nonzero vector given by
	$$
		v_{i} = \begin{cases}
			\beta_{i}   & \text{if } i \in I, \\
			- \beta_{i} & \text{if } i \in J,  \\
			0           & \text{otherwise,}
		\end{cases}
	$$
	we see that the tangent vector $\Delta=\beta v^{\top} +v \beta^{\top}$ is nonzero and \eqref{eqn:expansion_hessian} implies
	\begin{equation} \label{eq:in-the-kernel}
		\nabla^2 f(M_{\star})[\Delta,\Delta] = \sum_{i, j:\,\{i,j\} \in E^{*}_{\ell}} (v_{i}\beta_{j} + v_{j}\beta_{i})^{2} p_{ij} =\sum_{i, j:\,\{i,j\}\in E^{*}_{\ell}} (\beta_{i}\beta_{j} - \beta_{j}\beta_{i})^{2} p_{ij} = 0,
	\end{equation}
	where the first equality holds since $v$ is supported on $V^{*}_{\ell}$. Thus, $\mu_P \in \cE^{{\rm mc}}$, as claimed.\vspace{.2cm}
	
	\noindent\textit{Case 2}. Suppose there exists $k \in V^{0}$. Then $\beta_{k} = 0$, and for any $j \in \lbr d \rbr$ such that $p_{kj}>0$, we have $\beta_j=0$.
Taking $v = e_{k}$ in \eqref{eqn:expansion_hessian}, we deduce
$$\nabla^2 f(M_{\star})[\Delta,\Delta] = \sum_{i, j:\,p_{ij}>0} (v_i\beta_j+v_j\beta_i)^2 p_{ij}=0$$
since each summand vanishes. Therefore $\mu_{P}\in\cE^{{\rm mc}}$, as claimed.\vspace{.2cm}

	\noindent {$(\!\impliedby\!)$} Suppose $G$ satisfies Assumption~\ref{eq:mc-combinatorial-char}. To show that $\nabla^2 f(M_{\star})$ is nonsingular, suppose that the tangent vector $\Delta=\beta v^{\top} +v \beta^{\top}$ satisfies $\nabla^2 f(M_{\star})[\Delta,\Delta]=0$; our goal is to show $\Delta = 0$, or equivalently, $v=0$. To this end, note first that by \eqref{eqn:expansion_hessian}, the equality $\nabla^2 f(M_{\star})[\Delta,\Delta]=0$ implies
	\begin{equation}\label{eqn:propogation_kernel}
	v_i\beta_j+v_j\beta_i=0 ~~\text{whenever}~~ \{i,j\}\in E.
	\end{equation}
	Now let $G_{\ell}^{*} = (V^{*}_{\ell}, E^{*}_{\ell})$ be any connected component of $G^{*}$. Then $G^{*}_{\ell}$ is non-bipartite, so it must contain an odd-sized cycle $i_{1} \rightarrow i_{2} \rightarrow \dots \rightarrow i_{k} \rightarrow i_{1}$. If $k=1$, then $\{i_1\}\in E$ and hence \eqref{eqn:propogation_kernel} implies $2 \beta_{i_{1}} v_{i_{1}} =0$. Similarly, if $k\geq 3$, then we have the expansion
	$$
		2 \beta_{i_{1}} v_{i_{k}} = \big( \beta_{i_{1}} v_{i_{k}} + \beta_{i_{k}} v_{i_{1}}\big) - \frac{\beta_{i_{k}}}{\beta_{i_{2}}} \big( \beta_{i_{2}} v_{i_{1}} + \beta_{i_{1}} v_{i_{2}} \big) + \sum_{j = 2}^{k - 1} (-1)^{j} \frac{\beta_{i_{1}} \beta_{i_{k}}}{\beta_{i_{j}} \beta_{i_{j+1}}} \big( \beta_{i_{j+1}} v_{i_{j}} + \beta_{i_{j}} v_{i_{j+1}} \big) = 0,
	$$
	where the last equality follows since each term in parentheses is zero by \eqref{eqn:propogation_kernel}. Since $\beta_{i_{1}} \ne 0$, we deduce $v_{i_{k}} = 0$. Next, observe from \eqref{eqn:propogation_kernel} that for any vertex $j$ satisfying $\{i_{k}, j\}\in E$, we have $ v_{j} = -(\beta_{j}/\beta_{i_{k}})v_{i_{k}} = 0$; by induction along any path in $G^{*}_{\ell}$ terminating at $i_k$, we deduce $v_{j} = 0$ for all $j \in V^{*}_{\ell}$ by the connectivity of $G^{*}_{\ell}$, and since this holds for every connected component $G^{*}_{\ell}$ of $G^{*}$, we conclude $v_{j} = 0$ for all $j \in V^{*}$. Finally, consider any vertex $i \in \lbr d \rbr\setminus V^{*}$. Then since $V^{0}=\emptyset$, there exists some $j\in V^*$ such that $\{i,j\}\in E$. Using \eqref{eqn:propogation_kernel} again, we conclude $v_i\beta_j=0$; taking into account $\beta_j \ne 0$, we deduce $v_i = 0$. Therefore $v=0$ and we have established $\mu_{P} \notin \cE^{{\rm mc}}$.

	\paragraph{Distance formula.} We have now proved that for any $Q\in\cQ$, we have $\mu_{Q} \in \cE^{{\rm mc}}$ if and only if $\supp(Q)\in \Omega_{\betas}$. %\revised{it is supported on $\big\{\tfrac{1}{2}(e_ie_j^\top + e_je_i^\top) \mid (i,j) \in A\big\}$ for some  $A \in \Omega_{\betas}$.}
	Given any subset $A \subset \lbr d \rbr \times \lbr d \rbr$, let us define $\sP_{A}\colon \R^{d \times d} \rightarrow \R^{d \times d}$ to be the orthogonal projection onto the entries indexed by $A$ (i.e., the orthogonal projection onto $\Span\{e_{i}e_{j}^{\top} \mid (i,j)\in A\}$). Now consider any $A \in \Omega_{\betas}$ such that $A\subset \supp(P)$. Clearly, the pushforward measure $(\sP_{A})_{\#}\mu_P$ is equal to $\mu_Q$ where $Q := \sP_{A}P \in\cQ$ satisfies $\supp(Q) = A\in\Omega_{\betas}$. Thus, $(\sP_{A})_{\#}\mu_P\in\cE^{\text{mc}}$ and hence
	\begin{align*}
		{\rm RSE}(\mu_{P})^{2} = \min_{\nu \in \cE^{\text{mc}}} W_{2}^{2} (\mu_P, \nu)& \leq W_{2}^{2}\big(\mu_P, (\sP_{A})_{\#}\mu_P\big)= \sum_{(i,j)\in \supp(\bP) \setminus A} p_{ij},
	\end{align*}
	where the last equality follows from Lemma~\ref{lem:w2_proj}. Taking the minimum over all $A\in\Omega_{\betas}$ such that $A\subset\supp(P)$ yields the inequality $\leq$ in \eqref{eq:mat-comp-hard-dist}.

	Conversely, consider an arbitrary measure $\nu$ in $\cE^{\text{mc}}$. Then $\nu = \mu_Q$ for some $Q\in\cQ$ such that $\supp(Q)\in \Omega_{\betas}$. Now set $B:=\supp(P)$ and $A' := \supp(Q)\cap B$; since $A'$ is a symmetric subset of $\supp(Q)\in \Omega_{\betas}$, it follows readily that $A'\in \Omega_{\betas}$. Moreover, clearly $A'\subset B$ and the support of $(\sP_B)_{\#}\nu$ is contained in $\range(\sP_{A'})$, so \eqref{eqn:reverse_ineq} and Lemma~\ref{lem:w2_proj} imply
	\begin{align*}
		W_{2}^{2}(\mu_P, \nu) \geq W_{2}^{2}(\mu_P, (\sP_B)_{\#}\nu) \geq W_{2}^{2}(\mu_P, (\sP_{A'})_{\#} \mu_P)
		= \sum_{(i,j) \in B \setminus A'} p_{ij} \geq \min_{\substack{A \in \Omega_{\betas}\\{A\subset B}}}\sum_{(i,j)\in B \setminus A} p_{ij},
	\end{align*}
	where the first and last inequalities hold trivially. Since $\nu \in \cE^{\text{mc}}$ is arbitrary, this completes the proof of the reverse inequality $\geq$ in \eqref{eq:mat-comp-hard-dist}. Further, this argument shows that the $W_2$ distance from $\mu_P$ to $\cE^{\text{mc}}$ is attained by the measure $(\sP_{A^*})_{\#}\mu_P$ for any $A^*\subset\supp(P)$ in $\Omega_{\betas}$ minimizing \eqref{eq:mat-comp-hard-dist}.   

	\paragraph{Hardness.} Suppose that we have an algorithm \texttt{Alg} that outputs ${\rm RSE}(\mu_{P})^2$ for any given matrix completion data $(P,M_{\star})$, and let \texttt{MaxCut} be an algorithm that outputs the maximum cut size of any given undirected graph. We will construct a polynomial-time Turing reduction from \texttt{MaxCut} to \texttt{Alg}. To this end, let $G = (V, E)$ be an arbitrary undirected graph with connected components $G_{1} = (V_{1}, E_{1}), \dots, G_{k} = (V_{k}, E_{k})$. For each $\ell = 1,\ldots, k$, let $d_\ell = |V_\ell|$, $V_{\ell} = \{n_{\ell}(1), \ldots, n_{\ell}(d_\ell)\}$, and consider the instance of matrix completion data $(P_\ell,M_{\star\ell})$ where $P_\ell\in\bS^{d_\ell}$ is given by   
		$$
		\big(P_{\ell}\big)_{ij} = \begin{cases}
			\frac{1}{2|E_{\ell}|}   & \text{if } i \ne j \text{ and } \{n_{\ell}(i),n_{\ell}(j)\} \in E_{\ell}, \\
			\frac{1}{|E_{\ell}|}   & \text{if } i = j \text{ and } \{n_{\ell}(i)\} \in E_{\ell}, \\
			0           & \text{otherwise}
		\end{cases}
		$$
		and $M_{\star\ell} = \beta_{\star\ell} \beta_{\star\ell}^\top$ where $\beta_{\star\ell}$  is any vector in $\RR^{d_{\ell}}$ such that $\supp(\beta_{\star\ell}) = \lbr d_{\ell} \rbr$, e.g., $\beta_{\star\ell}^\top = (1,\ldots,1)$. Defining $B_{\ell}:=\supp(P_{\ell})$, we clearly have $G_{\ell} \cong G_{B_{\ell}}$, $G_{B_\ell}^{*} = G_{B_\ell}$, and $V_{B_\ell}^0 = \emptyset$. %further, since $G_\ell$ is connected, any maximum bipartite subgraph of $G_\ell$ is connected.
		It follows from \eqref{eq:mat-comp-hard-dist} and the connectivity of $G_\ell$ that the quantity  
		$$
		|E_{\ell}| \cdot \texttt{Alg}(\bP_{\ell}, M_{\star\ell}) = \min_{\substack{A \in \Omega_{\beta_{\star\ell}}\\{A\subset\supp(P_\ell)}}}\sum_{\substack{i\leq j \\ (i,j)\in \supp(\bP_{\ell}) \setminus A}} \mathbf{1}_{E_{\ell}}\big(\{n_{\ell}(i),n_{\ell}(j)\}\big)
		$$
		is equal to the minimum number of edges one needs to remove from $G_{\ell}$ so that the remaining subgraph is bipartite. Thus, $|E_{\ell}|\big(1 - \texttt{Alg}(\bP_{\ell}, M_{\star\ell})\big)$ is equal to the number of edges of a maximum bipartite subgraph of $G_\ell$, which is equal to $\texttt{MaxCut}(G_\ell).$ Therefore summing over the connected components of $G$ yields
	$$ \texttt{MaxCut}(G) = |E| - \sum_{\ell =1}^{k}|E_{\ell}|\cdot \texttt{Alg}(\bP_{\ell}, M_{\star\ell}),
	$$
	thereby providing a polynomial-time Turing reduction from the problem of computing \texttt{MaxCut}($G$) for any given undirected graph $G$ to the problem of computing ${\rm RSE}(\mu_{P})^2$ for any given matrix completion data $(P,M_{\star})$. Since the former problem is NP-hard in general, so is the latter.

\end{document}